\documentclass[11pt,reqno]{amsart}

\pdfoutput=1

\usepackage{geometry}
\usepackage{amsthm}
\usepackage{amssymb}
\usepackage{bbm}
\usepackage{mathrsfs}
\usepackage{enumitem}
\usepackage[english]{babel}

\usepackage [autostyle, english = british]{csquotes}
\MakeOuterQuote{"}

\makeatletter
\numberwithin{equation}{section}
\numberwithin{figure}{section}

\newcommand{\R}{\mathbb{R}}
\newcommand{\C}{\mathbb{C}}
\newcommand{\N}{\mathbb{N}}
\newcommand{\F}{\mathcal{F}}
\renewcommand{\P}{\mathbb{P}}
\newcommand{\E}{\mathbb{E}}
\newcommand{\e}{\varepsilon}
\newcommand{\1}{\mathbbm{1}}

\newtheorem{Theorem}{Theorem}[section]
\newtheorem{Proposition}[Theorem]{Proposition}\newtheorem{Corollary}[Theorem]{Corollary}\newtheorem{Lemma}[Theorem]{Lemma}\newtheorem{Remark}[Theorem]{Remark}\newtheorem{Definition}[Theorem]{Definition}\newtheorem{Example}[Theorem]{Example}

\numberwithin{equation}{section}

\makeatother

\begin{document}

\title[Limits of stochastic Volterra equations driven by Gaussian noise]{Limits of stochastic Volterra equations driven by Gaussian noise}

\author{Luigi Amedeo Bianchi}
\address[Luigi Amedeo Bianchi]{Department of Mathematics
\\ Trento University
\\ Trento, Italy}
\email{luigiamedeo.bianchi@unitn.it}

\author{Stefano Bonaccorsi}
\address[Stefano Bonaccorsi]{Department of Mathematics
\\ Trento University
\\ Trento, Italy}
\email{stefano.bonaccorsi@unitn.it}

\author{Martin Friesen}
\address[Martin Friesen]{School of Mathematical Sciences\\
Dublin City University\\ Glasnevin, Dublin 9, Ireland}
\email{martin.friesen@dcu.ie}

\date{\today}

\subjclass[2020]{Primary 60G22; 45D05; Secondary 60G10; 60B10}

\keywords{stochastic Volterra equation; mild solution; ergodicity; coupling; fractional heat equation}

\begin{abstract}
 We study stochastic Volterra equations in Hilbert spaces driven by cylindrical Gaussian noise. We derive a mild formulation for the stochastic Volterra equation, prove the equivalence of mild and strong solutions, the existence and uniqueness of mild solutions, and study space-time regularity. 
 Furthermore, we establish the stability of mild solutions in $L^q(\R_+)$, prove the existence of limit distributions in the Wasserstein $p$-distance with $p \in [1,\infty)$, and characterise when these limit distributions are independent of the initial state of the process despite the presence of memory. While our techniques allow for a general class of Volterra kernels, they are particularly suited for completely monotone kernels and fractional Riemann-Liouville kernels in the full range $\alpha \in (0,2)$.
\end{abstract}

\maketitle

\allowdisplaybreaks

\section{Introduction}

\subsection{Framework}
Let $(U, \langle \cdot, \cdot\rangle_U, \|\cdot\|_U)$ and $(H, \langle \cdot, \cdot \rangle_H, \|\cdot\|_H)$ be  separable Hilbert spaces. On a filtered probability space $(\Omega,\mathcal{F},(\mathcal{F}_t)_{t\in\mathbb{R}_+},\mathbb{P})$ with the usual conditions we let $(W_t)_{t \geq 0}$ be a cylindrical $(\F_t)_{t \geq 0}$-Wiener process on $U$. We study the $H$-valued stochastic Volterra equation with nonlinear drift and multiplicative noise given by
\begin{align}\label{VSPDE} 
 u(t) = g(t) + \int_0^t k(t-s)\left( Au(s) + F(u(s))\right) \, ds + \int_0^t h(t-s)\sigma(u(s)) \, dW_s,
\end{align} 
where $k \in L_{loc}^1(\R_+)$, and $h \in L_{loc}^2(\R_+)$ are scalar-valued kernels, $F$ denotes the drift, $\sigma$ the diffusion coefficient, $g: \R_+ \times \Omega \longrightarrow H$ is $\mathcal{B}(\R_+) \otimes \F_0/\mathcal{B}(H)$-measurable, and the linear operator $(A,D(A))$  on $H$ is closed and densely defined. 

Stochastic convolutions and Volterra equations on Hilbert spaces associated with Volterra Ornstein-Uhlenbeck type processes have been thoroughly discussed in \cite{MR1954075, MR1454409, MR1658451, MR2316483, MR3630332, MR3605714}, while the case of nonlinear drift and multiplicative noise was recently studied, e.g., in \cite{Benth} for regular kernels and \cite{FHK22} also for singular kernels. For such Volterra kernels $k,h$ the solution of \eqref{VSPDE} is neither a Markov process nor a semimartingale. Henceforth we cannot directly rely on the study of corresponding Kolmogorov equations or rely on mild or strong It\^o formulas as given in \cite{MR4009384} and \cite[Section 2.3]{MR2560625}. 
This makes the study of stochastic Volterra processes a challenging task.

In many concrete examples, fractional Riemann-Liouville kernels play a central role. The Volterra kernels $k,h$ are in such a case given by 
\begin{align}\label{eq: fractional kernels}
 k(t) = \frac{t^{\alpha - 1}}{\Gamma(\alpha)} \ \text{ and }\ h(t) = \frac{t^{\beta - 1}}{\Gamma(\beta)} \ \text{ with }\ \alpha \in (0,2), \ \ \beta > \frac{1}{2}.
\end{align}
These kernels allow for a flexible modelling framework addressing short and long-range dependencies, see \cite{S14}. For instance, if $\alpha < 1$, then the fractional kernel is completely monotone and has an integrable singularity at $t = 0$. Such choice is commonly used e.g. for time fractional (stochastic) evolution equations as studied in \cite{MR1454409, MR1658451}. When $\alpha > 1$, the fractional kernel becomes increasingly regular and may be used, e.g., to interpolate between the heat and the wave equation perturbed by random noise. The qualitative ergodic behaviour of the stochastic Volterra process differs for $\alpha < 1$ (short-ranged) and $\alpha > 1$ (long-ranged) as demonstrated below first for one-dimensional equations and subsequently for the general case when $(A,D(A))$ is self-adjoint, see also Section~\ref{sec:fractional_kernels} and Lemma \ref{lemma: Mittag-Leffler function}.

\subsection{Mild formulation}

A strong solution of \eqref{VSPDE} is an $(\F_t)_{t \geq 0}-$adapted process $u \in L_{loc}^1(\R_+; D(A))$ such that \eqref{VSPDE} holds $\P \otimes dt$-a.e. where it is implicitly assumed that all (stochastic) integrals are well-defined. The requirement that $u$ takes values in $D(A)$ is often too restrictive for specific applications. Our approach is based on the approach through resolvent operators, as introduced in the stochastic case by \cite{MR1454409, MR1658451}. This approach mimics the usual mild formulation for evolution equations, by taking into account the resolvent operator (or fundamental solution operator) instead of the semigroup generated by the leading operator, see also \cite{MR1954075, FHK22} for further applications of this approach.

For this purpose, let us define for $\rho \in \{k,h\}$ the resolvent operator $E_{\rho}$ as the unique solution of the linear Volterra equation
\begin{align*}
 E_{\rho}(t)x = \rho(t)x + \int_0^t k(t-s)E_{\rho}(s)Ax \, ds, \qquad x \in D(A).
\end{align*}
When setting $\rho = k = 1$, we obtain $E_{\rho}(t) = e^{tA}$ and hence the resolvent operator $E_{\rho}$ generalises to Volterra equations the role of the $C_0$-semigroup generated by $A$ for evolution problems. 
Without further assumptions on $\rho$ and $k$ and $(A, D(A))$, the resolvent operator $E_{\rho}$ may not exist, see Section~\ref{sec:resolvent} for further details and references. For given $g$, we let $Gg$ be the unique solution of the linear Volterra equation
\begin{align}\label{eq: Gg definition}
 Gg(t) = g(t) + A \int_0^t k(t-s)Gg(s)ds
\end{align}
provided it exists. In Section~\ref{sec:resolvent} we show that $Gg$ exists if and only if $g \in L_{loc}^1(\R_+; H)$ and $E_k \ast g \in L_{loc}^1(\R_+; D(A))$. In particular, the unique solution is given by $Gg(t) = g(t) + A \int_0^t E_k(t-s)g(s)ds$, see \eqref{eq: solution formula Ek}. Further examples are given in Example \ref{example: special choice}.

In Section~\ref{sec:strong_mild} we first prove that each strong solution of \eqref{VSPDE} satisfies the \textit{Volterra mild formulation}
\begin{align}\label{mild formulation intro}
 u(t) = Gg(t) + \int_0^t E_k(t-s)F(u(s))ds + \int_0^t E_h(t-s)\sigma(u(s))dW_s.
\end{align}
Conversely, if $u$ satisfies the Volterra mild formulation \eqref{mild formulation intro} and $u \in L_{loc}^1(\R_+; D(A))$, then $u$ is also a strong solution of \eqref{VSPDE}, see Proposition \ref{thm: mild formulation characterization}. Henceforth, we call a measurable process $u$ mild solution of \eqref{VSPDE} if it is $(\F_t)_{t \geq 0}$-adapted and satisfies \eqref{mild formulation intro} $\P \otimes dt$-a.e., where it is implicitly assumed that all integrals are well-defined and that $Gg, E_k, E_h$ exist. This notion of mild solutions unifies the definitions used in \cite{MR1954075, MR1658451, FHK22}. Moreover, when $k(t) = h(t) = t$ this formulation includes cosine families used for the study of the stochastic wave equation. In the remaining part of  Section~\ref{sec:strong_mild}, we study the existence, uniqueness and regularity in space-time for mild solutions of \eqref{VSPDE}.

\subsection{Limit distributions}

Our main contribution addresses the existence of limit distributions for mild solutions obtained from \eqref{mild formulation intro}. For Markovian systems, such a problem may be studied by a variety of methods based on the Markov property. Thus it is tempting to cast our framework into the Markovian setting. By adequately extending the state-space of the process, it is possible to recover the Markov property at the expense that the new state-space necessarily becomes infinite-dimensional (even if $H$ is finite-dimensional), see e.g. \cite{MR2401950, MR3057145, MR2511555, MR4181950, H23}. The resulting Markov process is called \textit{Markovian lift}. While there is no canonical way of constructing Markovian lifts, the stochastic equation satisfied by such a Markovian lift has, for singular Volterra kernels $k,h$, typically either unbounded coefficients or the domains of the "lifted" operators become more involved. Both effects reflect the possibly singular nature of the kernels $k$ and $h$ which makes this approach a nontrivial mathematical task.

Using the Markovian lift onto the Filipovic space of forward curves the authors have studied in \cite{Benth} limit distributions and stationary processes for stochastic Volterra equations on Hilbert spaces under the assumption that $k,h$ belong to the Filipovic space, i.e., have $W_{loc}^{1,2}(\R_+)$ regularity and satisfy $\int_0^{\infty}|\rho'(t)| e^{\lambda t}dt < \infty$ for some $\lambda > 0$ and $\rho \in \{k,h\}$. The latter rules out, e.g., the fractional Riemann-Liouville kernel \eqref{eq: fractional kernels}. For other related results on the long-time behaviour of finite-dimensional Volterra processes, we refer to \cite{FJ2022} and \cite{JPS22} where affine Volterra processes are studied. 

In this work, we study the general case applicable to a large class of possibly singular kernels including the fractional Riemann-Liouville kernels. In contrast to \cite{Benth, FJ2022, JPS22}, our methodology directly employs the \textit{mild formulation} of \eqref{VSPDE} to prove global stability and contraction bounds on the solution started at two different initial conditions. The latter allows us to deduce the existence of limit distributions via a \textit{coupling argument} that extends \cite{MR4241464} to stochastic Volterra equations. Such a coupling argument reflects the possibility of restarting the process after time $\tau$, allowing us to use the stability estimates from Section~\ref{sec: contraction}. 

Let $L_s(U, V)$ be the space of bounded linear operators from $U$ to $V$ equipped with the strong operator topology, and let $L_2(U, V)$ be the space of Hilbert-Schmidt operators from $U$ to $V$. For $p \in [1,\infty]$, let $L_{loc}^p(\R_+; L_s(U,V))$ stand for the $L_{loc}^p$-space of strongly measurable functions with values in $L_s(U,V)$. We assume that the following conditions are satisfied:
\begin{enumerate}[leftmargin = *]
 \item[(A1)] $(A,D(A))$ is a closed and densely defined linear operator on $H$, $k \in L_{loc}^1(\R_+)$, $h \in L_{loc}^2(\R_+)$, and the $k,h$-resolvents $E_k,E_h \in L_{loc}^1(\R_+; L_s(H))$ exist.
 
 \item[(A2)] There exist separable Hilbert spaces $H_F, H_{\sigma}, V$ such that 
 \[
  V \subset H, \qquad H \subset H_F, \qquad H \subset H_{\sigma}
 \]
 continuously, $E_k \in L_{loc}^1(\R_+; L_s(H_F, V))$, and $E_h \in L_{loc}^1(\R_+; L_s(H_{\sigma}, V))$.
 
 \item[(A3)] There exists a constant $C_{F,\mathrm{lip}} \geq 0$ such that $F: H \longrightarrow H_F$ satisfies \[
  \|F(x) - F(y)\|_{H_F} \leq C_{F,\mathrm{lip}} \|x-y\|_H, \qquad x,y \in H.
 \]
 Moreover, $\sigma: H \longrightarrow L(U,H_{\sigma})$ is strongly continuous, $E_h(t)\sigma(x) \in L_2(U, V)$ holds for a.a. $t > 0$ and each $x \in H$, and there exist $K_{\mathrm{lin}}, K_{\mathrm{lip}} \in L_{loc}^2(\R_+)$ such that, for a.a. $t > 0$ and all $x,y \in H$,
 \begin{align*}
  \| E_h(t)\sigma(x)\|_{L_2(U, V)} &\leq K_{\mathrm{lin}}(t)\left(1 + \|x\|_H\right), \qquad 
  \\ \| E_h(t)(\sigma(x) - \sigma(y))\|_{L_2(U, V)} &\leq K_{\mathrm{lip}}(t)\|x-y\|_H.
 \end{align*} 
\end{enumerate} 

Here and below we let $i: V \longrightarrow H$ be the continuous embedding with operator norm $\|i\|_{L(V,H)}$. Moreover, let $C_{F,\mathrm{lin}}$ be the linear growth constant of $F$, i.e.
\begin{align}\label{eq: F linear}
  C_{F,\mathrm{lin}} = \sup_{x \in H}\frac{\|F(x) \|_{H_F}}{1+\|x\|_H}.
\end{align}
For later convenience, we introduce a more general mild formulation where $Gg$ is replaced by an abstract function $\xi$, i.e., we study the equation
\begin{align}\label{eq: mild equation xi}
 u(t; \xi) &= \xi(t) + \int_0^t E_k(t-s)F(u(s;\xi))ds + \int_0^t E_h(t-s)\sigma(u(s;\xi))dW_s.
\end{align}
A solution $u$ of \eqref{eq: mild equation xi} is a measurable process that is $(\F_t)_{t \geq 0}$-adapted and satisfies \eqref{eq: mild equation xi} $\P\otimes dt$-a.e., where it is implicitly assumed that all integrals are well-defined. 

Let $\mathcal{P}(V)$ be the space of all Borel probability measures over $V$ and let $\mathcal{P}_p(V)$ be the subspace of all probability measures with finite $p$-th moment. Then $\mathcal{P}_p(V)$ becomes a Polish space when equipped with the $p$-Wasserstein distance
\[
 W_p(\pi, \widetilde{\pi}) = \left(\inf_{\nu \in \mathcal{C}(\pi,\widetilde{\pi})} \int_{V\times V} \| x-y\|_V^p \nu(dx,dy) \right)^{\frac{1}{p}}
\]
where $\mathcal{C}(\pi,\widetilde{\pi})$ denotes the set of all couplings of $\pi$ and $\widetilde{\pi}$. For additional details on couplings and Wasserstein distances we refer to \cite{MR2459454}. Given two random variables $X,Y$, we write by abuse of notation $W_p(X,Y)$ instead of $W_p(\mathcal{L}(X), \mathcal{L}(Y))$ where $\mathcal{L}(X), \mathcal{L}(Y)$ denote their laws.

\begin{Theorem}[general case]\label{thm: 1}
 Suppose that conditions (A1) -- (A3) are satisfied. Then for each $\F_0$-measurable $\xi \in L_{loc}^q(\R_+; L^p(\Omega, \P; V))$ with $p \in [2,\infty)$ and $q \in [p,\infty]$, there exists a unique solution $u(\cdot;\xi) \in L_{loc}^q(\R_+; L^p(\Omega, \P; V))$. Moreover, the following assertions hold:
 \begin{enumerate}
     \item[(a)] If $E_k \in L_{loc}^{\frac{p}{p-1}}(\R_+; L_s(H_F, V))$, $\xi \in L_{loc}^{\infty}(\R_+; L^p(\Omega, \P; V))$, and for $0 \leq s < t$ there exists a function $\widetilde{K}(s,t,\cdot) \in L_{loc}^2([s,t])$ such that
     \begin{align}\label{eq: continuity 1}
      \lim_{t-s \to 0}\int_0^s \widetilde{K}(t,s,r)^2 dr = 0,
     \end{align}
     and for all $x \in H$
     \begin{align}\label{eq: continuity 2}
     \| (E_h(t-r) - E_h(s-r))\sigma(x)\|_{L_2(U,V)} \leq \widetilde{K}(t,s,r)(1 + \|x\|_V),
     \end{align}
     then $u(\cdot;\xi) - \xi \in C(\R_+; L^p(\Omega, \P; V))$.
     
     \item[(b)] If $F(0) = 0$, $\sigma(0) = 0$, $\xi \in L^q(\R_+; L^2(\Omega, \P; V))$, and
     \begin{align}\label{eq: p2 stability}
      3\|i\|_{L(V,H)}^2\left( C_{F}^2 \|E_k\|_{L^1(\R_+; L(H_F,V))}^2 + \|K_{\mathrm{lip}}\|_{L^2(\R_+)}^2 \right) < 1,
     \end{align}
     then $u \in L^q(\R_+; L^2(\Omega, \P; V))$.
     
     \item[(c)] Suppose that $\xi \in L_{loc}^{\infty}(\R_+; L^p(\Omega, \P; V)) \cap C((0,\infty); L^2(\Omega, \P; V))$ is such that
     \begin{align}\label{eq: xi limit}
      \lim_{t \to \infty}\E\left[\|\xi(t) - \xi(\infty)\|_V^2 \right] = 0
     \end{align}
     for some $\xi(\infty) \in V$, $E_k \in L_{loc}^{\frac{p}{p-1}}(\R_+; L_s(H_F, V))$, and
     \begin{align}\label{eq: p2 limit}
      3\|i\|_{L(V,H)}^2 &\bigg(2 C_{F, \mathrm{lin}}^2 \|E_k\|_{L^1(\R_+; L(H_F,V))}^2 
      +  \|K_{\mathrm{lin}}\|_{L^2(\R_+)}^2 \bigg) < 1.
     \end{align}
     Then $u(\cdot;\xi) - \xi \in C_b(\R_+; L^2(\Omega; V))$ and there exists a unique probability measure $\pi_{\xi} \in \mathcal{P}_2(V)$ such that
     \[
      \lim_{t \to \infty} W_2(u(t;\xi), \pi_{\xi}) = 0.
     \]
     Moreover, if $\eta \in L_{loc}^{\infty}(\R_+; L^p(\Omega, \P; V)) \cap C((0,\infty); L^2(\Omega, \P; V))$ also satisfies \eqref{eq: xi limit} with $\eta(\infty) = \xi(\infty)$, then $\pi_{\xi} = \pi_{\eta}$.
 \end{enumerate}
\end{Theorem}

The fact $\pi_{\xi} = \pi_{\eta}$ whenever $\xi(\infty) = \eta(\infty)$ provides a characterisation of possible limit distributions. In particular, the limit distribution of \eqref{VSPDE} is unique provided the limit of $\xi(t) = Gg(t)$ as $t \to\infty$ is independent of $g$. Sufficient conditions and particular examples are discussed below, see also Section~\ref{sec:fractional_kernels}. The space $V$ incorporates additional space regularity of the solution. For concrete applications one typically takes $H_F, H_{\sigma}, V$ as the fractional domain spaces $D((-A)^{\vartheta})$ for different values of $\vartheta$. 

The existence and uniqueness of solutions follow both from Theorem \ref{thm: existence and uniqueness resolvent cylindrical case} where also continuity with respect to $\xi$ is shown. Assertion (a) is a particular case of Theorem \ref{cor: continuous modification} while assertion (b) follows from Theorem \ref{thm: uniform moment}. Finally, assertion (c) is proven in Theorem \ref{thm: limit distribution}. Under stronger conditions, it is also possible to obtain convergence in the $p$-Wasserstein distance with $p \geq 2$ as stated in Theorem \ref{thm: limit distribution}.

By inspection of the proofs, it turns out that the results can be further strengthened for the case of additive noise. Namely, if $\sigma(x) = \sigma_0$ for each $x \in H$, then $K_{\mathrm{lip}}(t) = 0$ and $K_{\mathrm{lin}}(t) = \| E_h(t)\sigma_0\|_{L_2(U,H)}$. The latter allows us to cover the full range $1 \leq p < \infty$ and improve the constants in the dissipativity conditions \eqref{eq: p2 stability} and \eqref{eq: p2 limit}.

\begin{Theorem}[additive noise]\label{thm: 2}
 Suppose that conditions (A1) -- (A3) are satisfied with $\sigma(x) \equiv \sigma_0$ for $x \in H$. Then for each $\F_0$-measurable $\xi \in L_{loc}^q(\R_+; L^p(\Omega, \P; V))$ with $p \in [1,\infty)$ and $q \in [q,\infty]$, there exists a unique solution $u(\cdot;\xi) \in L_{loc}^q(\R_+; L^p(\Omega, \P; V))$. Moreover, the following assertions hold:
 \begin{enumerate}[leftmargin = *]
     \item[(a)] If $E_k \in L_{loc}^{\frac{p}{p-1}}(\R_+; L_s(H_F, V))$, then $u(\cdot;\xi) - \xi \in C(\R_+; L^p(\Omega, \P; V))$.
     \item[(b)] Suppose that $\xi \in L_{loc}^{\infty}(\R_+; L^p(\Omega, \P; V)) \cap C((0,\infty); L^p(\Omega, \P; V))$ is such that \eqref{eq: xi limit} holds for some $\xi(\infty) \in V$, $E_k \in L_{loc}^{\frac{p}{p-1}}(\R_+; L_s(H_F, V))$, and
     \begin{align}\label{eq: limit distribution additive case}
      \|i\|_{L(V,H)} C_{F, \mathrm{lin}} \|E_k\|_{L^1(\R_+;L(H_F, V))} < 1.
     \end{align}
     Then $u(\cdot;\xi) - \xi \in C_b(\R_+; L^p(\Omega; V))$ and there exists a unique probability measure $\pi_{\xi} \in \mathcal{P}_p(V)$ such that
     \[
      \lim_{t \to \infty} W_p(u(t;\xi), \pi_{\xi}) = 0.
     \]
     Moreover, if $\eta \in L_{loc}^{\infty}(\R_+; L^p(\Omega, \P; V)) \cap C((0,\infty); L^p(\Omega, \P; V))$ also satisfies \eqref{eq: xi limit} with $\eta(\infty) = \xi(\infty)$, then $\pi_{\xi} = \pi_{\eta}$.
 \end{enumerate}
\end{Theorem}

The next remark shows that conditions \eqref{eq: p2 stability}, \eqref{eq: p2 limit}, and \eqref{eq: limit distribution additive case} are extensions of the classical dissipativity conditions for Markov processes as used, e.g., in \cite[Section 11.6]{MR3236753} and \cite[Example 7.1]{MR2560625} to study invariant measures for Markov processes.

\begin{Remark}
 Take $k = 1$, then $E_k(t) = e^{tA}$ and hence the integrability of $E_k$ is equivalent to uniform exponential stability of $(e^{tA})_{t \geq 0}$. Suppose that for some $\delta > 0$ the operator $(A,D(A))$ satisfies $\langle Ax,x \rangle_H \leq - \delta \| x\|_H^2$ for each $x \in D(A)$. Then by Lumer-Philips theorem $\| e^{tA}\|_{L(H)} \leq e^{-\delta t}$ and hence \eqref{eq: limit distribution additive case} is satisfied for $V = H$ whenever 
 \[
  C_{F,\mathrm{lin}} < {\delta}.
 \]
 If additionally $h(t) = 1$, then $E_h(t) = e^{tA}$ and similar conditions are also available for \eqref{eq: p2 stability} and \eqref{eq: p2 limit}. Such conditions are satisfied if the nonlinearities are small enough.
\end{Remark}

Although our results require Lipschitz continuous nonlinearities and Gaussian noise, both assumptions are not essential for the mild formulation. Despite the a.s. continuity of sample paths, the arguments naturally extend to the case of L\'evy driven stochastic Volterra equations when taking into account the methodology of \cite{MR2356959}. Cases with more subtle (non-Lipschitz) coefficients may be also treated by the same mild formulation, but require additional care that goes beyond the scope of this work and is henceforth left for future research.

\subsection{Characterization of limit distributions}

Since the stochastic dynamics is non-Markov it is not surprising that the limit distribution $\pi_{\xi}$ depends on the initial state of the process. Our results show that such a dependence is completely characterized by all possible limits $\xi(\infty)$ of $\xi$ in \eqref{eq: xi limit}. In some specific cases, it turns out that such limits are independent of $\xi(\infty)$, i.e. \textit{limit distributions are unique}. 

To demonstrate this, let us suppose that $\int_0^{\infty} \| E_k(s)\|_{L(H_F,V)}ds < \infty$ and let $g \in C(\R_+; D(A))$ be such that there exists $g(\infty) \in D(A)$ with $\lim_{t \to \infty}g(t) = g(\infty)$ in $D(A)$. Then $Gg$ given by \eqref{eq: Gg definition} satisfies $Gg \in C(\R_+; H)$ and it is easy to see that 
\[
 \lim_{t \to \infty}Gg(t) = g(\infty) + \int_0^{\infty}E_k(s)Ag(\infty)ds =: Gg(\infty).
\]
If $k$ has a resolvent of the first kind, then Lemma \ref{lemma: 3.3} implies that the $1$-resolvent $E_1$ exists. By Remark \ref{remark: EkE1} and since $(A,D(A))$ is closed, we obtain $\int_0^{\infty}E_k(s)ds \in L(H,D(A))$ and $Gg(\infty) = \left(\mathrm{id}_H + A\int_0^{\infty}E_k(s)ds\right)g(\infty) = E_1(\infty)g(\infty)$. Hence all limit distributions are parameterized by
\begin{align*}
 \left[ \mathcal{N}(E_1(\infty))\right]^{\perp} = \left[\mathcal{N}\left(\mathrm{id}_H + A\int_0^{\infty}E_k(s)ds\right)\right]^{\perp},
\end{align*}
where $\mathcal{N}(B)$ denotes the null space of the linear operator $B$. In particular, multiple limit distributions occur if and only if $E_1(\infty) \neq 0$. For finite-dimensional affine Volterra processes (with state-space $\R_+^m$), a similar observation was made in \cite{FJ2022}. This work shows that such an effect is generic for stochastic Volterra equations.

A complete characterization for the convergence of $E_1(t)$ as $t \to \infty$ was obtained in \cite{MR1147871}. In particular, using \cite[Theorem 4.6]{MR1147871} we may obtain the following characterization of possible limits $E_1(\infty)$.

\begin{Remark}
    Suppose that the $1$-resolvent $E_1$ exists, is strongly continuous, and satisfies $\sup_{t \geq 0}e^{-\lambda t}\|E_1(t)\|_{L(H)} < \infty$ for each $\lambda > 0$. Assume that there exists $\alpha_0 \geq 0$ such that $\int_0^{\infty}e^{-\alpha_0 t}|k(t)| dt < \infty$. Then the Laplace transform $\widehat{k}$ of $k$ has a meromorphic extension onto the right-half plane, and the limit $E_1(\infty)$, provided it exists, satisfies
    \[
    E_1(\infty) = \begin{cases} (\mathrm{id}_H - \widehat{k}(0)A)^{-1}, & \text{ if } \widehat{k}(0) \neq +\infty
    \\ P, & \text{ if } \widehat{k}(0) = \infty
    \end{cases}
    \]
    where $\widehat{k}(0) = \lim_{z \to 0+}\widehat{k}(z) \in \C \cup \{\infty\}$, and $P \in L(H)$ denotes the projection onto the null-space of $(A, D(A))$.
\end{Remark}
Hence, when $\widehat{k}(0) = + \infty$, we have $E_1(\infty) = P$ and limit distributions are characterized as in the Markovian case where uniformly stable semigroups correspond to $P = 0$. In contrast, for $P \neq 0$ multiple limit distributions, as studied in \cite{MR4241464}, may appear. When $\widehat{k}(0) \neq + \infty$, we necessarily have $E_1(\infty) \neq 0$ and the presence of memory will be reflected in the limit distribution. The latter is, e.g., always the case when $k \in L^1(\R_+)$.

\subsection{Structure of this work}

In the next section, we illustrate our results first for the case of one-dimensional equations and then study the nonlinear heat equation with fractional noise. Section~\ref{sec:resolvent} contains general properties of the $\rho$-resolvent $E_{\rho}$, sufficient conditions for its existence and application to solutions of deterministic linear Volterra equations. Further sufficient conditions for the existence and integrability of resolvents are discussed in Section~\ref{sec:fractional_kernels} under the additional assumption that $(A,D(A))$ is self-adjoint and has a discrete spectrum. In Section~\ref{sec:strong_mild} we prove the equivalence between \eqref{VSPDE} and its mild formulation, and then study the existence, uniqueness and regularity in time for this mild formulation. Finally, Section~\ref{sec:limit_dist} concerns the long-time behaviour of such mild solutions where we first prove general stability bounds for mild solutions and afterwards use a coupling argument to deduce the existence of limit distributions. Finally, a few technical results on one-dimensional linear Volterra equations are collected in Appendix~\ref{sec:1dvolterra}.

\section{Examples for self-adjoint $(A, D(A))$}

\subsection{One-dimensional}

For better illustration, let us first consider the case $H = \R$ with fractional kernels given as in \eqref{eq: fractional kernels}. The results stated below appear to be new even in this relatively simple case. Suppose that $A \in \R$ and $F,\sigma: \R \longrightarrow \R$ are globally Lipschitz continuous with constants $C_{F,\mathrm{lip}}, C_{\sigma,\mathrm{lip}} \geq 0$ and linear growth constants $C_{F,\mathrm{lin}}, C_{\sigma,\mathrm{lin}}$ defined as in \eqref{eq: F linear}. Finally, let 
\[
 g(t) = \int_0^t \frac{(t-s)^{\gamma-1}}{\Gamma(\gamma)}g_0(s)ds
\]
where $\gamma \in (0,\alpha]$ and $g_0 \in L_{loc}^p(\R_+)$ with $p \geq 2$ is for simplicity deterministic. Then $g \in L_{loc}^2(\R_+)$, conditions (A1) -- (A3) are satisfied, and Theorem \ref{thm: existence and uniqueness resolvent cylindrical case} yields the existence of a unique solution $u \in L^2(\Omega,\F,\P; L_{loc}^2(\R_+))$ of 
\begin{align*}
     u(t;g) &= g(t) + \int_0^t \frac{(t-s)^{\alpha-1}}{\Gamma(\alpha)}(t-s)\left(A u(s;g) + F(u(s;g))\right)ds 
     \\ \notag &\qquad \qquad + \int_0^t \frac{(t-s)^{\beta-1}}{\Gamma(\beta)}\sigma(u(s;g))dW_s
\end{align*}
where $(W_t)_{t \geq 0}$ is a one-dimensional standard Brownian motion. By Theorem \ref{cor: continuous modification}, $u - g$ has a modification with locally H\"older continuous sample paths of order $(\alpha \wedge \beta) - \frac{1}{2} - \e$, compare with \cite[Lemma 2.4]{MR4019885}. For the mild formulation, let us first note that, by \cite[Chapter 1]{MR1050319}, $E_k, E_h$ exist, and using Laplace transforms, it is easy to see that they are given by $E_k(t) = t^{\alpha - 1}E_{\alpha, \alpha}(A t^{\alpha})$ and $E_h(t) = t^{\beta - 1}E_{\alpha, \beta}(A t^{\alpha})$. Here $E_{\alpha,\beta}(z) = \sum_{n=0}^{\infty}\frac{z^n}{\Gamma(\alpha n + \beta)}$ denotes the two-parameter Mittag-Leffler function, see \cite{MR4179587} for an overview of its properties. Consequently, since $g \in L_{loc}^2(\R_+)$, we find that $Gg \in L_{loc}^2(\R_+)$ and Proposition \ref{thm: mild formulation characterization} implies that $u$ satisfies the mild formulation (compare with \cite[Lemma 2.5]{MR4019885}). 

Note that for $p > 1/\gamma$, $Gg \in C(\R_+)$ is given by
\[
 Gg(t) = \int_0^t (t-s)^{\gamma - 1}E_{\alpha,\gamma}(A(t-s)^{\alpha})g_0(s)ds.
\]
Moreover, if $A < 0$, $g_0 \in C((0,\infty))$, and $g_0(t) \longrightarrow x$, then $Gg \in C_b(\R_+)$ and 
\[
  Gg(t) \longrightarrow  \begin{cases} 0, & \text{ for } \gamma \in (0,\alpha)
    \\ |A|^{-1}x, & \text{ for } \gamma = \alpha. \end{cases}
\]
For additional details, we refer to Proposition \ref{prop: Gg example} where the infinite-dimensional case is studied. The above implies that limit distributions are unique if and only if $\gamma \in (0,\alpha)$. Define for $\alpha \in (0,2)$ and $\beta > 0$ the constant
\begin{align}
    \label{eq:c_q}
    c_q(\alpha,\beta) = \int_0^{\infty} \left| t^{\beta - 1} E_{\alpha,\beta}(-t^{\alpha})\right|^q dt, \qquad q \in [1,\infty).
\end{align}
Sufficient conditions on $\alpha,\beta$ that guarantee that $c_q(\alpha,\beta) < \infty$ are given in Lemma \ref{lemma: Mittag-Leffler function}. Note that $c_1(\alpha,\alpha) = 1$ for $\alpha \in (0,1]$, and using the Plancherel identity we obtain for $q = 2$, $\alpha \in (0,2)$ and $\beta > 1/2$
\[
 c_2(\alpha,\beta) = \frac{1}{\pi}\int_0^{\infty} \frac{dr}{r^{2\beta} + 2r^{2\beta - \alpha}\cos(\alpha \pi) + r^{2\beta - 2\alpha}}.
\]
The next theorem is a particular statement of our main results combined with Lemma \ref{lemma: Mittag-Leffler function} to verify the dissipativity conditions. 
\begin{Theorem}
 Suppose that $A < 0$ and $g$ is given as above. Furthermore, assume that one of the following conditions holds:
 \begin{enumerate}
     \item[(a)] $\sigma(x) \equiv \sigma_0$ is constant and $c_1(\alpha)\, C_{F,\mathrm{lin}} < |A|$.
     \item[(b)] $\displaystyle 6 C_{F,\mathrm{lin}}^2 c_1(\alpha)^2 |A|^{- 2} + 3c_2(\alpha,\beta)^2 C_{\sigma, \mathrm{lin}}^2 |A|^{-(2\beta - 1)/\alpha} < 1$.
 \end{enumerate}
 Then there exists a unique limit distribution $\pi_g$ with finite second moments such that $\lim_{t \to \infty}W_2(u(t;g), \pi_g) = 0$. The limit distribution is independent of $g$ if and only if $\gamma \neq \alpha$. 
\end{Theorem}

This theorem is a special case of Theorem \ref{thm: 1} and Theorem \ref{thm: 2} when taking into account Lemma \ref{lemma: Mittag-Leffler function}. In the case of additive noise, the limit distribution actually has finite moments of all orders and convergence holds in the Wasserstein-$p$ distance for any $p \in [2,\infty)$. Furthermore, when $F(0) = \sigma(0) = 0$, $A < 0$, $g_0 \in L^p(\R_+)$ satisfies $1 + \frac{1}{p} < \gamma + \frac{1}{q} < 1 + \alpha + \frac{1}{p}$, and  
\[
  3C_{F,\mathrm{lip}}^2 c_1(\alpha)^2|A|^{-2} + 3 c_2(\alpha,\beta)^2 C_{\sigma,\mathrm{lip}}^2 |A|^{-(2\beta-1)/\alpha} < 1,
\]
then $\int_0^{\infty} \left( \E[|u(t)|^2] \right)^{q/2} dt < \infty$. The latter is a particular case of Theorem \ref{prop: contraction estimate}, the specific representation for $Gg$ and Lemma \ref{lemma: Mittag-Leffler function} used to verify that $Gg \in L^q(\R_+)$.

\subsection{Infinite-dimensional case}
\label{sec: examples}

Next, we discuss the infinite-dimensional case under the assumption that $(A,D(A))$ is self-adjoint and there exists a nondecreasing sequence $(\mu_n)_{n \geq 1} \subset (0,\infty)$ with orthonormal basis $(e^H_n)_{n \geq 1}$ in $H$ such that $Ae^H_n = - \mu_n e^H_n$ holds for $n \in \N$. Define for $\lambda \in \R$ the fractional space
\[
 D((-A)^{\lambda}) := H^{\lambda} := \left\{ x = \sum_{n=1}^{\infty}x_n e_n^H \ : \ \sum_{n=1}^{\infty}\mu_n^{2\lambda}|x_n|^2 < \infty \right\}
\]
with inner product $\langle x,y \rangle_{\lambda} = \sum_{n=1}^{\infty}\mu_n^{2\lambda}x_n y_n$ and induced norm $\| \cdot\|_{\lambda}$. Then $H^{\lambda} \subset H^{\lambda'}$ for $\lambda' < \lambda$, and $(e_n^{H^{\lambda}})_{n\geq 1} = (\mu_n^{-\lambda}e_n^H)_{n \geq 1}$ constitutes an orthonormal basis of $H^{\lambda}$.

\begin{Remark}
 For $\rho \in L_{loc}^1(\R_+)$ and $\mu > 0$, let $e_{\rho}(\cdot; \mu) \in L_{loc}^1(\R_+)$ be the unique solution of
 \begin{align}\label{eq: e rho}
 e_{\rho}(t;\mu) + \mu \int_0^t k(t-s)e_{\rho}(s;\mu)ds = \rho(t).
 \end{align}
 If $\sup_{n \geq 1}|e_{\rho}(\cdot; \mu_n)| \in L_{loc}^p(\R_+)$ holds for some $p \in [1,\infty]$, then the $\rho$-resolvent exists, satisfies $\|E_{\rho}(t)\|_{L(H)} = \sup_{n \geq 1}|e_{\rho}(t;\mu_n)|$, and has representation
 \[
  E_{\rho}(t)x = \sum_{n = 1}^{\infty}e_{\rho}(t;\mu_n) \langle e_n^H, x \rangle_H e_n^H.
 \]
\end{Remark} 

Using this observation, we can obtain estimates for the $L^p$-norm of $E_k, E_h$ in terms of estimates on $e_k(\cdot;\mu_n)$ and $e_h(\cdot;\mu_n)$, see Section \ref{sec:fractional_kernels}. Our assumption on $(A,D(A))$ is, e.g., satisfied for the Dirichlet Laplace operator on a bounded and smooth domain, the strongly damped wave equation, and the plate equation, see \cite[Appendix]{MR3236753}. To illustrate our results, below we apply the estimates given therein to a few particular examples.

In the first statement, we discuss the case of additive noise with general kernels $k$ and $\sigma$ being Hilbert-Schmidt (or equivalently $W$ is a $Q$-Wiener process).
\begin{Theorem}
    Suppose that $F: H \longrightarrow H$ is globally Lipschitz continuous and $\sigma(x) \equiv \sigma_0 \in L_2(U, H)$ is constant. Suppose that $0 < k \in L_{loc}^2(\R_+) \cap C^1((0,\infty))$ is non-increasing, $\ln(k)$ and $\ln(-k')$ are convex on $(0,\infty)$, $\int_0^{\infty} e^{-\e t}k(t)dt < \infty$ for each $\e > 0$, there exists $\delta \in (0,1/2)$ and $C_{\delta} > 0$ such that $k(t) \leq C_{\delta}t^{-\delta}$ for $t \in (0,1)$, and $h = k \ast \nu$ where $\nu$ is a finite measure on $\R_+$. Finally, let $g(t) = \int_0^t k(t-s)g_0(s)ds$ with $g_0 \in L^{\infty}(\R_+; D((-A)^{-1}) \cap C((0,\infty); D((-A)^{-1})$ satisfying $g_0(t) \longrightarrow g_0(\infty)$ in $D((-A)^{-1})$ as $t \to \infty$. If $C_{F, \mathrm{lin}} < \mu_1$, then \eqref{VSPDE} admits a unique mild solution $u \in C_b(\R_+; L^p(\Omega; H))$. Moreover, there exists a unique limit distribution $\pi_g$ on $H$ with finite moments of all orders, $W_p(u(t;g), \pi_g) \longrightarrow 0$, and it is independent of $g$ if and only if $k \not \in L^1(\R_+)$.
\end{Theorem}
\begin{proof}
    Using Lemma \ref{lemma: Ek self adjoint} combined with \eqref{eq: 8} we obtain $\|E_k\|_{L^1(\R_+; L(H))} \leq \mu^{-1}$ and $E_k, E_h \in L^2(\R_+; L_s(H))$. Since $\sigma_0 \in L_2(U,H)$, it is easy to see that conditions (A1) -- (A3) are satisfied with $V = H_{\sigma} = H_{F} = H$, $K_{\mathrm{lin}}(t) = \|E_h(t)\sigma_0\|_{L_2(U,H)}$ and $K_{\mathrm{lip}}(t) = 0$. Proposition \ref{prop: Gg example general} implies that $Gg \in C_b(\R_+; H)$ such that $Gg(t) \longrightarrow \left( \|k\|_{L^1(\R_+)} - A\right)^{-1}g_0(\infty)$. Hence the assertion follows from Theorem \ref{thm: 2} combined with $0 \in \mathcal{\rho}(A)$ due to $\mu_1 > 0$ to conclude that $(-A)^{-1}$ exists.
\end{proof}

Note that each completely monotone kernel $k > 0$ on $(0,\infty)$ is non-increasing such that $\ln(k), \ln(-k')$ are convex, see Remark \ref{remark: sufficient condition for ek}. Thus, this result can be applied for $k(t) = \frac{t^{\alpha-1}}{\Gamma(\alpha)}$ with $\alpha \in (0,1)$, $k(t) = \log(1 + 1/t)$, and $k(t) = \int_2^{\infty} \frac{e^{-\theta t}}{\sqrt{\theta}\log(\theta)^2}d\theta$. Letting $\nu = \delta_0$ yields $h(t) = k(t)$, while letting $\nu = \sum_{i=1}^n c_i \delta_{t_i}$ gives $h(t) = \sum_{i=1}^n c_i k(t-t_i)\1_{\{t > t_i\}}$ as particular examples.

If $\sigma \not \in L_2(U,H)$ (e.g. the identity operator), then the conditions for the existence, uniqueness, and limit distributions require further summability conditions of the eigenvalues $(\mu_n)_{n \geq 1}$ and hence are, in general, dimension dependent. For simplicity, we consider from now on the Dirichlet Laplace operator $A = \Delta_x$ on $\mathcal{O} = [0, \pi]^d$ with domain $D(A) = H^2(\mathcal{O}) \cap H_0^1(\mathcal{O})$. Note that the eigenvalues of $-\Delta_x$ are given by $\mu_{n_1,\dots,n_d} = n_1^2 + \dots + n_d^2$ with $n_1,\dots,n_d \geq 1$. Furthermore, to obtain explicit and somewhat sharp conditions, we shall assume that $k,h$ are given by the fractional kernels \eqref{eq: fractional kernels}.

\begin{Example}[Nonlinear heat equation with additive noise]\label{thm: stochastic heat equation nonlinear drift}
    Consider the nonlinear heat equation on domain $\mathcal{O} = [0,\pi]^d$ given by
    \begin{align*}
        u(t,x) &= \frac{t^{\gamma}}{\Gamma(1+\gamma)}u_0(x) + \int_0^t \frac{(t-s)^{\alpha-1}}{\Gamma(\alpha)}\left(\Delta_{x}u(s,x) + f(u(s,x))\right) ds 
        \\ &\qquad \qquad \qquad \qquad \qquad  + \int_0^t \frac{(t-s)^{\beta - 1}}{\Gamma(\beta)}dW(s,x)
    \end{align*}
    where $0 \leq \gamma \leq \alpha < 2$, $\beta > 1/2$, and $f: \R \longrightarrow \R$ is bounded and globally Lipschitz continuous with Lipschitz constant $C_{f,\mathrm{lip}} > 0$.
    Let $\delta \ge 0$ be taken as specified in the following two cases.
    \begin{enumerate}
        \item[(a)] In any dimension $d \in \{1,2,3\}$, if $\alpha \in (0,1]$ and 
    \begin{align}\label{eq: 5}
     \frac{\alpha d}{4} + \frac{1}{2} < \beta \leq \alpha + \frac{1}{2},
    \end{align}
    define $\delta = 0$.

    \item[(b)] Only in dimension $d=1$, assume \eqref{eq: 5} holds for $d=1$ and assume further that  $\alpha \in (2/3, 2)$; then we can take $\delta > 0$ with the bound
    \begin{align}\label{eq: 9}
     0 \leq \delta < \min\left\{ \frac{3}{4} - \frac{1}{2\alpha},\ \frac{2\beta - 1}{2\alpha} - \frac{1}{4}\right\}.
    \end{align}
    \end{enumerate}
    Then for each $u_0 \in D((-\Delta)^{\delta - \gamma/\alpha})$ there exists a unique mild solution 
    \[
     u \in C(\R_+; L^2(\Omega, \P; D((-\Delta)^{\delta}))).
    \]
    Moreover, assume that
    \[
     \begin{cases}(C_{f, \mathrm{lip}}\vee \|f\|_{\infty}) < 1, & \text{ in case (a)}
     \\ \displaystyle (C_{f, \mathrm{lip}}\vee \|f\|_{\infty}) \frac{c_1(\alpha,\alpha)}{1-2\delta} < 1, & \text{ in case (b) }, \end{cases}
    \]
    where the constant $c_1(\alpha,\alpha)$ were defined in \eqref{eq:c_q}. Then there exists a unique limit distribution $\pi_{u_0}$ on $D((-\Delta)^{\delta})$ with finite moments of all orders such that $W_p(u(t), \pi_{u_0}) \longrightarrow 0$ holds for each $p \in [2,\infty)$. This limit distribution is independent of $u_0$ if and only if $\gamma \neq \alpha$.
\end{Example}

\begin{proof}
    Let $H = U = H_F = H_{\sigma} = L^2(\mathcal{O})$, $V = D((-\Delta)^{\delta})$, $\sigma(u) = \mathrm{id}_{L^2(\mathcal{O})}$, and $F(u)(x) = f(u(x))$ for $u \in L^2(\mathcal{O})$. Then we obtain the abstract formulation 
    \[
     u(t) = g(t) + \int_0^t \frac{(t-s)^{\alpha-1}}{\Gamma(\alpha)}\left(Au(s) + F(u(s)) \right) \, ds + \int_0^t \frac{(t-s)^{\beta - 1}}{\Gamma(\beta)} \, dW_s
    \]
    where $(A,D(A)) = (\Delta_x, H^2(\mathcal{O}) \cap H_0^1(\mathcal{O}))$ and $g(t) = \frac{t^{\gamma}}{\Gamma(1+\gamma)}u_0$. Let $\mu_n$ be the ordered Dirichlet eigenvalues of $-\Delta_x$. In case (a), by Lemma \ref{lemma: fractional kernels}.(c) applied with $q = 1,2$ we obtain $\|E_k\|_{L^1(\R_+; L(H))} \leq \mu_1^{-1} = 1$ and $E_k \in L^2(\R_+; L_s(H))$. 
    \\
    {In case (b), Lemma \ref{lemma: fractional kernels}.(a) applied to $q = 1$ yields from the particular form of the eigenvalues
    \begin{align*}
    \int_0^{\infty} \|E_k(t)\|_{L(H,V)} dt 
    &\leq c_1(\alpha, \alpha)\sum_{n=1}^{\infty}n^{-2(1-\delta)} 
    \\ &\leq c_1(\alpha, \alpha)\int_1^{\infty}x^{-2 + 2\delta} \, dx =  \frac{c_1(\alpha, \alpha)}{1-2\delta} < \infty
    \end{align*}
    with the series being finite provided that $d = 1$ and $\delta < \frac{1}{2}$. Finally, Lemma \ref{lemma: fractional kernels}.(a) applied to $q = 2$ gives 
    \begin{align*}
        \int_0^{\infty}\|E_k(t)\|_{L(H,V)}^2 dt \leq c_2(\alpha,\alpha)\sum_{n=1}^{\infty}\mu_n^{-2(1-\delta) + \frac{1}{\alpha}}
    \end{align*}
    which is finite if and only if $d=1$ and $\delta < \frac{3}{4} - \frac{1}{2\alpha}$. }
    For the diffusion component, we use in both cases Lemma \ref{lemma: fractional kernels}.(b) to find 
    \begin{align*}
        \int_0^{\infty}\|E_k(t)\|_{L_2(U,V)}^2 \, dt = c_2(\alpha,\beta)\sum_{n=1}^{\infty}\mu_n^{-\frac{2\beta}{\alpha} + 2\delta + \frac{1}{\alpha}}.
    \end{align*}
    The latter series is convergent if and only if $d = 1,2,3$, \eqref{eq: 5} holds and $\delta$ satisfies \eqref{eq: 9}. Hence conditions (A1) -- (A3) are satisfied with $C_{F,\mathrm{lip}} = C_{f,\mathrm{lip}}$ and $C_{F,\mathrm{lin}} = (C_{f,\mathrm{lip}} \vee \|f\|_{\infty})$. 
    Finally, using Proposition \ref{prop: Gg example}.(b) with $g_0(t) = u_0$ for $0 < \gamma \leq \alpha$ and Proposition \ref{prop: Gg example general}.(a) for $\gamma = 0$, we see that $Gg \in C_b(\R_+; H)$ with 
    \[
     Gg(t) \longrightarrow \begin{cases} 0, & 0 \leq \gamma < \alpha \\ (-\Delta)^{-1}x, & \gamma = \alpha. \end{cases}
    \]
    The assertion follows from Theorem \ref{thm: 2} combined with the explicit bounds on the $L^1(\R_+; L_s(H, V))$ norm of $E_k$ and since $\|i\|_{L(V,H)} \leq \mu_1^{-\delta} = 1$.
\end{proof}

Condition \eqref{eq: 5} constitutes a trade-off between the dimension $d$ and the fractional parameters $\alpha,\beta$. 
In the classical stochastic Heat equation, we have $\alpha = \beta = 1$ and hence necessarily $d = 1$. However, letting $\alpha = 1$ gives $\beta \in (\frac{d+2}{4}, \frac{3}{2}]$ which shows that higher dimensions can be treated by more regular driving noise. Likewise, if $\beta = 1$, then $\alpha \in (1/2,2/d)$ and no dimension $d > 3$ can be studied. Finally, if $\alpha = \beta$, then necessarily $d = 1$ and $\alpha = \beta \in (2/3, 2)$. The spatial regularity obtained in $d=1$ always satisfies $\delta < 1/2$ and hence we cannot guarantee that $u,\pi_g$ have a spatially continuous representative.  

Below we state an analogous result with multiplicative noise where $F \equiv 0$. Consequently, the mild formulation does not involve $E_k$ and hence we obtain more spatial regularity in the full range of dimensions $d=1,2,3$.

\begin{Example}[Heat equation with multiplicative noise]\label{thm: stochastic heat equation nonlinear diffusion}
    Consider the nonlinear heat equation on domain $\mathcal{O} = [0,\pi]^d$ with dimension $d \geq 1$ given by
    \begin{align*}
        u(t,x) = \frac{t^{\gamma}}{\Gamma(1+\gamma)}u_0(x) + \int_0^t \frac{(t-s)^{\alpha-1}}{\Gamma(\alpha)}\Delta_{x}u(s,x) ds + \int_0^t \frac{(t-s)^{\beta - 1}}{\Gamma(\beta)}f(u(s,x))dW(s,x)
    \end{align*}
    where $0 \leq \gamma \leq \alpha < 2$, $\beta > 1/2$, and $f: \R \longrightarrow \R$ is bounded and globally Lipschitz continuous with Lipschitz constant $C_{f,\mathrm{lip}} > 0$. If $d \in \{1,2,3\}$ and $\delta \geq 0$ satisfies
    \begin{align}\label{eq: 10}
     \frac{1}{2} + \left( \delta + \frac{d}{4}\right)\alpha < \beta \leq \alpha + \frac{1}{2},
    \end{align}
    then, for each $u_0 \in D((-\Delta)^{\delta - \gamma/\alpha})$, there exists a unique solution 
    \[
     u \in C(\R_+; L^2(\Omega, \P; D((-\Delta)^{\delta}))).
    \]
    Finally, if for $\omega_1 = 1$, $\omega_2 = \pi/2$ and $\omega_3 = \pi$
    \begin{align*}
     3 \, C_{f,\mathrm{lin}}^2 c_2(\alpha,\beta) \frac{\omega_d}{\frac{4\beta - 2}{\alpha} - 4\delta - d} < 1
    \end{align*}
    holds, then there exists a unique limit distribution $\pi_{u_0}$ on $D((-\Delta)^{\delta})$ with finite second moments such that $W_2(u(t), \pi_{u_0}) \longrightarrow 0$. This limit distribution is independent of $u_0$ if and only if $\gamma \neq \alpha$. 
\end{Example}
\begin{proof}
     Let $H = U = H_F = H_{\sigma} = L^2(\mathcal{O})$, $V = D((-\Delta)^{\delta})$, $(A,D(A)) = (\Delta_x, H^2(\mathcal{O}) \cap H_0^1(\mathcal{O}))$, $F \equiv 0$, and set $\sigma(v)u(x) = f(v(x))u(x)$ for $x \in \mathcal{O}$, and $u \in L^2(\mathcal{O})$. Then we arrive at the abstract formulation 
    \[
     u(t) = g(t) + \int_0^t \frac{(t-s)^{\alpha-1}}{\Gamma(\alpha)}Au(s)ds + \int_0^t \frac{(t-s)^{\beta - 1}}{\Gamma(\beta)}\sigma(u(s))dW_s
    \]
    with $g(t) = \frac{t^{\gamma}}{\Gamma(1+\gamma)}u_0$. 
    Since $F \equiv 0$, it is easy to see from the proofs given in Section \ref{sec:limit_dist}, that no conditions on $E_k$ need to be imposed. The latter is the reason why all dimensions $d=1,2,3$ can be treated simultaneously. 
    For the diffusion component we proceed exactly as in Example \ref{thm: stochastic heat equation nonlinear drift} which gives 
    \begin{align*}
        \int_0^{\infty} \|E_h(t)\|_{L_2(H,V)}^2 dt = c_2(\alpha,\beta)\sum_{n=1}^{\infty} \mu_n^{-\frac{2\beta}{\alpha} + 2\delta + \frac{1}{\alpha}} < \infty
    \end{align*}
    whenever $d = 1,2,3$ and \eqref{eq: 10} holds. Hence conditions (A1) and (A2) are satisfied with $C_{F, \mathrm{lip}} = C_{F, \mathrm{lin}} = 0$ and it is easy to see that condition (A3) is satisfied with 
    $K_{\mathrm{lin}}(t)^2 = C_{f,\mathrm{lin}}^2 \sum_{n=1}^{\infty} |e_h(t;\mu_n)|^2$ and $K_{\mathrm{lip}}(t)^2 = C_{f, \mathrm{lip}}^2 \sum_{n=1}^{\infty} |e_h(t;\mu_n)|^2$. Using Lemma \ref{lemma: Mittag-Leffler function} combined with the explicit representation $e_h(t;\mu_n) = t^{\beta-1}E_{\alpha,\beta}(-\mu_n t^{\alpha})$, we obtain
    \begin{align*}
     \int_0^{\infty}\sum_{n=1}^{\infty} |e_h(t;\mu_n)|^2 dt =  C_{f,\mathrm{lin}}^2 c_2(\alpha,\beta) \sum_{n=1}^{\infty} \mu_n^{- \frac{2\beta - 1}{\alpha}+2\delta}.
    \end{align*}
    When $d = 1$ we obtain from $\mu_n = n^2$
    \[
     \sum_{n=1}^{\infty} \mu_n^{- \frac{2\beta - 1}{\alpha}+2\delta} = \sum_{n=1}^{\infty} n^{4\delta - \frac{4\beta - 2}{\alpha}} \leq \int_1^{\infty} x^{4\delta - \frac{4\beta - 2}{\alpha}} dx = \frac{1}{\frac{4\beta - 2}{\alpha} - 4\delta - 1}
    \]
    Similarly, when $d = 2,3$ we bound this series by using the explicit representation of the eigenvalues and a comparison with the volume integrals over the ball of radius $n^2$, i.e.
    \begin{align*}
    \sum_{n=1}^{\infty} \mu_n^{- \frac{2\beta - 1}{\alpha}+4\delta}  
    &= \sum_{n_1,\dots, n_d = 1}^{\infty}(n_1^2 + \dots + n_d^2)^{- \frac{2\beta - 1}{\alpha}+4\delta}  
    \\ &\leq \frac{\omega_d}{4} \int_1^{\infty} x^{-\frac{4\beta - 2}{\alpha}+4\delta + d - 1}  dx
    = \frac{\omega_d/4}{\frac{4\beta - 2}{\alpha} - 4\delta - d}
    \end{align*}
    where $\omega_d$ denotes the surface of the unit ball in $\R^d$. Finally, $Gg \in C_b(\R_+; H)$ with $Gg(t) \longrightarrow Gg(\infty)$ can be shown as in the proof of Example \ref{thm: stochastic heat equation nonlinear drift}. The assertion follows from Theorem \ref{thm: 1}.
\end{proof}

Let $\beta = \alpha + \frac{1}{2}$. Then we may take any $\delta < 1 - \frac{d}{4}$. In particular, when $d = 1$, the solution belongs to $D((-\Delta)^{3/4 - \e}) \subset C^{1/2 - \e}(\mathcal{O})$ for each $\e > 0$ by Sobolev embeddings and hence has a H\"older continuous representative in the spatial variables.

\section{Resolvent operators and linear equations}
\label{sec:resolvent}

For $p \in [1,\infty]$, we let $L^p([0,T];H)$ be the Banach space of Bochner $p$-integrable functions on $[0,T]$ with $T > 0$ and norm denoted by $\| \cdot\|_{L^p([0,T];H)}$. Let $L_{loc}^p(\R_+;H)$ be the space of locally Bocher $p$-integrable functions. If $H = \R$, then we simply write $L_{loc}^p(\R_+)$ or $L^p([0,T])$, respectively. When $V$ is another separable Hilbert space, we let $L(H, V)$ be the space of bounded linear operators from $H$ to $V$. Note that this space is not separable, and hence the notion of Bochner integrability is too restrictive for $L(H, V)$-valued functions. A family of operators $(E(t))_{t \in [0, T]} \subset L(H, V)$ is called strongly operator measurable if $t \longmapsto E(t)x \in V$ is measurable for each $x \in H$. Consequently, $\|E(\cdot)\|_{L(H,V)} = \sup_{n \geq 1}\|E(\cdot)x_n\|_V$ is also measurable, where $(x_n)_{n \geq 1} \subset H$ is a dense subset of the unit ball in $H$. Thus, for $p \in [1,\infty]$, we let $L^p([0,T]; L_s(H, V))$ be the Banach spaces of strongly operator measurable mappings $\varphi: [0,T] \longrightarrow L(H,V)$ equipped with the norm
\[
 \| E\|_{L^p([0,T];L(H,V))} = \left( \int_0^T \|E(t)\|_{L(H,V)}^p dt \right)^{\frac{1}{p}}, \qquad p \in [1,\infty)
\] 
and with the obvious adaption for $p = \infty$. The integral $\int_0^T E(t)dt$ is well-defined in the strong operator topology. Similarly, we define $L^p(\R_+; L_s(H,V))$ with norm $\|\cdot\|_{L^p(\R_+; L(H,V))}$, and let $L_{loc}^p(\R_+; L_s(H,V))$ be the space of locally strongly measurable functions from $\R_+$ to $L(H,V)$ such that $E|_{[0,T]} \in L^p([0,T], L_s(H,V))$ holds for all $T > 0$. For additional details on measurability for operator-valued mappings, we refer to \cite{MR3564556}.

\subsection{Resolvent operators}

Fix $0 \neq k \in L_{loc}^1(\R_+)$ and let $(A, D(A))$ be a closed and densely defined operator on $H$. We equip $D(A)$ with the graph norm denoted by $\| x \|_{D(A)} = \left(\|x \|_H^2 + \|Ax\|_H^2\right)^{1/2}$ which makes it to a separable Hilbert space. Denote by $(a \ast b)(t) = \int_0^t a(t-s)b(s)ds$ the convolution product on $\R_+$, whenever it is well-defined.

\begin{Definition}\label{def: resolvent operators}
 Let $\rho \in L_{loc}^1(\R_+)$. A family of operators $E_{\rho} \in L_{loc}^1(\R_+; L_s(H))$ is called $\rho$-resolvent, if for a.a. $t \geq 0$, the operator $E_{\rho}(t)$ leaves $D(A)$ invariant, commutes with $A$ on $D(A)$, and satisfies for each $x \in D(A)$  
 \[
  E_{\rho}(t)x = \rho(t)x + \int_0^t k(t-s)AE_{\rho}(s)x \, ds.
 \]     
\end{Definition}

In contrast to \cite{LIZAMA2000278}, this notion of $\rho$-resolvents does not require that $\rho$ is continuous. Moreover, in contrast to \cite[Chapter 1]{MR1238939}, we do not suppose that when $\rho(t) = 1$, the $1$-resolvent $E_1$ is strongly continuous. In this way, we can unify the two special cases $\rho(t) = 1$ and $\rho(t) = k(t)$ corresponding to the resolvent and integral resolvent in the notation of \cite[Chapter 1]{MR1238939}. 

\begin{Remark}
    If $\rho, k \in L_{loc}^p(\R_+)$ for some $p \in [1,\infty]$, then by Young's convolution inequality (see Lemma \ref{lemma: continuity convolution}) it is easy to see that $E_{\rho}(\cdot)x \in L_{loc}^p(\R_+; H)$ for each $x \in D(A)$. If, additionally $\rho \in C((0,\infty))$ and $p \geq 2$, then even $E_{\rho}(\cdot)x \in C( (0,\infty); H)$ for $x \in D(A)$. 
\end{Remark}

For given $k, (A, D(A))$ and $\rho$, the existence of a $\rho$-resolvent is not necessarily guaranteed. Similarly to the classical case of $C_0$-semigroups (where $k = h = 1$), there are Hille-Yosida generation theorems for resolvents of stochastic Volterra equations, see \cite{MR1238939} for the cases $\rho = 1$ and $\rho = k$, and \cite{LIZAMA2000278} for the case where $\rho$ is continuous. Below we provide sufficient conditions for the existence of such $\rho$-resolvents beyond Example \ref{example: special choice}.

\begin{Example}[Finite-dimensional case]
 Let $H = \R^d$ and $\rho, k \in L_{loc}^p(\R_+)$. Then the $\rho$-resolvent $E_{\rho} \in L_{loc}^p(\R_+; \R^{d \times d})$ exists, see \cite[Chapter 1]{MR1050319}. For fractional kernels \eqref{eq: fractional kernels}, i.e., $k(t) = \frac{t^{\alpha - 1}}{\Gamma(\alpha)}$ and $\rho(t) = \frac{t^{\beta - 1}}{\Gamma(\beta)}$ we find $E_{\rho}(t) = t^{\beta - 1}E_{\alpha, \beta}(A t^{\alpha})$, where $E_{\alpha,\beta}(z) = \sum_{n=0}^{\infty}\frac{z^n}{\Gamma(\alpha n + \beta)}$ denotes the two-parameter Mittag-Leffler function, see \cite{MR4179587} for an overview of its properties.
\end{Example}

For $k \in L_{loc}^1(\R_+)$ with $\int_0^{\infty}|k(t)|e^{-\e t}dt < \infty$ for each $\e > 0$, we denote by $\widehat{k}$ the Laplace transform of $k$. The latter is well-defined on the positive half-plane. For $\theta_0 \in (0,\pi/2)$ let $\Sigma(\theta_0) = \{ \lambda \in \C \ : \ |\mathrm{arg}(\lambda)| < \theta_0 \}$ denote the sector of angle $\theta_0$ in the complex half-plane.
Our next example covers the case of $1$-resolvents in the parabolic case.
\begin{Example}[Parabolic case]
    Let $k \in L_{loc}^1(\R_+)$ be such that $\int_0^{\infty}|k(t)|e^{-\e t}dt < \infty$ for each $\e > 0$, $\widehat{k}(z) \neq 0$ and there exists $\theta \in (0,\pi)$ with $|\mathrm{arg}(\widehat{k}(z))| \leq \theta$ for $\mathrm{Re}(z) > 0$. Suppose that $(A,D(A))$ is closed, densely defined, satisfies $\rho(A) \subset \Sigma(\theta)$ and
    \[
     \sup_{z \in \Sigma(\theta)}|z|\| R(z; A)\|_{L(H)} < \infty.
    \]
    Then the $1$-resolvent exists and satisfies $E_1 \in L_{loc}^p(\R_+; L_s(H))$ for each $p \in [1,\infty)$, see \cite[Chapter 3]{MR1238939}. If, additionally $k$ is $2$-regular in the sense that there is $c > 0$ such that
    $|z^n \widehat{k}^{(n)}(z)| \leq c |\widehat{k}(z)|$ for $n \in \{0,1,2\}$ and $\mathrm{Re}(z) > 0$, then $E_1 \in C^1((0,\infty); L(H))$ and $E_1$ is uniformly bounded in the operator norm.
\end{Example}

Sufficient conditions for $k$ to be $2$-regular are discussed in \cite[Chapter 3]{MR1238939}. This example covers particularly the case where $(A, D(A))$ generates a $C_0$-semigroup and $k$ is completely monotone. In such a case, under a slight regularity condition on $k$, one can provide an explicit representation of $E_1$ in terms of the semigroup generated by $(A, D(A))$. 
\begin{Example}[Subordination principle]
 Let $(A,D(A))$ generate a $C_0$-semigroup $(T(t))_{t \geq 0}$ on $H$, suppose that $k$ is completely monotone and
 \[
 \exists \alpha \in [0,1), \ \e > 0: \qquad  \sup_{t \in (0,T)}t^{\alpha - \frac{1}{1+\e}}\| k\|_{L^{1+\e}([0,t])} < \infty.
 \]
 Let $\Psi(\cdot;\lambda) \in L_{loc}^1(\R_+)$ be the unique solution of $\Psi(\cdot;\lambda) + \lambda k \ast \Psi(\cdot;\lambda) = k$, and let $(\mathcal{P}_t(ds))_{t \geq 0}$ be the family of probability measures uniquely determined by their Laplace transforms $\int_0^{\infty} e^{-\lambda s}\mathcal{P}_t(ds) = \Psi(t;\lambda)$. Then, according to \cite[Theorem 1]{MR4493597}, $E_1$ exists and is given by
 \[
  E_1(t) = \int_0^{\infty} T(s) \mathcal{P}_t(ds)
 \]
\end{Example}

Below we show how the $1$-resolvent can be used to obtain information about $E_{\rho}$. As a first step, let us collect a few useful properties of resolvents.

\begin{Proposition}\label{prop: regularity resolvent}
 Let $0 \neq \rho \in L_{loc}^1(\R_+)$ and $E_{\rho}$ be a $\rho$-resolvent. 
 Then $k \ast E_{\rho} \in L_{loc}^1(\R_+; L_s(H,D(A)))$, and for each $x \in H$ one has  
 \[
  E_{\rho}(t)x = \rho(t)x + A \int_0^t k(t-s)E_{\rho}(s)x ds, \ \ \text{for a.a. } t \geq 0.
 \]
 Moreover, for each $f \in L_{loc}^1(\R_+; H)$ one has $k \ast E_{\rho} \ast f \in L_{loc}^1(\R_+; D(A))$ and 
 \[
  \| k \ast E_{\rho} \ast f \|_{L^1([0,T]; D(A))} \leq M(T)\| f\|_{L^1([0,T]; H)}.
 \]
 with $M(T) := \| k \|_{L^1([0,T])} \| E_{\rho}\|_{L^1([0,T]; L(H))} + \| E_{\rho}\|_{L^1([0,T]; L(H))} + \| \rho \|_{L^1([0,T])}$.
\end{Proposition}
\begin{proof}
 Let $x \in D(A)$. Since $E_{\rho}$ leaves $D(A)$ invariant and commutes with $A$ we have $\int_0^t \|A k(t-s)E_{\rho}(s)x\|_H ds = \int_0^t \| k(t-s)E_{\rho}(s)Ax\|_H ds < \infty$, and hence $k \ast E_{\rho}(t)x \in D(A)$ holds for a.a. $t \geq 0$. Moreover, using $A(k \ast E_{\rho})(t)x = E_{\rho}(t)x - \rho(t)x$ we get
 \begin{align*}
    \| k \ast E_{\rho}(t)x \|_{D(A)}
    &= \| k \ast E_{\rho}(t)x\|_H + \| E_{\rho}(t)x - \rho(t)x \|_H
    \\ &\leq  \left( \| k \ast E_{\rho}(t)\|_{L(H)} + \| E_{\rho}(t) \|_{L(H)} + |\rho(t)| \right) \|x\|_H.
 \end{align*}
 Since $D(A)$ is dense, we conclude that $k \ast E_{\rho}(t) \in L(H, D(A))$ holds for a.a. $t \geq 0$. Thus, we obtain for each $x \in H$
 \begin{align*}
     \ \| k \ast E_{\rho}(\cdot)x \|_{L^1([0,T];D(A))} 
     &\leq \| x\|_H \int_0^T \left( \| k \ast E_{\rho}(t)\|_{L(H)} + \| E_{\rho}(t) \|_{L(H)} + |\rho(t)| \right) dt 
    \\ &\leq M(T)\| x\|_H.
 \end{align*}
 Finally, given $f \in L_{\mathrm{loc}}^1(\R_+; H)$, we obtain 
 \begin{align*}
    &\ \| k \ast E_{\rho} \ast f \|_{L^1([0,T]; D(A))}
     \\ &\leq \int_0^T \int_0^t \| (k\ast E_{\rho})(s)f(t-s) \|_{D(A)}ds dt
     \\ &\leq \int_0^T \int_0^t \left( \| k \ast E_{\rho}(s)\|_{L(H)} + \| E_{\rho}(s) \|_{L(H)} + |k(s)| \right)\| f(t-s)\|_H ds dt
     \\ &\leq M(T) \| f\|_{L^1([0,T]; H)}
 \end{align*}
 which proves the assertion. 
\end{proof}

These relations imply that the $\rho$-resolvent is uniquely determined by $(A, D(A))$ and $\rho$. Indeed, let $E_{\rho}, \widetilde{E}_{\rho}$ be two $\rho$-resolvents. Then, for each $x \in H$,
\begin{align*}
 \rho \ast E_{\rho} x = (\widetilde{E}_{\rho} - A(k \ast \widetilde{E}_{\rho}))\ast E_{\rho}x
 = \widetilde{E}_{\rho} \ast \left( E_{\rho}x - A(k \ast E_{\rho}) x\right)
 = {\rho} \ast \widetilde{E}_{\rho} x,
\end{align*}
and hence Titchmarsh's theorem (see \cite[p. 166]{MR1336382}) implies $E_{\rho}(\cdot)x = \widetilde{E}_{\rho}(\cdot)x$. 

For $\rho \in L_{loc}^1(\R_+)$, its resolvent of the first kind provided it exists, is defined as the unique locally finite signed measure $\overline{\rho}$ such that $\rho \ast \overline{\rho} = \overline{\rho} \ast \rho = 1$. Sufficient conditions and further details are given in \cite[Chapter 5.5]{MR1050319}. For instance, $\overline{\rho}$ exists whenever $\rho$ is completely monotone or, more generally, $\rho \neq 0$ is nonnegative and non-increasing. The next lemma illustrates a relationship between $E_{\rho}$ and $E_{\widetilde{\rho}}$ for different $\rho$ and $\widetilde{\rho}$.
\begin{Lemma}\label{lemma: 3.3}
 Let $\rho \in L_{loc}^1(\R_+)$ and suppose that the $\rho$-resolvent $E_{\rho}$ exists. Then:
 \begin{enumerate}
     \item[(a)] For each $g \in L_{loc}^1(\R_+)$ 
     the $\rho\ast g$-resolvent exists and satisfies $E_{\rho \ast g} = E_{\rho} \ast g$.
     
     \item[(b)] If $\rho$ has a resolvent of the first kind $\overline{\rho}$, then $E_1$ exists and $E_1 = E_{\rho} \ast \overline{\rho}$.
 \end{enumerate}
\end{Lemma}
The proof is left to the reader. 
\begin{Remark}\label{remark: EkE1}
 Suppose that the $1$ and the $k$-resolvents exist. Then by Proposition \ref{prop: regularity resolvent} $1 \ast E_k = k \ast E_1 \in L_{loc}^1(\R_+; L_s(H, D(A)))$, and for each $x \in H$
 \begin{align*}
  E_1(t)x = x + A\int_0^t k(t-s)E_1(s)xds = x + A\int_0^t E_k(s)x ds. 
 \end{align*}
 In particular, for $x \in D(A)$, $E_1(\cdot)x$ is absolutely continuous and satisfies $E_1'(t)x = E_k(t)Ax$ for a.a. $t \geq 0$.
\end{Remark}

Below we discuss one natural way such resolvents may be constructed from the knowledge of $E_1$. Namely, if $E_1, E_{\rho}$ both exist, then $E_1 \ast \rho = E_{1 \ast \rho} = E_{\rho} \ast 1$ and hence taking derivatives gives $E_{\rho}(t) = \frac{d}{dt} \int_0^t E_1(t-s)\rho(s)ds$. However, without assuming beforehand that $E_{\rho}$ exists, it is not clear if $E_1 \ast \rho$ is absolutely continuous, thus additional regularity on $\rho$ and on $E_1$ needs to be assumed. Below we carry out such an approach for the case where $E_1$ is analytic.
\begin{Definition}
 The $1$-resolvent $E_1$ is called analytic, if $E_1$ has a strongly continuous version $E_1: \R_+ \longrightarrow L(H)$, and there exists some $\theta_0 \in (0, \pi/2)$ such that $E_1: \R_+ \longrightarrow L(H)$ admits an analytic extension to a sector $\Sigma(\theta_0)$. The family of operators $E_1$ is called of analyticity type $(\omega_0, \theta_0)$ if for each $\theta \in (0, \theta_0)$ and $\omega > \omega_0$ there exists some $M(\omega, \theta) \geq 1$ such that $\| E_1(z) \|_{L(H)} \leq M(\omega, \theta) e^{\omega \mathrm{Re}(z)}$ holds for $z \in \Sigma(\theta)$.
\end{Definition}

Let $E_1$ be an analytic resolvent of type $(\omega_0, \theta_0)$. By \cite[Corollary 2.1]{MR1238939} we find for each $\omega > \omega_0$ and $\theta \in (0, \theta_0)$ some $M(\omega, \theta) \geq 1$ such that
\begin{align}\label{eq: analytic resolvent estimate}
 \| E_1^{(n)}(t) \|_{L(H)} \leq M(\omega, \theta) n! e^{\omega t (1+\sin(\theta))} \sin(\theta)^{-n} t^{-n}, \qquad t > 0, \ \ n \in \N.
\end{align}
In particular, $E_1'$ is almost integrable near the origin. 
Firstly, general generation theorems can be found in \cite[Chapter 2]{MR1238939}. More specifically, the following scheme allows us to construct a large class of examples. 
\begin{Example}\cite[Chapter 2, Section 2.3]{MR1238939}
 The following assertions hold:
 \begin{enumerate}
     \item[(a)] If $k \in L_{loc}^1(\R_+)$ is completely monotone and $(A, D(A))$ generates an analytic semigroup $(e^{tA})_{t \geq 0}$ such that $\sup_{z \in \Sigma(0,\theta)}\|e^{tA}\|_{L(H)} < \infty$, then there exists an analytic resolvent $E_1$ of type $(0,\theta)$ associated with $(k,A,D(A))$.
     
     \item[(b)] If $k(t) = t^{\alpha - 1}/\Gamma(\alpha)$ with $\alpha \in (0,2)$. Then there exists an analytic resolvent $E_1$ if and only if $\Sigma\left( \alpha \pi/2 \right) \subset \rho(A)$ and $\sup_{z \in \Sigma\left( \alpha \pi/2 \right)}|z|\| (z - A)^{-1}\|_{L(H)} < \infty$. In this case, $\|E_1(z)\|_{L(H)}$ is uniformly bounded on $\Sigma\left( \alpha \pi/2  - \varepsilon\right)$ for each $\varepsilon \in (0,\alpha \pi/2)$.

     \item[(c)] Let $k(t) = \int_0^{\infty} \frac{t^{\alpha - 1}}{\Gamma(\alpha)}d\alpha$. If $(A,D(A))$ is such that $z \in \rho(A)$ for $|\mathrm{Im}(z)| \leq \pi/2$ and 
     \[
     \sup_{|\mathrm{Im}(z)| \leq \pi/2}|z| \| (z-A)^{-1}\|_{L(H)} < \infty,
     \]
     then $(k,A,D(A))$ generates a bounded analytic resolvent. On the other-hand if $\Sigma(\phi) \subset \rho(A)$ for some $\phi > 0$ and $\sup_{z \in \Sigma\left( \phi \right)}|z|\| (z - A)^{-1}\|_{L(H)} < \infty$, then $(k,A,D(A))$ generates an analytic resolvent of type $(\omega, \theta)$ where $\omega \geq \exp\left( \frac{\pi/2 + \theta}{\sin(\phi)}\right)/\cos(\theta)$.
 \end{enumerate} 
\end{Example}

The next proposition allows us to construct the $\rho$-resolvent $E_{\rho}$ from an analytic resolvent $E_1$.
\begin{Proposition}\label{prop: E rho from E 1}
 Let $\rho \in L_{loc}^1(\R_+)$ be such that
 \[
  \int_0^T |\rho(t)|^p dt + \int_0^T \left(\int_0^t \frac{|\rho(t)-\rho(s)|}{|t-s|}ds \right)^{p} dt < \infty, \qquad \forall T > 0
 \]
 holds for some $1 \leq p \leq \infty$ with the obvious adaption when $p = \infty$. Assume that the $1$-resolvent exists and is analytic. Then the $\rho$-resolvent exists, satisfies $E_{\rho} \in L_{loc}^p(\R_+; L_s(H))$, and is given by 
 \begin{align}\label{eq: E rho}
  E_{\rho}(t) = E_1(t)\rho(t) + \int_0^t E_1'(t-s)(\rho(s) - \rho(t))ds.
 \end{align}
\end{Proposition}
\begin{proof}
 We only consider the case $p \neq \infty$. The case $p = \infty$ follows in the same way. The right-hand side in \eqref{eq: E rho} is well-defined in $L_{loc}^p(\R_+; L_s(H))$. Indeed, we have $\int_0^T \| E_1(t)\rho\|_{L(H)}^p dt \leq \sup_{t \in [0,T]}\|E_1(t)\|_{L(H)}^p \int_0^T |\rho(t)|^p dt$. For the second term we use \eqref{eq: analytic resolvent estimate} to find a constant $C_T > 0$ such that $\| E_1'(r)\|_{L(H)} \leq C_T r^{-1}$ for $r \in (0,T]$. This gives
 \begin{align*}
     \int_0^T \left| \int_{0}^t \| E_1'(t-s)(\rho(s) - \rho(t))\|_{L(H)} ds \right|^p dt
    &\leq C_T^p \int_0^T \left( \int_{0}^t \frac{|\rho(t) - \rho(s)|}{|t-s|} ds \right)^p dt
 \end{align*}
 with the right-hand side being finite by assumption on $\rho$. Thus letting $v$ denote the right-hand side in \eqref{eq: E rho}, we have $v \in L_{loc}^p(\R_+; L_s(H))$. 
 
 Since $E_1$ leaves $D(A)$ invariant and commutes with $A$, it is easy to see that also $E_1'$ leaves $D(A)$ invariant and commutes with $A$ on $D(A)$.
 Next, we show that $v$ satisfies the resolvent equation and hence is the desired $\rho$-resolvent. Define for $\e \in (0,1)$ the function $v_{\e}(t) = E_1(t+\e)\rho(t) + \int_0^t E_1'(t-s+\e)(\rho(s)-\rho(t))ds$. Then using \eqref{eq: analytic resolvent estimate} for $n = 1$ we can apply dominated convergence to find that $v_{\e}(t) \longrightarrow v(t)$ and $\int_0^t v_{\e}(s)ds \longrightarrow \int_0^t v(s)ds$ for a.a. $t \geq 0$ strongly on $L(H)$. A short computation yields $v_{\e}(t) = \frac{d}{dt}\int_0^t E_1(t-s+\e)\rho(s)ds$ and hence $\int_0^t v_{\e}(s)ds = \int_0^t E_1(t-s+\e)\rho(s)ds$. Letting $\e \to 0$ we find that $\int_0^t v(s)ds = \int_0^t E_1(t-s)\rho(s)ds$. Let $x \in D(A)$. Then $v(\cdot)x, (1\ast v)(\cdot)x \in L_{loc}^1(\R_+; D(A))$. Using the resolvent equation for $E_{1}$ and $AE_1 = E_1A$ on $D(A)$ we obtain
 \begin{align*}
  (1 \ast v)(t)x &= (E_1 \ast \rho)(t)x 
  \\ &= (1 \ast \rho)(t)x + (k \ast E_1Ax)\ast \rho(t)
  \\ &= (1 \ast \rho)(t)x + (k \ast (E_1 \ast \rho)Ax)(t)
  \\ &= (1 \ast \rho)(t)x + (k \ast (1\ast v)Ax)(t)
  \\ &= 1\ast \left( \rho + (k \ast vAx)\right)(t).
 \end{align*}
 Taking derivatives, we find that $v(t)x = \rho(t)x + (k \ast vAx)(t)$, i.e.~$v$ satisfies the $\rho$-resolvent equation. 
 \end{proof}

A particularly interesting choice for $\rho$ is given by $\rho(t) = e^{-\lambda t}t^{\beta - 1}/\Gamma(\beta)$ with $\beta > 0$ and $\lambda \geq 0$. In such a case the assumptions of Proposition \ref{prop: E rho from E 1} are satisfied for all $1 \leq p  \leq \infty$ when $\beta \geq 1$, and for all $1 \leq p < \frac{1}{1- \beta}$ when $\beta \in (\frac{1}{2},1)$.

\subsection{Linear Volterra equations}

Using the notions of $\rho$-resolvents, in this section, we collect a few results on the existence and uniqueness of solutions to linear Volterra equations as given in the next definition.
\begin{Definition}
 Let $g \in L_{loc}^1(\R_+; H)$. A function $u \in L_{\mathrm{loc}}^1(\R_+;H)$ is called
\begin{enumerate}
 \item[(a)] strong solution, if $u \in L_{loc}^1(\R_+; D(A))$ and for a.a. $t \geq 0$
 \begin{align}\label{eq: deterministic Volterra equation}
 u(t) = g(t) + \int_0^t k(t-s)Au(s)ds.
\end{align}
 \item[(b)] mild solution, if $k \ast u \in L_{loc}^1(\R_+; D(A))$ and for a.a. $t \geq 0$
 \begin{align*}
  u(t) = g(t) + A\int_0^t k(t-s)u(s)ds.
 \end{align*}
\end{enumerate}
\end{Definition}
Each strong solution is also a mild solution. For continuous $g$, a classical solution theory was extensively developed in \cite{MR1238939} while results on Volterra equations in $L^p$-spaces are discussed in Section 10 therein. In this section, we follow this terminology and consider the case $g \in L_{loc}^1(\R_+; H)$. Also here the existence and uniqueness of solutions of \eqref{eq: deterministic Volterra equation} is closely related to the notion of $\rho$-resolvents.
\begin{Lemma}\label{prop: mild solution characterization}
 Let $u \in L_{loc}^1(\R_+; H)$ and suppose that the $\rho$-resolvent $E_{\rho}$ exists for some $\rho \in L_{loc}^1(\R_+)$. Then the following hold:
 \begin{enumerate}
     \item[(i)] Each mild solution $u$ of \eqref{eq: deterministic Volterra equation} satisfies for a.a. $t \geq 0$ the generalized variation of constants formula
     \begin{align}\label{eq: generalized variation of constants}
      \int_0^t \rho(t-s)u(s)ds = \int_0^t E_{\rho}(t-s)g(s)ds.
     \end{align}
     
     \item[(ii)] Suppose that $u$ satisfies the generalized variation of constants formula \eqref{eq: generalized variation of constants}. If $k \ast u \in L_{loc}^1(\R_+; D(A))$, then $u$ is a mild solution of \eqref{eq: deterministic Volterra equation}. Moreover, if $\rho(A) \neq \emptyset$, then $k \ast u \in L_{loc}^1(\R_+; D(A))$ and hence $u$ is a mild solution of \eqref{eq: deterministic Volterra equation}.
 \end{enumerate} 
\end{Lemma}
\begin{proof}
 Suppose that $u$ is a mild solution. Then $k \ast u \in L_{loc}^1(\R_+; D(A))$ and using the resolvent equation we obtain
 \begin{align*}
     \rho \ast u &= \left(E_{\rho} - A(k \ast E_{\rho}) \right) \ast u
     \\ &= E_\rho \ast \left(u - A(k \ast u) \right)
     = E_{\rho} \ast g.
 \end{align*}
 
 Conversely, suppose that $u$ satisfies $k \ast u \in L_{loc}^1(\R_+; D(A))$ and \eqref{eq: generalized variation of constants}. Then $v := \rho \ast u$ satisfies $k \ast v = k \ast E_{\rho} \ast g \in L_{loc}^1(\R_+; D(A))$ by Proposition \ref{prop: regularity resolvent}, and hence 
 \begin{align*}
     \rho \ast u &= E_{\rho} \ast g 
     \\ &= \left( \rho + A(k\ast E_{\rho}) \right) \ast g
     \\ &= \rho \ast g + A (k \ast v)
     = \rho\ast \left( g + A (k \ast u) \right).
 \end{align*}
 Titchmarsh's theorem (see \cite[p. 166]{MR1336382}) implies that
 $u = g + A(k \ast u)$. It remains to prove that $k \ast u \in L_{loc}^1(\R_+; D(A))$ holds under $\rho(A) \neq \emptyset$. The latter follows along the lines of \cite[Theorem 2.7]{LIZAMA2000278} with the difference that all relations now hold $dt$-a.e.
\end{proof}
The relation \eqref{eq: generalized variation of constants} implies the uniqueness of mild solutions. Indeed, if $u,v$ are both mild solutions of \eqref{eq: deterministic Volterra equation}, then $\rho \ast u = E_{\rho} \ast g = \rho \ast v$ holds. Titchmarsh's theorem implies that $u = v$. Define
\[
 \mathcal{G}_1 = \{ g \in L_{loc}^1(\R_+; H) \ : \ E_k \ast g \in L_{loc}^1(\R_+; D(A)) \},
\] 
and let 
\begin{align}\label{eq: solution formula Ek}
 Gg(t) = g(t) + A \int_0^t E_k(t-s)g(s)ds, \ \ \text{ for a.a. } t \geq 0.
\end{align}
The next statement provides the existence of mild solutions when $g \in \mathcal{G}_1$.
\begin{Theorem}\label{thm: integral resolvent solution}
 Suppose that the $k$-resolvent exists and let $g \in \mathcal{G}_1$. Then $Gg$ is the unique mild solution of \eqref{eq: deterministic Volterra equation}. Conversely, suppose that $u$ is a mild solution of \eqref{eq: deterministic Volterra equation} with $g \in L_{loc}^1(\R_+; H)$. Then $g \in \mathcal{G}_1$ and $u = Gg$.
\end{Theorem}
\begin{proof}
  Let $g \in \mathcal{G}_1$, then $Gg$ is well-defined as an element in $L_{loc}^1(\R_+; H)$. By Proposition \ref{prop: regularity resolvent} we obtain $k \ast Gg = k \ast g + A(k \ast E_k) \ast g = E_k \ast g$. Hence $k \ast Gg \in L_{loc}^1(\R_+; D(A))$. By Lemma \ref{prop: mild solution characterization}.(b) we conclude that $Gg$ is a mild solution of \eqref{eq: deterministic Volterra equation}. Conversely, let $u \in L_{loc}^1(\R_+;H)$ be a mild solution. Lemma \ref{prop: mild solution characterization} applied for $\rho = k$ yields $E_k \ast g = k \ast u \in L_{loc}^1(\R_+; D(A))$ and hence $g \in \mathcal{G}_1$. Thus $Gg$ is another mild solution, and uniqueness yields $u = Gg$. This proves the assertion.
\end{proof}
In the same spirit, we can also consider subspaces of $\mathcal{G}_1$ that consist of $p$-locally integrable functions. The precise statement is summarized, without proof, in the following remark.
\begin{Remark}\label{remark: Gp}
  Define the spaces
\begin{align*}
  \mathcal{G}_p &= \{ g \in L_{loc}^p(\R_+; H) \ : \ E_k \ast g \in L_{loc}^p(\R_+; D(A)) \}, \qquad p \in [1,\infty),
  \\ \mathcal{G}_{\infty} &= \{ g \in C(\R_+; H) \ : \ E_k \ast g \in C(\R_+; D(A)) \}.
 \end{align*}
 Then $L_{loc}^p(\R_+; D(A)) \subset \mathcal{G}_p$, $C(\R_+; D(A)) \subset \mathcal{G}_{\infty}$, and $\mathcal{G}_p \subset \mathcal{G}_q$ for $q \leq p$. Finally, it holds that $G: \mathcal{G}_p \longrightarrow L_{loc}^p(\R_+; H)$ when $p \in [1,\infty)$, and $G: \mathcal{G}_{\infty} \longrightarrow C(\R_+;H)$ when $p = \infty$.
\end{Remark}
 Below we consider a class of examples for elements in $\mathcal{G}_1$ for which $Gg$ can be computed explicitly. The latter allows us to obtain the regularity of $Gg$ directly from the particular form of $g$.
 \begin{Example}\label{example: special choice}
  Let $\rho \in L_{loc}^1(\R_+)$ and suppose that the $k,\rho$-resolvents $E_k, E_{\rho}$ exist.
  \begin{enumerate}
    \item[(a)] Let $g(t) = \rho(t)x$ with $x \in H$. Then $E_k \ast \rho = E_{\rho} \ast k$, and hence, by Proposition \ref{prop: regularity resolvent}, $g \in \mathcal{G}_1$ with $Gg(t) = E_{\rho}(t)x$. 

  \item[(b)] Let $g(t) = \int_0^t \rho(t-s)g_0(s)ds$ with $g_0 \in L_{loc}^1(\R_+; H)$. Then using Proposition \ref{prop: regularity resolvent}, we have $E_k \ast g = (E_{\rho} \ast k)\ast g_0$ which belongs to $L_{loc}^1(\R_+;D(A))$. Hence $g \in \mathcal{G}_1$ with $Gg(t) = \int_0^t E_{\rho}(t-s)g_0(s)ds$.
  \end{enumerate}
\end{Example}
Since $\mathcal{G}_1$ and $G$ are both linear, the superposition principle holds from which other examples can be constructed. 

Similarly to the construction of $E_{\rho}$ from an analytic resolvent $E_1$ (see Proposition \ref{prop: E rho from E 1}), we can also obtain the mild solution of \eqref{eq: deterministic Volterra equation}. Namely, suppose first that $E_1$ exists and $u$ is a mild solution of \eqref{eq: deterministic Volterra equation}. Applying \eqref{eq: generalized variation of constants} for $\rho = 1$ gives $1 \ast u = E_1 \ast f$ and hence
$u(t) = \frac{d}{dt}\int_0^t E_1(t-s)f(s)ds$ holds for a.a. $t \geq 0$. For the existence of a mild solution of \eqref{eq: deterministic Volterra equation} we need additional regularity assumptions to guarantee that the convolution $E_1 \ast f$ is absolutely continuous. If $f \in W_{loc}^{1,1}(\R_+; H)$, then \eqref{eq: deterministic Volterra equation} has a mild solution, while $f \in W^{1,1}_{loc}(\R_+; D(A))$ or $f = k \ast g$ with $g \in W^{1,1}(\R_+; D(A))$
implies the existence of a strong solution, see \cite[Proposition 1.2]{MR1238939}. H\"older continuous driving forces $f$ have been considered in \cite[Corollary 2.6]{MR1238939}. In all of these cases, it is shown that the solution $u$ is actually continuous. Below we complement these results by considering the case where $f$ has some type of fractional regularity. 
\begin{Theorem}\label{thm: Sobolev solution formula}
 Suppose that there exists an analytic $1$-resolvent $E_1$, that $\rho(A) \neq \emptyset$, and let $f: \R_+ \longrightarrow H$ be measurable such that
 \[
  [f]_p^p := \int_0^T \|f(t)\|^p dt + \int_0^T \left(\int_0^t \frac{\|f(t)-f(s)\|}{|t-s|}ds \right)^{p} dt < \infty, \qquad \forall T > 0
 \]
 holds for some $p \in [2,\infty)$. Then: 
 \begin{enumerate}
   \item[(a)] Equation \eqref{eq: deterministic Volterra equation} has a unique mild solution $u \in L_{loc}^p(\R_+;H)$ which is given by 
   \begin{align}\label{eq: resolvent solution formula}
    u(t) = E_1(t)f(t) + \int_0^t E_1'(t-s)(f(s) - f(t))ds, \qquad t \geq 0.
   \end{align}
   Moreover, for each $T > 0$, there exists a constant $C(p,T,S_A,S_A') > 0$ such that 
   \[
    \|u\|_{L^{p}([0,T]; H)} \leq C(p,T, S_A, S_A') [f]_p;
   \]
  
   \item[(b)] If $f = k \ast g$ with $[g]_{p,T} < \infty$ for each $T > 0$ and some $p \in [2,\infty)$, then $u$ given by \eqref{eq: resolvent solution formula} is a strong solution of \eqref{eq: deterministic Volterra equation}. 
 \end{enumerate}
\end{Theorem}
\begin{proof}
 Following the steps in the proof of Proposition \ref{prop: E rho from E 1} we find that $u$ is well defined and satisfies the desired $L^p$-bound. Using the same approximation as in the second part of the proof of Proposition \ref{prop: E rho from E 1} we obtain $1 \ast u = E_1 \ast f$. Hence, by Lemma \ref{prop: mild solution characterization}, $u$ is a mild solution which proves assertion (a). Let us prove (b). Let $v$ be given by \eqref{eq: resolvent solution formula} with $f$ replaced by $g$. Part (b) shows that $v$ is a mild solution of \eqref{eq: deterministic Volterra equation}.
 Hence $k \ast v \in L_{\mathrm{loc}}^1(\R_+; D(A))$ satisfies
 $k \ast v = k \ast \left( g + A(k \ast v) \right) = k \ast g + k \ast A(k\ast v)$, i.e.~$k \ast v$ is a strong solution of \eqref{eq: deterministic Volterra equation}. Since $u$ is also a mild solution, uniqueness gives $u = k \ast v$.
\end{proof}
If $k$ satisfies $\int_0^{\infty}|k(t)|e^{\lambda t}dt < \infty$ for some $\lambda \in \R$, and $E_1$ exists, then $\rho(A) \neq \emptyset$ is always satisfied due to \cite[Chapter 2, Theorem 2.1]{MR1238939}.

\subsection{Self-adjoint and diagonalisable $(A,D(A))$}
\label{sec:fractional_kernels}

In this section we shall assume that $(A,D(A))$ is self-adjoint on $H$ and has an orthonormal basis $(e_n^H)_{n \geq 1}$ such that $Ae_n^H = -\mu_ne_n^H$, $n \geq 1$. Below we adapt the same notation as in Section \ref{sec: examples}.
In the next lemma, we study the case where $k,h$ are fractional kernels in which case sharp bounds can be obtained due to Lemma \ref{lemma: Mittag-Leffler function}.

\begin{Lemma}\label{lemma: fractional kernels}
    Let $k(t) = t^{\alpha-1}/\Gamma(\alpha)$ and $h(t) = t^{\beta-1}/\Gamma(\beta)$ with $\alpha \in (0,2)$ and $\beta > 0$. Take $q \in [1,\infty)$ and $\rho, \lambda \in \R$, then the following assertions hold:
    \begin{enumerate}
        \item[(a)] If $1 - \frac{1}{q} < \beta < \alpha + 1 - \frac{1}{q}$, then 
        \[
         \int_0^{\infty} \|E_h(t)\|_{L(H^{\lambda}, H^{\rho})}^q dt
         \leq c_q(\alpha,\beta) \sum_{n=1}^{\infty} \mu_n^{-q(\lambda + \frac{\beta}{\alpha}) + q \rho + \frac{q-1}{\alpha}},
        \]
        where the constant $c_q(\alpha, \beta)$ is given in \eqref{eq:c_q}. The same bound also holds for $\alpha = \beta$ and $q \in [1,\infty)$ provided that $1 < \alpha + \frac{1}{q}$.
        
        \item[(b)] If $\frac{1}{2} < \beta \leq \alpha + \frac{1}{2}$, then 
        \[
         \int_0^{\infty} \|E_h(t)\|_{L_2(H^{\lambda}, H^{\rho})}^2dt = c_2(\alpha,\beta) \sum_{n=1}^{\infty} \mu_n^{-2(\lambda + \frac{\beta}{\alpha}) + 2 \rho + \frac{1}{\alpha}}.
        \]
        \item[(c)] If $\alpha \in (0,1]$, $\beta \geq \alpha$, and $1 < \beta + \frac{1}{q}$, then
        \[
        \int_0^{\infty} \|E_h(t)\|_{L(H)}^q dt = c_q(\alpha,\beta) \mu_1^{- \frac{\beta q}{\alpha} + \frac{q-1}{\alpha}}.
        \]
    \end{enumerate}
\end{Lemma}

\begin{proof}
    Note that $e_h(t;\mu_n) = t^{\beta - 1}E_{\alpha,\beta}(-\mu_n t^{\beta})$. For case (a) we bound 
    \begin{align}\label{eq: 7}
     \|E_h(t)\|_{L(H^{\lambda}, H^{\rho})}^q = \sup_{n \geq 1}\mu_n^{q(\rho - \lambda)} |e_h(t;\mu_n)|^q \leq \sum_{n=1}^{\infty}\mu_n^{q(\rho - \lambda)} |e_h(t;\mu_n)|^q.
    \end{align}
    Hence the desired bound follows from 
    \begin{align}\label{eq:131123-1}
        \int_0^{\infty}|e_h(t;\mu_n)|^q \, dt = c_q(\alpha,\beta)\mu_n^{-q\beta/\alpha + (q-1)/\alpha},
    \end{align}
    see Lemma \ref{lemma: Mittag-Leffler function}. For case (b) we use $\displaystyle \|E_h(t)\|_{L_2(H^{\lambda}, H^{\rho})}^2 = \sum_{n=1}^{\infty} \mu_n^{2(\rho - \lambda)}|e_h(t;\mu_n)|^2$ and proceed as above. 
    Finally, in case (c) we note that $e_h$ is completely monotone due to Schneider \cite{MR1382012}. Hence by Lemma \ref{lemma: monotonicity} we obtain $\displaystyle \|E_h(t)\|_{L(H)} = \sup_{n \geq 1}e_h(t;\mu_n) = e_h(t;\mu_1)$ and the assertion follows from \eqref{eq:131123-1}.
\end{proof}

\begin{Lemma}\label{lemma: Ek self adjoint}
 Let $\rho,\lambda \in \R$ with $\rho \geq \lambda$. Then the following assertions hold:
 \begin{enumerate}
     \item[(a)] Let $k \in L_{loc}^1(\R_+) \cap C^1((0,\infty))$ be nonincreasing, $k > 0$ on $(0,\infty)$, $\int_0^{\infty}e^{-\e t}k(t)dt < \infty$ for each $\e > 0$, and assume that $\ln(k)$ and $\ln(-k')$ are convex on $(0,\infty)$. Then
     \[
      \int_0^{\infty} \|E_k(t)\|_{L(H)}dt \leq \frac{1}{\mu_1}
     \]
     and 
     \[
      \int_0^{\infty} \|E_k(t)\|_{L(H^{\lambda},H^{\rho})}dt \leq \sum_{n=1}^{\infty}\mu_n^{-(1+\lambda) + \rho}.
     \]

     \item[(b)] Let $k \in L_{loc}^1(\R_+)$ be given as in part (a), and assume that there exists $\delta \in (0,1)$ and $C_{\delta} > 0$ such that $k(t) \leq C_{\delta}t^{-\delta}$ for $t \in (0,1]$. Then, for for $q \in (1,1/\delta)$,
     \[
      \int_0^{\infty}\|E_k(t)\|_{L(H)}^q dt \leq \max\left\{ \frac{C_{\delta}^q}{1-q\delta}, C_{\delta}^{q-1} \right\} (1\vee \mu_1)^{-1 + \frac{(q-1)\delta}{1-\delta}}
     \]
     and
     \[
      \int_0^{\infty} \|E_k(t)\|_{L(H^{\lambda},H^{\rho})}^q dt \leq \max\left\{\frac{C_{\delta}^q}{1-q\delta}, C_{\delta}^{q-1}\right\} \sum_{n=1}^{\infty}(1 \vee \mu_n)^{-q(1+\lambda)  + q\rho + \frac{(q-1)\delta}{1-\delta}}.
     \]
     \end{enumerate}
\end{Lemma}
\begin{proof}
    Remark \ref{remark: sufficient condition for ek} implies that $e_k(t;\mu_1) \geq 0$ is non-increasing. Thus Lemma \ref{lemma: monotonicity} applied to $h = k$ yields $\|E_k(t)\|_{L(H)} = \sup_{n \geq 1}|e_k(t;\mu_n)| = e_k(t;\mu_1)$ and using Lemma \ref{lemma: general ek integral} we arrive at
    \[
     \int_0^{\infty}\|E_k(t)\|_{L(H)} dt = \int_0^{\infty}e_k(t;\mu_1)dt \leq \mu_1^{-1}.
    \]
    The second inequality follows from \eqref{eq: 7} combined with $\int_{0}^{\infty}e_k(t;\mu_n)dt \leq \mu_n^{-1}$ due to Lemma \ref{lemma: general ek integral}. This proves (a). For assertion (b) we use Lemma \ref{lemma: monotonicity} applied to $h = k$, to obtain
    \begin{align*}
        \int_0^{\infty} \|E_k(t)\|_{L(H)}^q dt
         &= \int_0^{\infty} e_k(t;\mu_1)^q dt
        \\ &\leq \max\left\{\frac{C_{\delta}^q}{1-q\delta}, C_{\delta}^{q-1}\right\} (1\vee \mu_1)^{- \frac{1 - q\delta}{1-\delta}}
    \end{align*}
    where the last inequality follows from Lemma \ref{lemma: Lq norm eh}. For the second inequality we first apply \eqref{eq: 7} and then Lemma \ref{lemma: Lq norm eh} to find that
    \begin{align*}
        \int_0^{\infty}\|E_k(t)\|_{L(H^{\lambda}, H^{\rho})}^q dt 
        &\leq \sum_{n=1}^{\infty} \mu_n^{q(\rho - \lambda)} \int_0^{\infty}|e_k(t;\mu_n)|^q dt
        \\ &\leq \nu(\R_+)^q \max\left\{\frac{C_{\delta}^q}{1-q\delta}, C_{\delta}^{q-1}\right\} \sum_{n=1}^{\infty}\mu_n^{q(\rho - \lambda)}(1\vee\mu_n)^{- \frac{1 - q\delta}{1-\delta}}.
    \end{align*} 
\end{proof}

This lemma can also be used to obtain bounds on $E_h$, provided that $h$ 
\[
 h(t) = \int_{[0,t]}k(t-s)\nu(ds)
\]
where $\nu$ is a finite measure on $\R_+$. In such a case $h \in L_{loc}^p(\R_+)$ whenever $k \in L_{loc}^p(\R_+)$. Moreover, it is easy to verify that $e_h(t;\mu_n) = (e_k(\cdot;\mu_n) \ast \nu)(t)$ and hence $\|e_h(\cdot;\mu_n)\|_{L^q(\R_+)} \leq \nu(\R_+)\|e_k(\cdot;\mu_n)\|_{L^q(\R_+)}$. Consequently, we obtain for all $\lambda,\rho \in \R$
\begin{align}\label{eq: 8}
 \|E_h(t)\|_{L(H^{\lambda}, H^{\rho})} \leq \nu(\R_+)\|E_k(t)\|_{L(H^{\lambda}, H^{\rho})}.
\end{align}
 The same argument shows that under condition (b) from previous lemma, it holds that
\begin{align*}
     \int_0^{\infty} \|E_h(t)\|_{L_2(H^{\lambda}, H^{\rho})}^2 dt 
     &\leq \nu(\R_+)^2\sum_{n=1}^{\infty}\mu_n^{2(\rho - \lambda)}\int_0^{\infty}|e_k(t;\mu_n)|^2dt
     \\ &\leq \nu(\R_+)^2 \max\left\{\frac{C_{\delta}^2}{1-2\delta}, C_{\delta}\right\}\sum_{n=1}^{\infty}(1\vee \mu_n)^{- 2(1+\lambda) + 2\rho + \frac{\delta}{1-\delta}}.
\end{align*}
When $k(t) = t^{\alpha-1}/\Gamma(\alpha)$ with $\alpha \in (0,2)$, then we even obtain 
\begin{align*}
 \int_0^{\infty} \|E_h(t)\|_{L_2(H^{\lambda}, H^{\rho})}^2 dt &\leq \nu(\R_+) \sum_{n=1}^{\infty}\mu_n^{2(\rho - \lambda)} \int_0^{\infty} |e_k(t;\mu_n)|^2 dt
 \\ &\leq \nu(\R_+)^2 c_2(\alpha,\alpha) \sum_{n=1}^{\infty} \mu_n^{-2(1+\lambda) + 2 \rho + \frac{1}{\alpha}}
\end{align*}
where we have used $e_k(t;\mu_n) = t^{\alpha-1}E_{\alpha,\alpha}(- \mu_n t^{\alpha})$ combined with Lemma \ref{lemma: Mittag-Leffler function}.

Finally, let us briefly apply these estimates to obtain admissible classes of functions $g$ such that $Gg$ exists and satisfies \eqref{eq: xi limit}.

\begin{Proposition}\label{prop: Gg example general}
    Let $k \in L_{loc}^1(\R_+) \cap C^1((0,\infty))$ be nonincreasing, $k > 0$ on $(0,\infty)$, $\int_0^{\infty}e^{-\e t}k(t)dt < \infty$ for each $\e > 0$, and assume that $\ln(k)$ and $\ln(-k')$ are convex on $(0,\infty)$. Let $\rho \in \R$. Then the following hold:
    \begin{enumerate}
        \item[(a)] If $g(t) = h(t)x$ with $x \in H^{\rho}$ and $h = k \ast \rho \in C_b(\R_+)$ where $\rho$ is a locally finite measure on $\R_+$, then $Gg(t) \in C_b(\R_+; H^{\rho})$. If $e_h(t;\mu_n) \longrightarrow 0$ as $t \to \infty$ for each $n \geq 1$, then also $Gg(t) \longrightarrow 0$ as $t \to \infty$.

        \item[(b)] If $g(t) = k \ast g_0(t)$ with $g_0 \in L^{\infty}(\R_+; H^{\rho-1})$, then $Gg \in C_b(\R_+;H^{\rho})$. If additionally $g_0(t) \longrightarrow g_0(\infty)$ in $H^{\rho-1}$, then 
        \[
         Gg(t) \longrightarrow Gg(\infty) = \left(\|k\|_{L^1(\R_+)}^{-1} - A \right)^{-1}g_0(\infty).
        \]
    \end{enumerate}
\end{Proposition}
\begin{proof}
 (a) Write $x = \sum_{n=1}^{\infty}x_n \mu_n^{-\rho}e_n^H$ where $\sum_{n=1}^{\infty}|x_n|^2 < \infty$. Then 
 \begin{align*}
     Gg(t) = E_{h}(t)x = \sum_{n=1}^{\infty} e_{h}(t;\mu_n)x_n \mu_n^{-\rho} e_n^H. 
 \end{align*}
 Since $e_k(\cdot;\mu_n) \geq 0$
 by Remark \ref{remark: sufficient condition for ek} and $e_h(t;\mu_n) = e_k(\cdot;\mu_n) \ast \rho(t)$, we obtain $0 \leq e_h(\cdot; \mu_n) \leq h(t)$. Since $h \in C_b(\R_+)$, also $e_h(\cdot; \mu_n) \in C_b(\R_+)$ and hence $Gg \in C(\R_+; H^{\rho})$. Finally, since $e_h(t;\mu_n) \longrightarrow 0$ for each $n \geq 1$, we obtain $Gg(t) \longrightarrow 0$ by dominated convergence. 

 (b) Write $g_0 = \sum_{n=1}^{\infty} g_0^{(n)} \mu_n^{1-\rho}e_n^H$ where $\sum_{n=1}^{\infty}\|g_0^{(n)}\|_{L^{\infty}(\R_+)}^2 < \infty$. Then $Gg(t) = \sum_{n=1}^{\infty} e_k(\cdot;\mu_n) \ast g_0^{(n)}(t) \mu_n^{1-\rho}e_n^H$. Since
 \[
  \|e_k(\cdot;\mu_n)\ast g_0^{(n)}\|_{L^{\infty}(\R_+)} \leq \|e_k(\cdot;\mu_n)\|_{L^1(\R_+)}\|g_0^{(n)}\|_{L^{\infty}(\R_+)} \leq \mu_n^{-1}\|g_0^{(n)}\|_{L^{\infty}(\R_+)},
 \]
 each term in the sum belongs to $C_b(\R_+)$ from which we deduce that $Gg \in C_b(\R_+; H^{\rho})$. The limit satisfies
 \begin{align*}
 &\ \left| \int_0^{t}e_k(s;\mu_n)g_0^{(n)}(t-s)ds - g_0^{(n)}(\infty) \int_0^{\infty}e_k(s;\mu_n)ds \right|
 \\ &\leq \int_0^{t/2} e_k(s;\mu_n)|g_0^{(n)}(t-s) - g_0^{(n)}(\infty)|ds
 \\ &\qquad + \|g_0^{(n)}\|_{L^{\infty}(\R_+)}\int_{t/2}^t e_k(s;\mu_n)ds + |g_0^{(n)}(\infty)| \int_{t/2}^{\infty}e_k(s;\mu_n)ds
 \\ &\leq \mu_n^{-1}\sup_{s \geq t/2}|g_0^{(n)}(s) - g_0^{(n)}(\infty)| + \left( \|g_0^{(n)}\|_{L^{\infty}(\R_+)} + |g_0^{(n)}(\infty)| \right) \int_{t/2}^{\infty}e_k(s;\mu_n)ds.
 \end{align*}
 The right-hand side converges for each $n \geq 1$ to zero when $t \to \infty$. By dominated convergence, we see that $Gg(t) \longrightarrow Gg(\infty)$ in $H^{\rho}$ where $Gg(\infty)$ has components
 \[
  g_0^{(n)}(\infty) \int_0^{\infty}e_k(s;\mu_n)ds =
  \frac{g_0^{(n)}(\infty)}{ \|k\|^{-1}_{L^1(\R_+)} + \mu_n}, \qquad n \geq 1.
 \]
\end{proof}

Note that the above statement (a) applied for the case $\rho = \overline{k}$ being the resolvent of the first kind which gives $e_h(\cdot;\mu_n) = e_1(\cdot;\mu_n) \in C_b(\R_+)$ and to $\rho = \delta_0$ which gives $e_h(\cdot;\mu_n) = e_k(\cdot;\mu_n)$.
When $k$ is the fractional kernel, stronger results can be obtained. 

\begin{Proposition}\label{prop: Gg example}
 Suppose that $k(t) = \frac{t^{\alpha-1}}{\Gamma(\alpha)}$ with $\alpha \in (0,2)$ and define 
 \[
  g(t) = \int_0^t \frac{(t-s)^{\gamma-1}}{\Gamma(\gamma)}g_0(s)ds
 \]
 where $\gamma > 0$. Then the following assertions hold:
 \begin{enumerate}
     \item[(a)] Let $p,q, p' \in [1,\infty]$ satisfy $\frac{1}{p} + \frac{1}{q} = 1 + \frac{1}{p'}$ and
      \[
        1 - \frac{1}{q} < \gamma < \alpha + 1 - \frac{1}{q}. 
     \]
     Suppose that $g_0 \in L^p(\R_+;H^{\rho - \varkappa})$ for some $\rho \in \R$ where 
     \[
     \varkappa = \frac{\gamma q}{\alpha} - \frac{q-1}{\alpha}.
     \]
     Then $Gg \in L^{p'}(\R_+; H^{\rho})$, and if $p' = \infty$ then even $Gg \in C_b(\R_+; H^{\rho})$.

     \item[(b)] Suppose that $0 < \gamma \leq \alpha$ and $g_0 \in L^{\infty}(\R_+; H^{\rho - \gamma/\alpha})$ for some $\rho \in \R$. Then $Gg \in C_b(\R_+; H^{\rho})$. If additionally $g_0(t) \longrightarrow g_0(\infty)$ in $H^{\rho - \gamma/\alpha}$, then 
     \[
      Gg(t) \longrightarrow \begin{cases} 0, & 0 < \gamma < \alpha 
      \\ (-A)^{-1}g_0(\infty), & \gamma = \alpha.
      \end{cases}
     \]
 \end{enumerate}
\end{Proposition}
\begin{proof}
    (a) Write $g_0 = \sum_{n=1}^{\infty} g_0^{(n)} \mu_n^{\varkappa - \rho}e_n^H$ with $\sum_{n=1}^{\infty}\|g_0^{(n)}\|_{L^p(\R_+)}^2 < \infty$. Then $Gg = \sum_{n=1}^{\infty} \mu_n^{\varkappa - \rho}f_{\alpha,\gamma}^{(n)} \ast g_0^{(n)} e_n^H$ where $f_{\alpha,\gamma}^{(n)}(t) = t^{\gamma-1} E_{\alpha,\gamma}(-\mu_n t^{\alpha})$. Each term in the sum satisfies by Young's inequality combined with Lemma \ref{lemma: Mittag-Leffler function}
    \[
     \|f_{\alpha,\gamma}^{(n)} \ast g_0^{(n)} \|_{L^{p'}(\R_+)}
     \leq \|f_{\alpha,\gamma}^{(n)}\|_{L^q(\R_+)} \|g_0^{(n)}\|_{L^p(\R_+)}
     = c_q(\alpha, \gamma) \mu_n^{-\frac{\gamma q}{\alpha} + \frac{q-1}{\alpha}} \|g_0^{(n)}\|_{L^p(\R_+)}.
    \]
    Hence $f_{\alpha,\gamma}^{(n)} \ast g_0^{(n)} \in L^{p'}(\R_+)$ for each $n \geq 1$ from which we deduce $Gg \in L^{p'}(\R_+; H^{\rho})$. If $p' = \infty$, then Young's inequality implies that $Gg \in C_b(\R_+; H^{\rho})$.

    (b) If $\gamma \in (0, \alpha)$, then $Gg \in C_b(\R_+; H^{\rho})$ follows from (a) applied to $q = 1$ and $p = p' = \infty$, while for $\gamma = \alpha$ we obtain $f_{\alpha,\gamma}^{(n)} \in L^1(\R_+)$ due to Lemma \ref{lemma: Mittag-Leffler function} and hence $\|f_{\alpha,\gamma}^{(n)} \ast g_0^{(n)}\|_{L^{\infty}(\R^d)} \leq \|f_{\alpha,\gamma}^{(n)}\|_{L^1(\R^d)} \|g_0^{(n)}\|_{L^{\infty}(\R^d)} \leq c_1(\alpha,\gamma)\mu_n^{-1}\|g_0^{(n)}\|_{L^{\infty}(\R^d)}$ and hence $Gg \in C_b(\R^d; H^{\rho})$. In both cases, this implies
    \begin{align*}
     f_{\alpha,\gamma}^{(n)} \ast g_0^{(n)}(t) \longrightarrow g_0^{(n)}(\infty)\int_0^{\infty} f_{\alpha, \gamma}^{(n)}(s)ds 
     = g_0^{(n)}(\infty)\begin{cases} \mu_n^{-1}, & \gamma = \alpha \\ 0, & \gamma < \alpha. \end{cases}
    \end{align*}
    due to the explicit form of the Laplace transform of $f_{\alpha,\gamma}^{(n)}$. 
\end{proof}

\section{Strong and mild solutions}
\label{sec:strong_mild}

Let $(\Omega, \F, (\F_t)_{t \geq 0}, \P)$ be a stochastic basis with the usual conditions, $(U, \langle \cdot, \cdot \rangle_U, \| \cdot \|_U)$ and $(H, \langle \cdot, \cdot \rangle_H, \| \cdot\|_H)$ be separable Hilbert spaces, and let $W$ be an $(\F_t)_{t \geq 0}$-cylindrical Wiener process over $U$. Note that, for $x \in U$, $W_t(x) = \sum_{n=1}^{\infty}\langle x, e_n^U\rangle_U e_n^U \beta_n(t)$ where $(\beta_n(t))_{t \geq 0}$ is a sequence of independent standard Brownian motions and $(e_j^U)_{j \geq 1}$ denotes an orthonormal basis of $U$. Denote by $L_2(U,H)$ be the separable space of Hilbert-Schmidt operators from $U$ to $H$ equipped with the norm $\|T\|_{L_2(U,H)}^2 = \sum_{j=1}^{\infty}\|Te_j\|_H^2$. The stochastic integral 
\[
 t \longmapsto \int_0^t \varphi(s)dW_s := \sum_{n=1}^{\infty} \int_0^t \varphi(s)e^U_n\ d\beta_n(s)
\]
defines a continuous $H$-valued martingale, whenever $\varphi: [0,T] \times \Omega \longrightarrow L_2(U, H)$ is measurable, $(\F_t)_{t \geq 0}$-adapted, and satisfies $\int_0^T \E[ \|\varphi(s) \|_{L_2(U,H)}^2] dt < \infty$. We refer to \cite{MR3236753, MR2560625} for additional details on the stochastic integration of Gaussian processes.

\subsection{Stochastic convolutions}

In this section, we study the stochastic convolution of $W$ with a general family of operators $(E(t))_{t \in [0,T]} \subset L(H)$ given by
\[
 (E\ast \varphi dW)_t := \int_0^t E(t-s)\varphi(s)dW_s, \qquad t \in [0,T].
\]
While, for each fixed $t \in [0, T]$, the integral can be defined in the It\^o sense, we are particularly interested in the properties of this convolution seen as a stochastic process in $t$.

Let $V$ be another separable Hilbert space. We mainly use $V = \R$, $H$, or $L_2(U,H)$. 
For $q \in [1,\infty]$, $p \in [1,\infty)$ and $T > 0$, we let $\mathcal{H}^{q,p}([0,T]; V)$ be the Banach space of all measurable functions $\varphi: [0,T] \times \Omega \longrightarrow V$ that are $(\F_t)_{t \in [0,T]}$-adapted, and have finite norm 
\begin{align}\label{Hqp norm}
 \|\varphi\|_{\mathcal{H}^{q,p}([0,T];V)} = \left( \int_0^T \left( \E\left[ \|\varphi(t)\|_V^p \right]\right)^{\frac{q}{p}}dt \right)^{\frac{1}{q}}
\end{align}
when $q < \infty$, and with obvious adaptions for $q = \infty$. Similarly, we let $\mathcal{H}^{q,p}(\R_+;V)$ be the space of $(\F_t)_{t \geq 0}$-adapted and measurable processes $\varphi: \R_+ \times \Omega \longrightarrow V$ equipped with the norm $\|\varphi\|_{\mathcal{H}^{q,p}(\R_+;V)}$ given as in \eqref{Hqp norm} with $T = \infty$ and $q < \infty$, and with the obvious adaption for $q = \infty$. Finally, we let $\mathcal{H}^{q,p}_{loc}(\R_+;V)$ be the space of of all measurable functions $\varphi: \R_+ \times \Omega \longrightarrow V$ that are $(\F_t)_{t \in [0,T]}$-adapted, and satisfy $\| \varphi\|_{\mathcal{H}^{q,p}([0,T]; V)} < \infty$ for each $T > 0$. Let $c_2 = 1$ and 
\begin{align}\label{eq: Cp constant}
 c_p = \left(\frac{p(p-1)}{2}\right)^p \left( \frac{p}{p-1} \right)^{\frac{p^2}{2}}, \qquad p \in (2,\infty).
\end{align}
The following is our key estimate to bound the stochastic convolution. 
\begin{Theorem}\label{thm: stochastic convolution cylindrical case}
 Let $(\varphi(t))_{t \in [0,T]} \subset L(U,H)$ be $(\F_t)_{t \in [0,T]}$-adapted such that $[0,T] \ni t \longmapsto \varphi(t)x \in H$ is, for each $x \in U$, jointly measurable in $(\omega,t)$. Let $(V,\|\cdot\|_V)$ be another seperable Hilbert space and let $[0,T] \ni t \longmapsto E(t) \in L(H,V)$ be strongly measurable such that $E(t-r)\varphi(r) \in L_2(U,V)$ holds $\P$-a.s. for a.a. $r \in [0,t)$ and $t \in [0,T]$. Suppose that, for $p \in [2,\infty)$ and $q \in [p,\infty]$ one has
 \begin{align}\label{eq: varphi admissible}
  \left\| \left(\int_0^{\cdot} \| E(\cdot - s)\varphi(s)\|_{L_2(U,V)}^2 ds \right)^{\frac{1}{2}} \right\|_{\mathcal{H}^{q,p}([0,T]; \R)} < \infty.
 \end{align}
 Then the stochastic convolution belongs to $\mathcal{H}^{q,p}([0,T];V)$, satisfies for a.a. $t \in [0,T]$
 \begin{align}\label{eq: stochastic convolution basic estimate}
  \E\left[ \left\| \int_0^t E(t-s)\varphi(s)dW_s \right\|_V^p \right] \leq c_p \E\left[ \left( \int_0^t \|E(t-s)\varphi(s)\|_{L_2(U,V)}^2 ds \right)^{\frac{p}{2}}\right],
 \end{align}
 and it holds that 
 \begin{align*}
  \left\| \int_0^{\cdot} E(\cdot-s)\varphi(s)dW_s \right\|_{\mathcal{H}^{q,p}([0,T];V)} 
  \leq c_p^{1/p} \left\| \left(\int_0^{\cdot} \| E(\cdot - s)\varphi(s)\|_{L_2(U,V)}^2 ds \right)^{\frac{1}{2}} \right\|_{\mathcal{H}^{q,p}([0,T]; \R)}.
 \end{align*}
 Moreover, if
 \begin{align*}
  &\ \lim_{t-s\to 0} \E\left[ \left( \int_0^s \|(E(t-r) - E(s-r))\varphi(r)\|_{L_2(U,V)}^2 dr \right)^{\frac{p}{2}}\right] = 0,
   \\ &\ \ \ \lim_{t-s\to 0} \E\left[ \left( \int_s^t \|E(t-r)\varphi(r)\|_{L_2(U,V)}^2 dr \right)^{\frac{p}{2}} \right] = 0,
 \end{align*}
 then $t \longmapsto \int_0^t E(t-s)\varphi(s)dW_s \in L^{p}(\Omega, \F, \P;V)$ is continuous. If there exists $\gamma \in \left( \frac{1}{p},1\right]$ and $C_{T,p} > 0$ with
 \begin{align*}
  &\ \E\left[ \left( \int_0^s \|(E(t-r) - E(s-r))\varphi(r)\|_{L_2(U,V)}^2 dr \right)^{\frac{p}{2}}\right] 
  \\ &\qquad + \E\left[ \left( \int_s^t \|E(t-r)\varphi(r)\|_{L_2(U,V)}^2 dr \right)^{\frac{p}{2}} \right] 
  \leq C_{T,p} (t-s)^{p\gamma}
 \end{align*}
 for $s,t \in [0,T]$ with $s \leq t$, then for each $\alpha \in \left(0, \gamma - \frac{1}{p}\right)$, $t \longmapsto \int_0^t E(t-s)\varphi(s)dW_s$ has a modification with $\alpha$-H\"older continuous sample paths.
\end{Theorem}
\begin{proof}
 First note that $\|E(t-r)\varphi(r)\|_{L_2(U,V)}$ is jointly measurable in $(r,t,\omega)$ due to $\|E(t-r)\varphi(r)\|_{L_2(U,V)}^2 = \sum_{n=1}^{\infty}\|E(t-r)\varphi(r)e^U_n\|_V^2$ and the strong measurability of $E$, where $(e^U_n)_{n \geq 1}$ is an orthonormal basis of $U$. Define for $s,t \in [0,T]$ the two-parameter process $\Phi_t(s) = \1_{\{s \leq t\}}E(t-s)\varphi(s)$. Then $\Phi_t$ is jointly measurable and $(\F_s)_{s \in [0,T]}$-adapted for fixed $t \in [0,T]$. Moreover, it holds that
 \begin{align*}
   &\ \int_0^T \left( \E\left[ \left(\int_0^T \| \Phi_t(s) \|_{L_2(U,V)}^2 ds  \right)^{\frac{p}{2}}\right] \right)^{\frac{1}{p}}dt
     \\ &= \int_0^T \left( \E\left[ \left(\int_0^t \| E(t-s)\varphi(s) \|_{L_2(U,V)}^2 ds  \right)^{\frac{p}{2}}\right] \right)^{\frac{1}{p}}dt < \infty
 \end{align*}
 due to \eqref{eq: varphi admissible}. Hence, by stochastic Fubini theorem (see \cite{MR1346469}), the two-parameter process 
 $X_s^t = \int_0^s \Phi_t(r)dW_r$ is well-defined as an It\^o integral and has an $\mathcal{O} \otimes \mathcal{B}(\R_+)$-measurable version w.r.t. the variables $(\omega, s, t) \in \Omega \times [0, T] \times [0, T]$, where $\mathcal{O}$-denotes the optional $\sigma$-algebra on $\Omega \times [0,T]$. Using this version, we let $\int_0^t E(t-s)\varphi(r)dW_r := X_t^t$ which is then jointly measurable in $(\omega, t)$ and also $(\F_t)_{t \in [0,T]}$-adapted.
 
 Estimate \eqref{eq: stochastic convolution basic estimate} is an immediate consequence of the It\^o isometry (whence $c_2 = 1$). For $p > 2$, it follows from the estimates given in the proof of \cite[Lemma 3.1]{MR2560625} when applied to $[0,t] \ni s \longmapsto \int_0^s E(t-r)\varphi(r)dW_r$. Taking the $L^q([0,T])$ norm in \eqref{eq: stochastic convolution basic estimate} readily yields the desired estimates in $\mathcal{H}^{q,p}([0,T]; V)$. To prove the continuity, let $0 \leq s < t \leq T$. Then by similar bounds, we obtain
 \begin{align*}
  \E\left[ \left\| X_t^t - X_s^s \right\|_V^{p} \right] &\leq 2^{p-1} \E\left[ \left\| \int_0^s (E(t-r) - E(s-r))\varphi(r)dW_r \right\|_V^{p}\right] 
  \\ &\ \ \ + 2^{p-1} \E\left[ \left\| \int_s^t E(t-r)\varphi(r)dW_r \right\|_V^{p} \right]
  \\ &\leq 2^{p-1}c_p \E\left[ \left( \int_0^s \|(E(t-r) - E(s-r))\varphi(r)\|_{L_2(U,V)}^2 dr \right)^{\frac{p}{2}}\right] 
  \\ &\ \ \ + 2^{p-1}c_p \E\left[ \left( \int_s^t \|E(t-r)\varphi(r)\|_{L_2(U,V)}^2 dr \right)^{\frac{p}{2}} \right].
 \end{align*}
 This proves the continuity in the $p$-th sense. The H\"older continuity follows from the Kolmogorov-Chentsov theorem.
\end{proof}

Below we consider the special case where $\varphi$ takes values in $L_2(U, H)$. 

\begin{Corollary}\label{thm: stochastic convolution}
 Let $p \in [2,\infty)$, $q \in [p,\infty]$, and $a,b \in [1,\infty]$ be such that
 \begin{align*}
  \frac{1}{2} + \frac{1}{q} = \frac{1}{a} + \frac{1}{b}
 \end{align*}
 with the usual convention $\frac{1}{\infty} = 0$. Let $(V,\|\cdot\|_V)$ be a separable Hilbert space. If $E \in L^a([0,T]; L_s(H,V))$, and $\varphi$ is $(\F_t)_{t \in [0,T]}$-adapted such that 
 \[
  \varphi \in L^p(\Omega, \F, \P; L^b([0,T]; L_2(U, H))),
 \]
 then $\int_0^{\cdot}E(\cdot-s)\varphi(s)dW_s \in \mathcal{H}^{q,p}([0,T]; V)$ and there exists a constant $C_{p,q,a,b} > 0$ such that
 \begin{align*}
  \left\| \int_0^{\cdot}E(\cdot-s)\varphi(s)dW_s \right\|_{\mathcal{H}^{q,p}([0,T];V)} 
  \leq C_{p,q,a,b} \| E \|_{L^a([0,T]; L(H,V))} \cdot \| \varphi \|_{L^p(\Omega; L^b([0,T]; L_2(U,H)))}.
 \end{align*}
\end{Corollary}
\begin{proof}
 If $p \leq q < \infty$, then using the Minkowski inequality and then Young's convolution inequality, we find that
 \begin{align*}
   &\ \left\| \left(\int_0^{\cdot} \| E(\cdot - s)\varphi(s)\|_{L_2(U,V)}^2 ds \right)^{\frac{1}{2}} \right\|_{\mathcal{H}^{q,p}([0,T]; \R)}
   \\ &\leq \left( \E\left[ \left( \int_0^T \left( \int_0^t \| E(t-s)\|_{L(H,V)}^2 \| \varphi(s)\|_{L_2(U,H)}^2 ds \right)^{\frac{q}{2}} dt \right)^{\frac{p}{q}} \right] \right)^{\frac{1}{p}}
     \\ &\leq \| E\|_{L^a([0,T]; L(H,V))} \left( \E\left[ \| \varphi \|_{L^b([0,T]; L_2(U,H))}^p \right] \right)^{\frac{1}{p}} < \infty.
 \end{align*}
 Similarly, when $q = \infty$, we obtain
 \begin{align*}
    &\ \left\| \left(\int_0^{\cdot} \| E(\cdot - s)\varphi(s)\|_{L_2(U,V)}^2 ds \right)^{\frac{1}{2}} \right\|_{\mathcal{H}^{\infty,p}([0,T]; \R)} 
   \\ &\leq \sup_{t \in [0,T]}\E\left[ \left( \int_0^t \|E(t-s)\|_{L(H,V)}^2 \|\varphi(s) \|_{L_2(U,H)}^2 ds \right)^{\frac{p}{2}} \right] 
   \\ &\leq \|E\|^p_{L^{a}([0,T];L(H,V))} \E\left[ \| \varphi\|_{L^{b}([0,T]; L_2(U;H))}^p \right].
 \end{align*}
 Hence condition \eqref{eq: varphi admissible} is satisfied for all cases of $2 \leq p$ and $q \in [p, \infty]$, and above estimates combined with Theorem \ref{thm: stochastic convolution cylindrical case} yield the assertion.
\end{proof}

\begin{Remark}
The previous two statements are also valid for the case of global bounds with respect to $t$ where $[0, T]$ is replaced by $\R_+$.
\end{Remark}

These results extend \cite{MR2401950, MR1954075} to a general class of operators $E(t)$. In contrast to classical stochastic convolutions with a $C_0$-semigroup, here we do not obtain any bounds on $L^{p}(\Omega, \F, \P; L^{\infty}([0, T]; V))$. Moreover, the factorization lemma does not apply in these cases, and hence we do not obtain continuity of sample paths from the integrability of $E$ around zero (compare with \cite[Section 3, Lemma 3.3]{MR2560625}). However, under additional conditions on the integrand, a stronger version of both results including H\"older continuous sample paths, was recently discussed in \cite[Lemma 3.1]{FHK22}. We close this section with an auxiliary result on the associativity of the stochastic convolution. 
\begin{Lemma}
 Suppose that $\varphi \in \mathcal{H}^{2,2}([0,T]; L_2(U,H))$, that $G \in L^2([0,T]; L_s(H,H_0))$, and $F \in L_{loc}^1(\R_+; L_s(H_0,H_1))$ where $H_0,H_1$ are separable Hilbert spaces. Then $F \ast (G \ast \varphi dW) = (F \ast G) \ast \varphi dW$ holds in $\mathcal{H}^{2,2}([0,T]; H_1)$ where all (stochastic) convolutions are well-defined.
\end{Lemma}
\begin{proof}
 It is clear that $F \ast G(t) = \int_0^t F(t-s)G(s)ds$ is an element of $L^2([0,T]; L_s(H,H_1))$. In view of Corollary \ref{thm: stochastic convolution}, it follows that $G \ast \varphi dW$ is well-defined in $\mathcal{H}^{2,2}([0,T]; H_0)$ and $(F\ast G) \ast \varphi dW$ is well-defined in $\mathcal{H}^{2,2}([0,T]; H_1)$. Hence $F \ast (G \ast \varphi dW)$ is well-defined in $\mathcal{H}^{2,2}([0,T];H_1)$. The relation $F \ast (G \ast \varphi dW) = (F \ast G) \ast \varphi dW$ follows from the stochastic Fubini theorem and a direct computation.
\end{proof}

\subsection{Equivalence of formulations}

Let us start with the definition of a strong solution of \eqref{VSPDE}.
\begin{Definition}
 Suppose that $F: H \longrightarrow H$ and $\sigma: H \longrightarrow L_2(U, H)$ are measurable, and that $g: \R_+ \times \Omega \longrightarrow H$ is jointly measurable with respect to $\mathcal{B}(\R_+) \otimes \F_0$. A strong solution of \eqref{VSPDE} is a measurable and $(\F_t)_{t \geq 0}$-adapted process $u$ that is $D(A)$-valued, satisfies
 \[
  \int_0^T \E\left[ \|F(u(t))\|_H + \| u(t) \|_{D(A)} + \|\sigma(u(t))\|_{L_2(U,H)}^2\right]ds < \infty, \qquad \forall T > 0,
 \]
 and \eqref{VSPDE} holds $\P \otimes dt$-a.e.
\end{Definition}

Under the above integrability condition on $u$, it is easy to see that all (stochastic) integrals in \eqref{VSPDE} are well-defined, see Corollary \ref{thm: stochastic convolution}.
By localization, one could weaken the integrability condition towards the requirement that $F(u) \in L_{loc}^1(\R_+;H)$, $u \in L_{loc}^1(\R_+;D(A))$, and $\sigma(u) \in L_{loc}^2(\R_+; L_2(U,H))$ holds $\P$-a.s. Since we will not use this, we leave such an extension for the interested reader. 

\begin{Definition}
 Let $F: H \longrightarrow H$ be measurable and $\sigma: H \longrightarrow L(U, H)$ be strongly operator measurable. Suppose that the $k,h$-resolvents $E_k,E_h$ exist, and let $g: \R_+ \times \Omega \longrightarrow H$ be jointly measurable with respect to $\mathcal{B}(\R_+) \otimes \F_0$ satisfying 
 \begin{align}\label{eq: minimal condition g}
  \P\left[ g \in L_{loc}^1(\R_+; H) \text{ and } E_k \ast g \in L_{loc}^1(\R_+; D(A))\right] = 1.
 \end{align}
 A mild solution of \eqref{VSPDE}, is a measurable and $(\F_t)_{t \geq 0}$-adapted process $u$ such that $E_h(t-s)\sigma(u(s)) \in L_2(U,H)$ holds a.s. for a.a. $0 \leq s < t$,  
 \[
  \int_0^T \E\left[\| F(u(t))\|_H + \int_0^t \| E_h(t-s)\sigma(u(s))\|_{L_2(U,H)}^2ds \right] dt < \infty, \qquad \forall T > 0,
 \]
 and for $Gg \in L_{loc}^1(\R_+; H)$ given by \eqref{eq: Gg definition} one has $\P \otimes dt$-a.e.
 \begin{align}\label{mild formulation}
     u(t) = Gg(t) + \int_0^t E_k(t-s)F(u(s))ds + \int_0^t E_h(t-s)\sigma(u(s))dW_s.
 \end{align}
\end{Definition}

Since $F(u) \in L_{loc}^1(\R_+; H)$ and $E_k$ is locally integrable, it follows that $E_k \ast F(u)(t) := \int_0^t E_k(t-s)F(u(s))ds$ belongs to $L_{loc}^1(\R_+; H)$. Moreover, by Theorem \ref{thm: stochastic convolution cylindrical case} the stochastic integral in \eqref{mild formulation} is well-defined. Finally, assumption \eqref{eq: minimal condition g} implies that the unique mild solution $Gg \in L_{loc}^1(\R_+;H)$ always exists and is given by \eqref{eq: solution formula Ek}. The next proposition elaborates on the relationship between strong and mild solutions.

\begin{Proposition}\label{thm: mild formulation characterization}
 Let $F: H \longrightarrow H$ and $\sigma: H \longrightarrow L_2(U,H)$ be measurable. Suppose that the $k,h$-resolvents $E_k, E_h$ exist, $E_h \in L_{loc}^2(\R_+; L_s(H))$, and let $g: \R_+ \times \Omega \longrightarrow H$ be measurable with respect to $\mathcal{B}(\R_+) \otimes \F_0$ such that $g \in L_{loc}^1(\R_+; H)$ and $E_k \ast g \in L_{loc}^1(\R_+; D(A))\ \P$-a.s. Then the following assertions hold:
 \begin{enumerate}
     \item[(a)] If $u$ is a strong solution then $u$ is a mild solution.
     \item[(b)] Conversely, let $u$ be a mild solution of \eqref{VSPDE} with values in $D(A)$ that satisfies 
 \begin{align}\label{eq: integrability equivalent formulations}
  \int_0^T \E\left[ \| \sigma(u(t))\|_{L_2(U,H)}^2 + \| u(t)\|_{D(A)} \right] \, dt < \infty, \qquad \forall T > 0.
 \end{align}
 Then $u$ is a strong solution of \eqref{VSPDE}.
 \end{enumerate}
\end{Proposition}
\begin{proof}
 Let $u$ be a strong solution and define
 \begin{align}\label{eq: f definition}
  f(t) = g(t) + \int_0^t k(t-s)F(u(s))ds + \int_0^t h(t-s)\sigma(u(s))dW_s.
 \end{align}
 Then $f \in L_{loc}^1(\R_+; H)$ a.s., and $u$ is a mild solution of $u = f + A(k \ast u)$. Since $E_k \ast g \in L_{loc}^1(\R_+; D(A))$, it follows by Theorem \ref{thm: integral resolvent solution} that $Gg \in L_{loc}^1(\R_+;H)$ exists and is given by $Gg = g + A E_k \ast g$. In particular, $k \ast Gg = k \ast g + A(k \ast E_k) \ast g = k \ast g + (E_k - k)\ast g = E_k \ast g$. Moreover, by Lemma \ref{prop: mild solution characterization} we obtain $E_k \ast f = k \ast u$. Hence, denoting by $h \ast \sigma(u)dW$, $E_h \ast \sigma(u)dW$ the stochastic convolutions, we find 
 \begin{align}
  k \ast u &= E_k \ast g + E_k \ast k \ast F(u) + E_k \ast (h \ast \sigma(u)dW) \notag
  \\ &= k \ast Gg + k \ast \left(E_k \ast F(u) + E_h \ast \sigma(u)dW \right) \label{eq:00}
 \end{align}
 where we have used the associativity of the convolution, and relation $E_k \ast h = E_{k \ast h} = E_h \ast k$ due to Lemma \ref{lemma: 3.3}. Moreover, note that $E_h \ast \sigma(u)dW$ is well-defined due to Theorem \ref{thm: stochastic convolution cylindrical case} since
 \begin{align*}
    &\ \int_0^T \E\left[ \int_0^t \| E_h(t-s)\sigma(u(s))\|_{L_2(U,H)}^2ds \right] dt
    \\ &\leq \int_0^T \int_0^t \|E_h(t-s)\|_{L(H)}^2 \E\left[ \|\sigma(u(s))\|_{L_2(U,H)}^2 \right] ds dt
     \\ &\leq \int_0^T \|E_h(t)\|_{L(H)}^2 dt \int_0^T \E\left[ \|\sigma(u(s))\|_{L_2(U,H)}^2 \right]dt < \infty.
 \end{align*}
 Titchmarch's Theorem applied to \eqref{eq:00} shows that $u$ is a mild solution of \eqref{VSPDE}.
 
 Conversely, let $u$ be a mild solution of \eqref{VSPDE} satisfying \eqref{eq: integrability equivalent formulations}. Then $u = Gg + E_k \ast F(u) + E_h \ast \sigma(u)dW$ and using the same argument as above we arrive at $k \ast u = E_k \ast f$ where $f$ is given as in \eqref{eq: f definition}. By Lemma \ref{prop: mild solution characterization} we conclude that $u$ satisfies $u = f + A(k\ast u)$. Since $u \in L_{loc}^1(\R_+; D(A))$, we obtain $A(k\ast u) = k \ast Au$. The particular form of $f$ shows that $u$ is a strong solution of \eqref{VSPDE}.
\end{proof}

\subsection{Existence and uniqueness of mild solutions}

The following is our main result on the existence and uniqueness of \eqref{eq: mild equation xi}. In particular, the existence and uniqueness result described in the Introduction, see Theorem \ref{thm: 1}, follows from the next theorem.
\begin{Theorem}\label{thm: existence and uniqueness resolvent cylindrical case}
 Suppose that conditions (A1) -- (A3) are satisfied. Fix $p \in [2,\infty)$ and $q \in [p,\infty]$. Then, for each $\F_0$-measurable $\xi \in \mathcal{H}_{loc}^{q,p}(\R_+; V)$, \eqref{eq: mild equation xi} has a unique solution $u(\cdot; \xi) \in \mathcal{H}_{loc}^{q,p}(\R_+; V)$. Moreover, for each $T > 0$ there exists a constant $C_T > 0$ such that, for all $\F_0$-measurable $\xi,\eta \in \mathcal{H}_{loc}^{q,p}(\R_+;V)$, 
 \[
  \| u(\cdot; \xi) - u(\cdot, \eta) \|_{\mathcal{H}^{q,p}([0,T]; V)} \leq C_T\| \xi - \eta \|_{\mathcal{H}^{q,p}([0,T];V)}.
 \]
\end{Theorem} 
\begin{proof}
 \textit{Step 1.} Fix $\lambda < 0$ and define $h_{\lambda}(t) = e^{\lambda t}h(t)$ and $k_{\lambda}(t) = e^{\lambda t}k(t)$. Then $E_{k_{\lambda}},E_{h_{\lambda}}$ exist and satisfy $E_{h_{\lambda}}(t) = e^{\lambda t}E_h(t)$ and $E_{k_{\lambda}}(t) = e^{\lambda t}E_{k}(t)$. Set $\xi_{\lambda}(t) = e^{\lambda t}\xi(t)$ and let $u(\cdot;\xi)$ be a solution of \eqref{eq: mild equation xi}. Then $v_{\lambda}(t;\xi) = e^{\lambda t}u(t;\xi)$ satisfies
 \begin{align}\label{eq: modified mild VSPDE}
  v_{\lambda}(t; \xi) = \xi_{\lambda}(t) + \int_0^t E_{k_{\lambda}}(t-s)F_{\lambda}(s,v_{\lambda}(s; \xi))ds + \int_0^t E_{h_{\lambda}}(t-s)\sigma_{\lambda}(s,v_{\lambda}(s;\xi))dW_s,
 \end{align}
 where $F_{\lambda}(s,x) = e^{\lambda s}F(e^{-\lambda s}x)$ and $\sigma_{\lambda}(s,x) = e^{\lambda s}\sigma(e^{-\lambda s}x)$. Conversely, let $v_{\lambda}$ be a solution of \eqref{eq: modified mild VSPDE}, then $u_{\lambda}(t; \xi) = e^{-\lambda t} v_{\lambda}(t;\xi)$ satisfies \eqref{eq: mild equation xi}. Therefore, it suffices to prove the existence and uniqueness for \eqref{eq: modified mild VSPDE}.

 \textit{Step 2.} We solve \eqref{eq: modified mild VSPDE} by a fixed point argument for the case $q = p$. Fix $T > 0$, then $\xi|_{[0,T]} \in \mathcal{H}^{q,p}([0,T]; V) \subset \mathcal{H}^{p,p}([0,T]; V)$. For given $\lambda < 0$, define 
 \[
  \mathcal{T}_{\lambda}(u;\xi)(t) := \xi_{\lambda}(t) + \int_0^t E_{k_{\lambda}}(t-s)F_{\lambda}(s,u(s))ds + \int_0^t E_{h_{\lambda}}(t-s)\sigma_{\lambda}(s,u(s))dW_s.
 \]
 Then \eqref{eq: modified mild VSPDE} is equivalent to $u_{\lambda} = \mathcal{T}_{\lambda}(u_{\lambda}; \xi)$. Below we show that $\mathcal{T}_{\lambda}(\cdot; \xi)$ is a contraction on $\mathcal{H}^{p,p}([0,T]; V)$ when $\lambda < 0$ is small enough. For brevity, we let $\| \cdot \|_p = \| \cdot \|_{\mathcal{H}^{p,p}([0,T]; V)}$. Take $u \in \mathcal{H}^{p,p}([0,T];V)$, then 
 \[
  \| \mathcal{T}_{\lambda}(u; \xi)\|_{p} \leq \| \xi\|_{p} + \left\| E_k \ast F_{\lambda}(\cdot, u) \right\|_p + \left\| E_{h_{\lambda}} \ast \sigma_{\lambda}(\cdot, u)dW \right\|_{p},
 \]
 and we estimate all terms separately. For the drift we use
 \begin{align}\label{eq: 2}
  \|F_{\lambda}(s,x)\|_{H_F} \leq C_{F,\mathrm{lin}}(e^{\lambda s} + \|i\|_{L(V,H)}\|x\|_{V}),
 \end{align}
 and Young's inequality to find
 \begin{align*}
     \left\| E_k \ast F_{\lambda}(\cdot, u) \right\|_p 
     &= \left\| \|E_{k_{\lambda}} \ast F_{\lambda}(\cdot, u)\|_{L^p([0,T];V)} \right\|_{L^p(\Omega)}
     \\ &\leq \|E_{k_{\lambda}}\|_{L^1([0,T];L(H_F,V))} \left\| \| F_{\lambda}(\cdot, u)\|_{L^p([0,T];H_F)}\right\|_{L^p(\Omega)} 
     \\ &\leq \max\{1,\|i\|_{L(V, H)}\} \|E_{k_{\lambda}}\|_{L^1([0,T];L(H_F,V))}C_{F,\mathrm{lin}}(T^{1/p} + \|u\|_p).
 \end{align*}
 For the stochastic convolution, we use Theorem \ref{thm: stochastic convolution cylindrical case},
 \begin{align}\label{eq: 6}
  \|E_{h_{\lambda}}(t-s)\sigma_{\lambda}(s,x)\|_{L_2(U,V)}
  \leq e^{\lambda(t-s)}K_{\mathrm{lin}}(t-s)\left( e^{\lambda s} + \|i\|_{L(V,H)} \|x\|_V\right)
 \end{align}
 and then Young's inequality to find that
 \begin{align*}
    &\ \left\| E_{h_{\lambda}} \ast \sigma_{\lambda}(\cdot, u)dW \right\|_{p} 
    \\ &\leq c_p^{1/p}\left( \int_0^T\E\left[ \left( \int_0^t \|E_{h_{\lambda}}(t-s)\sigma_{\lambda}(s,u(s))\|_{L_2(U,V)}^2 ds \right)^{\frac{p}{2}} dt \right] \right)^{\frac{1}{p}}
    \\ &\leq c_p^{1/p}\left( \E\left[\int_0^T \left( \int_0^t e^{2\lambda (t-s)}K_{\mathrm{lin}}(t-s)^2\left( e^{\lambda s}+\|u(s)\|_H\right)^2 ds\right)^{\frac{p}{2}} dt \right] \right)^{\frac{1}{p}}
    \\ &\leq c_p^{1/p}\max\{1,\|i\|_{L(V, H)}\}\left( \int_0^T e^{2\lambda s}K_{\mathrm{lin}}(s)^2 ds \right)^{\frac{1}{2}}\left( \E\left[\int_0^T \left(e^{\lambda s}+\|u(s)\|_V \right)^p ds \right] \right)^{\frac{1}{p}}
    \\ &\leq c_p^{1/p}\max\{1,\|i\|_{L(V, H)}\}\left( \int_0^T e^{2\lambda s}K_{\mathrm{lin}}(s)^2 ds \right)^{\frac{1}{2}}\left( T^{1/p} + \|u\|_p\right).
 \end{align*}
 Hence $\mathcal{T}_{\lambda}(\cdot; \xi)$ leaves the space $\mathcal{H}^{p,p}([0,T];V)$ invariant. In the same way, if $u,v \in \mathcal{H}^{p,p}([0,T]; V)$, we obtain
 \begin{align*}
     \| \mathcal{T}_{\lambda}(u; \xi) - \mathcal{T}_{\lambda}(v; \xi) \|_p &\leq \| E_{k_{\lambda}} \ast (F_{\lambda}(\cdot,u) - F_{\lambda}(\cdot,v)) \|_p
     \\ &\qquad \qquad + \left\| E_{h_{\lambda}} \ast (\sigma_{\lambda}(\cdot,u) - \sigma_{\lambda}(\cdot,v))dW \right\|_{p}.
 \end{align*}
 The drift is estimated by
 \begin{align*}
     \| E_{k_{\lambda}} \ast (F_{\lambda}(\cdot,u) - F_{\lambda}(\cdot,v)) \|_p
     &\leq \left\| \|E_{k_{\lambda}} \ast (F_{\lambda}(\cdot,u) - F_{\lambda}(\cdot,v))\|_{L^p([0,T];V)} \right\|_{L^p(\Omega)} 
     \\ &\leq \|E_{k_{\lambda}}\|_{L^1([0,T];L(H_F,V))} \left\| \| F_{\lambda}(\cdot,u) - F_{\lambda}(\cdot,v)\|_{L^p([0,T];H_F)}\right\|_{L^p(\Omega)} 
     \\ &\leq \|i\|_{L(V, H)}C_{F,\mathrm{lip}}\|E_{k_{\lambda}}\|_{L^1([0,T];L(H_F,V))}\| u-v\|_p,
 \end{align*}
 while the stochastic convolution is again estimated by Theorem \ref{thm: stochastic convolution cylindrical case} as follows
 \begin{align*}
     &\ \left\| E_{h_{\lambda}} \ast (\sigma_{\lambda}(\cdot,u) - \sigma_{\lambda}(\cdot,v))dW \right\|_{p} 
     \\ &\leq c_p^{1/p}\left( \E\left[ \int_0^T \left( \int_0^t \| e^{\lambda (t-s)}E_h(t-s))e^{\lambda s}(\sigma(e^{-\lambda s}u(s)) - \sigma(e^{-\lambda s}v(s))\|_{L_2(U,V)}^2 ds \right)^{\frac{p}{2}} dt \right] \right)^{\frac{1}{p}}
     \\ &\leq c_p^{1/p}\left( \E\left[ \int_0^T \left( \int_0^t e^{2\lambda (t-s)}K_{\mathrm{lip}}(t-s)^2 \| u(s) - v(s)\|_H^2 ds \right)^{\frac{p}{2}} dt \right] \right)^{\frac{1}{p}}
     \\ &\leq \|i\|_{L(V,H)}c_p^{1/p}\left( \int_0^T e^{2\lambda s}K_{\mathrm{lip}}(s)^2 ds \right)^{\frac{1}{2}}\|u - v\|_{p}.
 \end{align*}
 By dominated convergence combined with $E_{k_{\lambda}}(t) = e^{\lambda t}E_k(t)$, we see that 
 \[
   \|E_{k_{\lambda}}\|_{L^1([0,T];L(H_F,V))} + \ \int_0^T e^{2\lambda s}K_{\mathrm{lip}}(s)^2 ds \longrightarrow 0, \qquad \lambda \to - \infty.
 \]
 Hence we can choose $\lambda < 0$ sufficiently small such that $\mathcal{T}_{\lambda}(\cdot;\xi)$ is a contraction and thus has a unique fixed point $v_{\lambda}(\cdot;\xi)$ which is the unique solution of \eqref{eq: modified mild VSPDE}. Then $u(t;\xi) = e^{-\lambda t}v_{\lambda}(t;\xi)$ is the unique solution of \eqref{eq: mild equation xi} in $\mathcal{H}^{p,p}([0,T];V)$. 

 \textit{Step 3.} We prove that, if $\xi \in \mathcal{H}_{loc}^{q,p}(\R_+; V)$, then also $u(\cdot; \xi) \in \mathcal{H}_{loc}^{q,p}(\R_+; V)$. Fix $T > 0$, then $u(\cdot; \xi) \in \mathcal{H}^{p,p}([0,T]; V)$ and by the Fubini theorem, we find a representative in $L^p([0,T])$ for $t \longmapsto \E\left[ \|u(t;\xi)\|_{V}^p \right]$. Thus all inequalities below hold $dt$-a.e. on $[0,T]$. First note that
 \begin{align*}
     \E\left[ \|u(t;\xi)\|_{V}^p \right] &\leq 3^{p-1}\E\left[\|\xi(t)\|_{V}^p \right] 
     + 3^{p-1}\E\left[ \left\| \int_0^t E_k(t-s)F(u(s;\xi))ds \right\|_{V}^p \right]
     \\ &\qquad + 3^{p-1}\E\left[ \left\| \int_0^t E_h(t-s)\sigma(u(s;\xi))dW_s \right\|_{V}^p \right].
 \end{align*}
 The first convolution is estimated by \eqref{eq: 2} and Young's inequality as
 \begin{align*}
    &\  \E\left[ \left\| \int_0^t E_k(t-s)F(u(s;\xi))ds \right\|_{V}^p \right]
    \\ &\leq \E\left[ \left( \int_0^t \|E_k(t-s)\|_{L(H_F,V)}\|F(u(s;\xi))\|_{H_F}ds \right)^p \right]
     \\ &\leq \left( \int_0^t \|E_k(s)\|_{L(H_F,V)}ds \right)^{p-1}  \int_0^t \|E_k(t-s)\|_{L(H_F,V)}\E\left[\|F(u(s;\xi))\|_{H_F}^p \right] ds 
     \\ &\leq 2^{p-1} \|E_k\|_{L^1([0,T];L(H_F,V))}^{p} C_{F,\mathrm{lin}}^p 
     \\ &\qquad + \|i\|_{L(V,H)}^p\|E_k\|_{L^1([0,T];L(H_F,V))}^{p-1} 2^{p-1} C_{F,\mathrm{lin}}^p \int_0^t \|E_k(t-s)\|_{L(H_F,V)}\E\left[\|u(s;\xi)\|_{V}^p \right] ds
 \end{align*}
 where the integral on the right-hand side is finite as the convolution of two locally integrable functions. Similarly, using \eqref{eq: 6} we estimate the stochastic convolution as
 \begin{align*}
    &\ \E\left[ \left\| \int_0^t E_h(t-s)\sigma(u(s;\xi))dW_s \right\|_{V}^p \right]
     \\ &\leq c_p \E\left[ \left( \int_0^t \|E_h(t-s)\sigma(u(s;\xi))\|_{L_2(U,V)}^2 ds \right)^{\frac{p}{2}} \right]
     \\ &\leq c_p \E\left[ \left(\int_0^t K_{\mathrm{lin}}(t-s)^2 \left(1 + \|u(s;\xi)\|_H\right)^2 ds \right)^{\frac{p}{2}}\right]
     \\ &\leq c_p \left( \int_0^{t}K_{\mathrm{lin}}(s)^2 ds \right)^{\frac{p}{2}-1} \int_0^t K_{\mathrm{lin}}(t-s)^2 \E\left[\left(1 + \|u(s;\xi)\|_H\right)^p \right] ds
     \\ &\leq c_p 2^{p-1} \| K_{\mathrm{lin}}\|_{L^2([0,T])}^p + \|i\|_{L(V,H)}^p c_p 2^{p-1} \| K_{\mathrm{lin}}\|_{L^2([0,T])}^{p-2} \int_0^{t} K_{\mathrm{lin}}(t-s)^2 \E\left[\|u(s;\xi)\|_{V}^p \right] ds.
 \end{align*}
 Thus, there exists a constant $A_T \geq 0$ and some $\rho_T \in L^1([0,T])$ such that 
 \begin{align}\label{eq: pre gronwall local}
     \E\left[\|u(t;\xi)\|_{V}^p \right] \leq A_T + 3^{p-1}\E\left[ \|\xi(t)\|_{V}^p \right] + \int_0^t \rho_T(t-s)\E\left[ \|u(s;\xi)\|_{V}^p\right]ds.
 \end{align} 
 Let us denote by $r_T \in L_{loc}^1(\R_+)$ the unique solution of $r_T = \rho_T + \rho_T \ast r_T$. By Lemma \ref{lemma: positive solution Volterra} we have $r_T \geq 0$ and 
 \begin{align*}
  \E\left[\|u(t;\xi)\|_{V}^p \right] &\leq A_T + 3^{p-1}\E\left[ \|\xi(t)\|_{V}^p \right] + \int_0^t r_T(t-s)\left(A_T + 3^{p-1}\E\left[ \|\xi(s)\|_{V}^p \right] \right)ds.
 \end{align*} 
 Since $\E\left[ \|\xi(\cdot)\|_{V}^p \right] \in L^q([0, T])$ and $r_T \in L^1([0, T])$, we conclude that the upper bound on the right-hand side belongs to $L^q([0, T])$, and hence $u(\cdot;\xi) \in \mathcal{H}^{q,p}([0, T]; V)$. Since $T > 0$ is arbitrary, we conclude that $u(\cdot;\xi) \in \mathcal{H}_{loc}^{q,p}(\R_+; V)$.

 \textit{Step 4.} It remains to prove the stability estimate with respect to the initial condition. So let $\xi,\eta \in \mathcal{H}_{loc}^{q,p}(\R_+; V)$ be $\F_0$-measurable and let $u(\cdot;\xi), u(\cdot;\eta)$ be the corresponding solutions of \eqref{eq: mild equation xi}. Proceeding similarly to the proof of \eqref{eq: pre gronwall local}, we find 
 \[
  \E\left[\|u(t;\xi) - u(t;\eta)\|_{V}^p \right] \leq 3^{p-1}\E\left[ \|\xi(t) - \eta(t)\|_{V}^p \right] + \int_0^t \rho_T(t-s)\E\left[ \|u(s;\xi) - u(s;\eta)\|_{V}^p\right]ds.
 \]
 Thus applying first Lemma \ref{lemma: positive solution Volterra} to derive an a-priori bound on $\E\left[\|u(t;\xi) - u(t;\eta)\|_{V}^p \right]$, and then using Young's inequality, we readily deduce the desired stability estimate with $C_T = 3^{1-1/p}(1 + \|r_T\|^{1/p}_{L^1([0,T])})$. 
\end{proof}

\begin{Remark}
    An alternative proof could be obtained by considering the scaled norms
    \begin{align*}
     \|u\|_{p,\lambda} = \left( \int_0^T e^{\lambda p t}\E\left[ \|u(t)\|_{V}^p \right] dt\right)^{\frac{1}{p}}, \qquad \lambda < 0.
    \end{align*}
    To carry out the computations with respect to this choice of norms we would essentially use the same arguments. 
\end{Remark}

For equations with additive noise, the above statement can be slightly improved by also allowing $p \in [1,2)$. The precise result is formulated below and, in particular, implies the existence and uniqueness statement in Theorem \ref{thm: 2}.

\begin{Theorem}\label{thm: existence and uniqueness additive noise}
 Suppose that conditions (A1) -- (A3) are satisfied with $\sigma(x) = \sigma_0$ for all $x \in H$. Then all assertions of Theorem \ref{thm: existence and uniqueness resolvent cylindrical case} remain valid for the full range of parameters $1 \leq p < \infty$ and $p \leq q \leq \infty$. 
\end{Theorem}
\begin{proof}
 Given Theorem \ref{thm: existence and uniqueness resolvent cylindrical case}, it suffices to prove the assertion for $p \in [1,2)$. Step 1 is exactly the same as for the general case. Thus we proceed, to solve the fixed point equation $u_{\lambda}(\cdot; \xi) = \mathcal{T}_{\lambda}(u_{\lambda}; \xi)$ in $\mathcal{H}^{p,p}([0,T];V)$ with $T > 0$ fixed, where
 \[
  \mathcal{T}_{\lambda}(u;\xi)(t) := \xi_{\lambda}(t) + \int_0^t E_{k_{\lambda}}(t-s)F_{\lambda}(s,u(s))ds + \int_0^t E_{h_{\lambda}}(t-s)e^{-\lambda s}\sigma_0 dW_s
 \]
 since $\sigma_{\lambda}(s,x) = e^{\lambda s}\sigma_0$. For brevity, we let $\| \cdot \|_p = \| \cdot \|_{\mathcal{H}^{p,p}([0,T];V)}$. Let $u \in \mathcal{H}^{p,p}([0,T];V)$. Then $\| \mathcal{T}_{\lambda}(u; \xi)\|_{p} \leq \| \xi\|_{p} + \left\| E_k \ast F_{\lambda}(\cdot, u) \right\|_p + \left\| E_{h_{\lambda}} \ast e^{-\lambda \cdot} \sigma_0 dW \right\|_{p}$. For the drift we obtain by Young's inequality and \eqref{eq: 2}
 \begin{align*}
     \left\| E_k \ast F_{\lambda}(\cdot, u) \right\|_p 
     &\leq \max\{1,\|i\|_{L(V,H)}\}\|E_{k_{\lambda}}\|_{L^1([0,T];L(H_F,V))}C_{F,\mathrm{lin}}\left(T^{1/p}+\|u\|_p \right).
 \end{align*}
 For the stochastic convolution, we use Theorem \ref{thm: stochastic convolution cylindrical case} to find that  
 \begin{align*}
    \left\| E_{h_{\lambda}} \ast e^{\lambda \cdot}\sigma_{0}dW \right\|_{p} 
    &\leq c_p^{1/p}\left( \int_0^T \left(\E\left[ \left\| \int_0^t E_{h_{\lambda}}(t-s)e^{\lambda s}\sigma_0dW_s \right\|_{V}^2 \right] \right)^{\frac{p}{2}} dt \right)^{\frac{1}{p}}
    \\ &= c_p^{1/p}\left( \int_0^T e^{\lambda p t}\left( \int_0^t \| E_{h}(s)\sigma_0\|_{L_2(U,V)}^2 ds \right)^{\frac{p}{2}} dt \right)^{\frac{1}{p}}
    \\ &\leq c_p^{1/p}\left(\frac{1}{\lambda p}\right)^{\frac{1}{p}}\left(\int_0^T \| E_{h}(s)\sigma_0\|_{L_2(U, V)}^2 ds \right)^{\frac{1}{2}}.
 \end{align*}
 Hence $\mathcal{T}_{\lambda}(\cdot; \xi)$ leaves $\mathcal{H}^{p,p}([0,T];V)$ invariant. In the same way, if $u,v \in \mathcal{H}^{p,p}([0, T]; V)$, then 
 \begin{align*}
  \| \mathcal{T}_{\lambda}(u; \xi) - \mathcal{T}_{\lambda}(v; \xi) \|_p 
  &\leq \| E_{k_{\lambda}} \ast (F_{\lambda}(\cdot,u) - F_{\lambda}(\cdot,v)) \|_p 
  \\ &\leq \|i\|_{L(V,H)}C_{F,\mathrm{lip}}\|E_{k_{\lambda}}\|_{L^1([0, T]; L(H_F, V))}\| u-v\|_p
 \end{align*}
 where we have used the same estimate for the drift as in the proof of Theorem \ref{thm: existence and uniqueness resolvent cylindrical case}. Again, by dominated convergence combined with $E_{k_{\lambda}}(t) = e^{\lambda t}E_k(t)$, we can choose $\lambda < 0$ sufficiently small such that $\mathcal{T}_{\lambda}(\cdot; \xi)$ is a contraction and thus has a unique fixed point $v_{\lambda}(\cdot;\xi)$ which is the unique solution of \eqref{eq: modified mild VSPDE} with $\sigma_{\lambda}(s,x) = e^{\lambda s}\sigma_0$. Then $u(t;\xi) = e^{-\lambda t}v_{\lambda}(t;\xi)$ is the unique solution of \eqref{eq: mild equation xi} with $\sigma(x) = \sigma_0$ in $\mathcal{H}^{p,p}([0,T];V)$. 

 For the a-priori estimate on the solution $u$ we obtain, as before,
 \begin{align*}
     \E\left[ \|u(t;\xi)\|_{V}^p \right] &\leq A_T + 3^{p-1}\E\left[ \|\xi(t)\|_{V}^p \right] + \int_0^t \rho_T(t-s)\E\left[ \|u(s;\xi)\|_{V}^p\right]ds, 
 \end{align*}
 where all inequalities below hold $dt$-a.e. on $[0,T]$, $A_T \geq 0$ is some constant and $\rho_T \in L_{loc}^1([0,T])$. By Lemma \ref{lemma: positive solution Volterra} and the same arguments as in Theorem \ref{thm: existence and uniqueness resolvent cylindrical case} we deduce that $u(\cdot;\xi) \in \mathcal{H}^{q,p}([0, T]; V)$ whenever $\xi \in \mathcal{H}^{q,p}([0,T]; V)$. Finally, the desired stability estimate with respect to the initial condition, can be shown exactly in the same way as in the proof of Theorem \ref{thm: existence and uniqueness resolvent cylindrical case}. 
\end{proof}

\subsection{Sample path regularity}

Let $C([0,T]; V)$ and $C(\R_+;V)$ stand for the space of all continuous functions on $[0,T]$ or $\R_+$, respectively, with values in $V$. In particular, the next result proves Theorem \ref{thm: 1}.(a). 
\begin{Proposition}\label{cor: continuous modification}
 Suppose that conditions (A1) -- (A3) are satisfied, $\xi \in \mathcal{H}_{loc}^{\infty,p}(\R_+; V)$ for some $p \in [2,\infty)$, $E_k \in L_{loc}^{\frac{p}{p-1}}(\R_+; L_s(H_F,V))$ and \eqref{eq: continuity 1}, \eqref{eq: continuity 2} hold. Then $u - \xi \in C(\R_+; L^p(\Omega, \F,\P; V))$. Moreover, if there exists $\gamma \in (1/p,1]$ and for each $T > 0$ there exists a constant $C_{T,p} > 0$ such that, for all $s,t \in [0,T]$ with $s < t$, one has
 \begin{align*}
   &\ \| E_k(t-s + \cdot) - E_k(\cdot)\|_{L^{\frac{p}{p-1}}([0,T];L(H_F,V))} 
   \\ &\qquad \qquad + \|E_k\|_{L^{\frac{p}{p-1}}([0,t-s];L(H_F,V))}
   \leq C_{T,p}(t-s)^{\gamma}
 \end{align*}
 and
 \begin{align*}
     \int_0^s \widetilde{K}(t,s,r)^2 dr + \int_0^{t-s} |K_{\mathrm{lin}}(r)|^2 dr  \leq C_T (t-s)^{2\gamma}
 \end{align*}
 then $u - \xi$ has, for each $\alpha < \gamma - \frac{1}{p}$, a modification with locally $\alpha$-H\"older continuous sample paths.
\end{Proposition} 
\begin{proof}
 Since $u \in \mathcal{H}_{loc}^{p,p}(\R_+;V)$, we see that $u \in L^p_{loc}(\R_+; V)$ a.s., and hence $F(u) \in L_{loc}^p(\R_+; H_F)$ a.s.. Thus, $E_k \ast F(u)$ is a.s.~continuous by Lemma \ref{lemma: continuity convolution} and $E_k \ast F(u) \in C(\R_+; L^p(\Omega, \P; V))$ by dominated convergence. Secondly, note that
 \begin{align*}
  &\ \E\left[ \left( \int_0^s \|(E_h(t-r) - E_h(s-r))\sigma(u(r))\|_{L_2(U,V)}^2 dr \right)^{\frac{p}{2}}\right] 
  \\ &\qquad + \E\left[ \left( \int_s^t \|E_h(t-r)\sigma(u(r))\|_{L_2(U,V)}^2 dr \right)^{\frac{p}{2}} \right] 
  \\ &\leq \E\left[ \left( \int_0^s \widetilde{K}(t,s,r)^2(1+\|u(r)\|_H)^2 dr \right)^{\frac{p}{2}}\right] + \E\left[ \left( \int_s^t K_{\mathrm{lin}}(t-r)^2 (1+\|u(r)\|_H)^2 dr \right)^{\frac{p}{2}} \right] 
  \\ &\leq \left( \int_0^s \widetilde{K}(t,s,r)^2 dr \right)^{\frac{p}{2}-1}\int_0^s \widetilde{K}(t,s,r)^2\E\left[(1+\|u(r)\|_H)^p  \right] dr
  \\ &\qquad + \left( \int_s^t K_{\mathrm{lin}}(t-r)^2 dr \right)^{\frac{p}{2}-1} \int_s^t K_{\mathrm{lin}}(t-r)^2 \E\left[(1+\|u(r)\|_H)^p \right] dr
  \\ &\leq C_{T,p} \left(1 + \| u\|_{\mathcal{H}^{\infty,p}([0,T];V)}^p \right)\left( \left( \int_0^s \widetilde{K}(t,s,r)^2 dr \right)^{\frac{p}{2}} + \left( \int_s^t K_{\mathrm{lin}}(t-r)^2 dr \right)^{\frac{p}{2}} \right)
 \end{align*}
 where the right-hand side is finite due to $u \in \mathcal{H}^{\infty,p}([0,T];V)$ since $\xi \in \mathcal{H}^{\infty,p}([0,T];V)$. Hence Theorem \ref{thm: stochastic convolution cylindrical case} implies that $E_h \ast \sigma(u)dW \in C(\R_+; L^p(\Omega, \P; V))$. This proves the first assertion.

 The H\"older continuity follows from the Kolmogorov-Chentsov theorem, the assumptions on $K, \widetilde{K}$ and, for $0 \leq s < t \leq T$, the estimate 
 \begin{align*}
     &\ \left\| \int_0^t E_k(t-r)F(u(r))dr - \int_0^s E_k(s-r)F(u(r))dr \right\|_{V}
     \\ &\leq \int_0^s \| (E_k(t-r) - E_k(s-r))F(u(r))\|_{V} dr + \int_s^t \|E_k(t-r)F(u(r))\|_{V}dr
     \\ &\leq \| F(u)\|_{L^{p}([0,T];H_F)} 
      \| E_k(t-s + \cdot) - E_k(\cdot)\|_{L^{\frac{p}{p-1}}([0,T];L(H_F,V))}
      \\ &\qquad \qquad + \| F(u)\|_{L^{p}([0,T];H_F)}  \|E_k\|_{L^{\frac{p}{p-1}}([0,t-s];L(H_F,V))} 
     \\ &\leq \| F(u)\|_{L^{p}([0,T];H_F)}C_{T,p} (t-s)^{\gamma_F}.
 \end{align*} 
\end{proof}

If $\sigma$ is bounded in the sense that $\|E_h(t)\sigma(x)\|_{L_2(U,V)} \leq K_{\mathrm{lin}}(t)$ holds for each $x \in H$ and the right-hand side in \eqref{eq: continuity 2} is independent of $x$, then the assumption $\xi \in \mathcal{H}_{loc}^{\infty,p}(\R_+;V)$ can be weakened to $\xi \in \mathcal{H}_{loc}^{p,p}(\R_+;V)$. Finally, in the particular case of additive noise, the above statement can be further simplified as formulated below. In particular, the next statement proves Theorem \ref{thm: 2}.(a).

\begin{Proposition}\label{cor: continuous modification additive case}
 Suppose that conditions (A1) -- (A3) are satisfied and $\sigma(x) = \sigma_0$ for all $x \in H$. Let $\xi \in \mathcal{H}_{loc}^{p,p}(\R_+; H)$ be $\F_0$-measurable with $1 \leq p < \infty$ and suppose that $E_k \in L_{loc}^{\frac{p}{p-1}}(\R_+; L_s(H_F, V))$. Then $u - \xi \in C(\R_+; L^p(\Omega, \F, \P; V))$. 
\end{Proposition}
\begin{proof}
 The pathwise continuity of $E_k \ast F(u)$ follows again from Lemma \ref{lemma: continuity convolution}. To prove the continuity of $t \longmapsto E_h \ast \sigma_0 dW \in L^p(\Omega,\F, \P; V)$, we first note that $\lim_{t-s \to 0}\int_s^t \|E_h(t-r)\sigma_0\|_{L_2(U,V)}^2 dr = 0$ holds by dominated convergence. Moreover, since the translation operator is continuous in $L^p([0,T];V)$, we find
 \[
  \lim_{t-s \to 0}\int_0^T \| (E_h(t-s+r) - E_h(r))\sigma_0 e^U_n\|_{V}^2 dr = 0, \qquad n \geq 1,
 \]
 where $(e^U_n)_{n \geq 1}$ is an orthonormal basis of $U$. For $0 \leq t-s \leq 1$, we have
 \begin{align*}
  \int_0^T \| (E_h(t-s+r) - E_h(r))\sigma_0 e^U_n\|_{V}^2 dr &\leq 2 \int_{0}^{T + 1}\|E_h(r)\sigma_0 e^U_n \|_{V}^2 \, dr.
 \end{align*}
 Since the right-hand side is summable in $n$ uniformly in $t-s$, we conclude that 
 \begin{align*}
     &\ \lim_{t-s \to 0}\int_0^T \| (E_h(t-s+r) - E_h(r))\sigma_0\|_{L_2(U,V)}^2 dr 
     \\ &= \sum_{n=1}^{\infty} \lim_{t-s \to 0} \int_0^T \| (E_h(t-s+r) - E_h(r))\sigma_0 e^U_n\|_{V}^2 \, dr = 0.
 \end{align*}
 Thus the desired continuity is a consequence of Theorem \ref{thm: stochastic convolution cylindrical case}.
\end{proof}

If the drift $F$ is bounded, i.e. $\sup_{x \in H}\|F(x)\|_{H_F} < \infty$, then $F(u) \in L^{\infty}(\R_+; H_F)$ and we can replace $p/(p-1)$ in both statements by $1$.

\section{Limit distributions}
\label{sec:limit_dist}

\subsection{Global stability bounds}
\label{sec:glob_mom}

In this section, we extend the previous bounds from $[0,T]$ to $\R_+$ and show that $u \in \mathcal{H}^{q,p}(\R_+; V)$ whenever $\xi \in \mathcal{H}^{q,p}(\R_+; V)$. Recall that $C_{F,\mathrm{lin}}$ is the linear growth constant of $F$, $K_{\mathrm{lin}}$ is given in condition (A3), and $c_p$ is the constant from Theorem \ref{thm: stochastic convolution cylindrical case} as given in \eqref{eq: Cp constant}. Here and below we write $\xi_{\lambda}(t) = e^{\lambda t}\xi(t)$ and $u_{\lambda}(t) = e^{\lambda t}u(t)$ with $\lambda \in \R$. 

\begin{Theorem}\label{thm: uniform moment}
 Suppose that conditions (A1) -- (A3) are satisfied. Fix $p \in [2,\infty)$, $q \in [p, \infty]$ and let $\lambda \in \R$ such that one of the following conditions is satisfied:
 \begin{enumerate}
     \item[(i)] $\lambda \leq 0$, $e^{\lambda p t}\int_0^t \|E_k(s)\|_{L(H_F,V)}ds, \ \left(e^{2\lambda t}\int_0^{t} K_{\mathrm{lin}}(s)^2 ds\right)^{\frac{p}{2}} \in L^q(\R_+)$, and
    \begin{align}\label{eq: growth 1}
    &\ \|i\|_{L(V,H)}^p C_{F,\mathrm{lin}}^p 6^{p-1}\left( \int_0^{\infty} e^{\lambda t}\| E_{k}(t)\|_{L(H_F,V)} dt\right)^{p} 
    \\ \notag &\qquad \qquad + \|i\|_{L(V,H)}^p3^{p-1}2^{\frac{p}{2} -1} c_p \left( \int_0^{\infty} e^{2\lambda t} K_{\mathrm{lin}}(t)^2 dt \right)^{\frac{p}{2}} < 1
    \end{align}
    \item[(ii)] $\lambda \geq 0$, $F(0) = 0$, $\sigma(0) = 0$, and
    \begin{align}\label{eq: growth 2}
     &\ \|i\|_{L(V,H)}^p 3^{p-1}C_{F}^p\left( \int_0^{\infty}e^{\lambda t} \|E_{k}(t)\|_{L(H_F,V)}\right)^p  
     \\ \notag &\qquad \qquad + 3^{p-1}\|i\|_{L(V,H)}^p c_p \left(\int_0^{\infty} e^{2\lambda t}K_{\mathrm{lip}}(t)^2 dt \right)^{\frac{p}{2}} < 1.
     \end{align} 
 \end{enumerate}
    Then for each $\F_0$-measurable $\xi$ satisfying $\xi_{\lambda} \in \mathcal{H}^{q,p}(\R_+; H)$, it holds that
    \[
    \int_0^{\infty} e^{\lambda q t} \left( \E\left[ \|u(t;\xi)\|_{V}^p \right] \right)^{\frac{q}{p}} dt < \infty
    \]
    when $q < \infty$, while for $q = \infty$ we have 
    \[
    \E\left[ \|u(t;\xi)\|_{V}^p \right] \leq e^{-\lambda p t} \| u(\cdot;\xi)\|_{\mathcal{H}^{\infty,p}(\R_+; V)}^p < \infty.
    \]
\end{Theorem} 
\begin{proof}
Define $k_{\lambda}(t) = e^{\lambda t}k(t)$, $h_{\lambda}(t) = e^{\lambda t}h(t)$, $\sigma_{\lambda}(t,x) = e^{\lambda t}\sigma(e^{-\lambda t}x)$, and $F_{\lambda}(t,x) = e^{\lambda t}F(e^{-\lambda t}x)$. 
Then, using the same estimates as in step 3 of the proof of Theorem \ref{thm: existence and uniqueness resolvent cylindrical case}, we obtain for a.a.~$t$
 \begin{align*}
     \E\left[ \|u_{\lambda}(t)\|_{V}^p \right] &\leq A_{\lambda}(t) + 3^{p-1}\E\left[\|\xi_{\lambda}(t)\|_{V}^p \right] + \int_0^t \rho_{\lambda}(t-s)\E\left[\|u_{\lambda}(s)\|_{V}^p\right] ds,
 \end{align*}
 where $\rho_{\lambda} \in L^1(\R_+)$ is given by
 \begin{align*}
  \rho_{\lambda}(t) &= \|i\|_{L(V,H)}^p 6^{p-1}\|E_{k_{\lambda}}\|_{L^1(\R_+;L(H_F,V))}^{p-1} C_{F,\mathrm{lin}}^p \|E_{k_{\lambda}}(t)\|_{L(H_F,V)} 
  \\ &\qquad + \|i\|_{L(V,H)}^p 3^{p-1} 2^{\frac{p}{2}-1}c_p \| e^{\lambda \cdot}K_{\mathrm{lin}}\|_{L^2(\R_+)}^{p-2} e^{2\lambda t}K_{\mathrm{lin}}(t)^2,
 \end{align*}
 and $A_{\lambda} \in L^q(\R_+)$ is defined by
 \begin{align}\label{eq: Alambda}
  A_{\lambda}(t) &= 6^{p-1}\|E_{k_{\lambda}}\|_{L^1(\R_+;L(H_F,V))}^{p-1} C_{F,\mathrm{lin}}^p e^{\lambda p t}\int_0^t \|E_k(s)\|_{L(H_F,V)}ds 
  \\ &\qquad + 3^{p-1}2^{\frac{p}{2}-1}c_p \left( e^{2\lambda t}\int_0^{t} K_{\mathrm{lin}}(s)^2 ds \right)^{\frac{p}{2}}. \notag
 \end{align}
 Let $r_{\lambda}$ be the unique solution of $r_{\lambda} = \rho_{\lambda} + \rho_{\lambda} \ast r_{\lambda}$. Lemma \ref{lemma: integrability resolvent} combined with \eqref{eq: growth 1} implies that $r_{\lambda} \in L^1(\R_+)$. By Lemma \ref{lemma: positive solution Volterra} we obtain $r_{\lambda} \geq 0$ and
 \begin{align*}
  \E\left[\|u_{\lambda}(t;\xi)\|_{V}^p \right] &\leq A_{\lambda}(t) + 3^{p-1}\E\left[ \|\xi_{\lambda}(t)\|_{V}^p \right] 
  \\ &\qquad + \int_0^t r_{\lambda}(t-s) \left( A_{\lambda}(s) + 3^{p-1}\E\left[ \|\xi_{\lambda}(s)\|_{V}^p \right] \right)ds.
 \end{align*} 
 Since $A_{\lambda} \in L^q(\R_+)$ and $r_{\lambda} \in L^1(\R_+)$, it follows that $u_{\lambda} \in \mathcal{H}^{q,p}(\R_+; V)$ which proves the assertion under case (i).

 For case (ii) note that, since $F(0) = 0$ and $\sigma(0) = 0$, \eqref{eq: 2} and \eqref{eq: 6} may be improved to $\|F_{\lambda}(s,x)\|_{H_F} \leq \|i\|_{L(V,H)}C_{F}\|x\|_{V}$ and $\|E_{h_{\lambda}}(t-s,x)\sigma_{\lambda}(s,x)\|_{L_2(U,V)} \leq \|i\|_{L(V,H)}e^{\lambda(t-s)}K_{\mathrm{lip}}(t-s)\|x\|_{V}$ so that $A_{\lambda} \equiv 0$ in view of \eqref{eq: Alambda}. The improved condition \eqref{eq: growth 2} follows by repeating the bounds in the previous proof.
\end{proof}

Case (ii) with $p=2$ implies Theorem \ref{thm: 1}.(b). This theorem gives a bound on the (possibly exponential) growth rate for the solution in the $L^q(\R_+;L^p(\Omega, \F, \P;V))$ norm. In particular, if $q = \infty$ and $\lambda \leq 0$, we obtain ultimate boundedness of solutions as discussed in \cite[Chapter 7]{MR2560625} for classical SPDEs. For $\lambda = 0$ we see that condition \eqref{eq: growth 1} always implies $A_{\lambda} \in L^{\infty}(\R_+)$. Thus letting $q = \infty$, we obtain the uniform boundedness of the $p$-th moment of the process. Below we provide an analogous result for the case of additive noise.  

\begin{Theorem}
 Suppose that conditions (A1) -- (A3) are satisfied and $\sigma(x) = \sigma_0$ for all $x \in H$. Let $p \in [1,\infty)$, $q \in [p, \infty]$, and suppose that there exists $\lambda \leq 0$ such that 
 \begin{align}\label{eq: growth additive case}
   \|i\|_{L(V,H)}C_{F,\mathrm{lin}} \int_0^{\infty} e^{\lambda t}\| E_{k}(t)\|_{L(H_F,V)} dt < 1,
 \end{align}
 and assume that 
 \begin{align}\label{eq: A Lq bound additive noise}
 e^{\lambda t}\int_0^t \|E_k(s)\|_{L(H_F,V)}ds + e^{2 \lambda t}\int_0^{t} \| E_h(s)\sigma_0\|_{L_2(U,V)}^2 ds \in L^q(\R_+).
 \end{align}
 Then, for each $\F_0$-measurable $\xi$ satisfying $\xi_{\lambda} \in \mathcal{H}^{q,p}(\R_+; V)$, it holds that $u_{\lambda}(\cdot;\xi) \in \mathcal{H}^{q, p}(\R_+; V)$. 
\end{Theorem}
\begin{proof}
 Using first the Minkowski inequality and then Theorem \ref{thm: stochastic convolution cylindrical case}, we find that 
 \begin{align*}
     \|u_{\lambda}(t)\|_{L^p(\Omega,\P;V)}
     &\leq \left\| \xi_{\lambda}(t) \right\|_{L^p(\Omega,\P;V)}
     + c_p^{1/p}\left( \E\left[ \left(\int_0^t \|E_{h_{\lambda}}(t-s)e^{\lambda s}\sigma_0\|_{L_2(U,V)}^2 ds \right)^{\frac{p}{2}} \right] \right)^{\frac{1}{p}} 
     \\ &\qquad + \int_0^t \|E_{k_{\lambda}}(t-s)\|_{L(H_F,V)} \|F_{\lambda}(s,u_{\lambda}(s))\|_{L^p(\Omega,\P;H_F)} ds
     \\ &\leq A_{\lambda}(t) + \left\| \xi_{\lambda}(t) \right\|_{L^p(\Omega,\P;V)} + \int_0^t \rho_{\lambda}(t-s)\|u_{\lambda}(s)\|_{L^p(\Omega,\P;V)} ds
 \end{align*}
 where the drift has been estimated as before, $\rho_{\lambda}(t) = \|i\|_{L(V,H)}C_{F,\mathrm{lin}} \|E_{k_{\lambda}}(t)\|_{L(H_F,V)}$, and
 \[
  A_{\lambda}(t) =  c_p^{1/p}\left( e^{2\lambda t}\int_0^t \|E_{h}(s)\sigma_0\|_{L_2(U,V)}^2 ds \right)^{\frac{1}{2}} + e^{\lambda t}C_{F,\mathrm{lin}}\int_0^t \|E_{k}(s)\|_{L(H_F,V)} ds.
 \]
 Let $r_{\lambda}$ be the unique nonnegative solution of $r_{\lambda} = \rho_{\lambda} + \rho_{\lambda} \ast r_{\lambda}$. Lemma \ref{lemma: integrability resolvent} combined with \eqref{eq: growth additive case} implies that $r_{\lambda} \in L^1(\R_+)$. The assertion follows from Lemma \ref{lemma: positive solution Volterra} combined with $r \in L^1(\R_+)$ and $A_{\lambda} \in L^q(\R_+)$. 
\end{proof}

In regards to assumption \eqref{eq: A Lq bound additive noise}, if $\lambda = 0$, then we necessarily need to assume that $q = + \infty$. For $\lambda < 0$, we may consider the full range $q \in [p, \infty]$. Finally, for $\lambda > 0$ such a condition is never satisfied unless $\sigma_0 = 0$ which is the reason why we restrict the statement to $\lambda \leq 0$.

\subsection{Global contraction estimates}
\label{sec: contraction}

In this section, we prove auxiliary bounds for the difference $u(\cdot;\xi) - u(\cdot;\eta)$ globally on $\R_+$ instead of $[0, T]$. 
\begin{Proposition}\label{prop: contraction estimate}
 Suppose that conditions (A1) -- (A3) are satisfied. Fix $p \in [2,\infty)$ and $q \in [p, \infty]$. Then the following assertions hold:
 \begin{enumerate}
     \item[(a)] Assume that
     \begin{align}\label{eq: growth limit distribution}
    &\ \|i\|_{L(V,H)}^pC_{F,\mathrm{lip}}^p 6^{p-1}\left( \int_0^{\infty} \| E_{k}(t)\|_{L(H_F,V)} dt\right)^{p} 
    \\ \notag &\qquad \qquad + \|i\|_{L(V,H)}^p 3^{p-1}2^{\frac{p}{2} -1} c_p \left( \int_0^{\infty} K_{\mathrm{lip}}(t)^2 dt \right)^{\frac{p}{2}} < 1.
    \end{align}
    Then there exists $0 \leq r \in L^1(\R_+)$, such that  
    \begin{align}\label{eq: global stability estimate}
    \E\left[ \|u(t;\xi) - u(t;\eta)\|_{V}^p \right] &\leq 3^{p-1}\E\left[ \|\xi(t) - \eta(t)\|_{V}^p \right] 
    \\ \notag &\qquad + 3^{p-1}\int_0^t r(t-s) \E\left[ \|\xi(s) - \eta(s)\|_{V}^p \right] ds
    \end{align}
    holds for all $\F_0$-measurable $\xi,\eta \in \mathcal{H}_{loc}^{q,p}(\R_+; V)$. 
    \item[(b)] If $\sigma(x) = \sigma_0$ for all $x \in H$, and
    \begin{align*}
    \|i\|_{L(V,H)} C_{F,\mathrm{lip}} \int_0^{\infty} \| E_{k}(t)\|_{L(H_F,V)} dt < 1,
    \end{align*}
    then there exists $0 \leq r \in L^1(\R_+)$ such that 
    \begin{align*}
     \|u(t;\xi) - u(t;\eta)\|_{L^p(\Omega,\P;V)}
    &\leq \|\xi(t) - \eta(t)\|_{L^p(\Omega,\P;V)}
    \\ \notag &\qquad \qquad + \int_0^t r(t-s)\|\xi(s) - \eta(s)\|_{L^p(\Omega,\P;V)} ds
    \end{align*}
    holds for all $\F_0$-measurable $\xi,\eta \in \mathcal{H}_{loc}^{q,p}(\R_+; H)$. 
 \end{enumerate}
\end{Proposition} 
\begin{proof}
 (a) To prove the stability bound \eqref{eq: global stability estimate}, we proceed as in the proof of Theorem \ref{thm: existence and uniqueness resolvent cylindrical case}. Namely, we obtain $dt$-a.e. 
 \begin{align*}
  &\ \E\left[ \|u(t;\xi) - u(t;\eta)\|_{V}^p \right] 
  \\ &\leq 3^{p-1}\E\left[\|\xi(t) - \eta(t)\|_{V}^p \right] 
   \\ &\qquad + 3^{p-1}\E\left[ \left\| \int_0^t E_{k}(t-s)(F(u(s;\xi)) - F(u(s;\eta)))ds \right\|_{V}^p \right]
  \\ &\qquad + 3^{p-1}\E\left[ \left\| \int_0^t E_{h}(t-s)(\sigma(u(s;\xi)) - \sigma(u(s;\eta)))dW_s \right\|_{V}^p \right]
  \\ &\leq 3^{p-1}\E\left[\|\xi(t) - \eta(t)\|_{V}^p \right] + \int_0^t \rho(t-s)\E\left[ \|u(s;\xi) - u(s;\eta)\|_{V}^p\right]ds
 \end{align*}
 where $\rho \in L^1(\R_+)$ is given by
 \begin{align*}
  \rho(t) = 3^{p-1}C_{F,\mathrm{lip}}^p \| E_{k} \|_{L^1(\R_+; L(H_F,V))}^{p-1} \|E_{k}(t)\|_{L(H_F,V)} + 3^{p-1}2^{\frac{p}{2}-1}c_p \|K_{\mathrm{lip}}\|_{L^2(\R_+)}^{p-2} K_{\mathrm{lip}}(t)^2.
 \end{align*}
 Let $r \in L_{loc}^1(\R_+)$ be given by 
 \begin{align}\label{eq: r definition}
  r(t) = \rho(t) + \int_0^t \rho(t-s)r(s)ds.
 \end{align}
 Then, the assertion follows from Lemma \ref{lemma: positive solution Volterra} where $r \in L^1(\R_+)$ due to Lemma \ref{lemma: integrability resolvent}.

 (b) Proceeding as in the general case, we use the Minkowski inequality for the $L^p(\Omega)$ norm to find that 
 \begin{align*}
 &\ \|u(t;\xi) - u(t;\eta)\|_{L^p(\Omega,\P;V)}
   \\ &\leq \|\xi(t) - \eta(t)\|_{L^p(\Omega,\P;V)}
   + \int_0^t \|E_k (t-s)\|_{L(H_F,V)} \| F(u(s;\xi)) - F(u(s;\eta))\|_{L^p(\Omega,\P;H_F)} ds
   \\ &\leq \|\xi(t) - \eta(t)\|_{L^p(\Omega,\P;V)} + \int_0^t \rho(t-s) \|u(s;\xi) - u(s;\eta)\|_{L^p(\Omega,\P;V)} ds
 \end{align*}
 where $\rho \in L^1(\R_+)$ is given by $\rho(t) = \|i\|_{L(V,H)}C_{F,\mathrm{lip}} \|E_{k}(t)\|_{L(H_F,V)}$. The assertion follows again by a combination of Lemma \ref{lemma: positive solution Volterra} and Lemma \ref{lemma: integrability resolvent} applied to the unique solution of $r(t) = \rho(t) + \int_0^t \rho(t-s)r(s)ds$.  
\end{proof}

\subsection{Limit distributions}

Finally, we state and prove our main result on the existence and characterisation of limit distributions. In particular, the next theorem proves Theorem \ref{thm: 1}.(c). when applied to $p=2$.

\begin{Theorem}\label{thm: limit distribution}
 Suppose that conditions (A1) -- (A3) are satisfied, and that conditions \eqref{eq: continuity 1}, \eqref{eq: continuity 2}, \eqref{eq: growth 1} with $\lambda = 0$, and \eqref{eq: growth limit distribution} are satisfied. For $p \in [2,\infty)$ let $\xi \in \mathcal{H}_{loc}^{\infty,p}(\R_+; V) \cap C((0,\infty); L^p(\Omega, \F_0, \P; V))$ and suppose that \eqref{eq: xi limit} holds for $\xi(\infty) \in V$. If $E_k \in L_{loc}^{p/(p-1)}(\R_+; L_s(H_F,V))$, then there exists $\pi_{\xi} \in \mathcal{P}_p(V)$ such that 
 \begin{align}\label{eq: convergence rate}
  W_p(u(t;\xi), \pi_{\xi}) \leq C\left( D(t/2) + \left(\int_{t/2}^{\infty}r(s)ds\right)^{\frac{1}{p}} \right), \qquad t \geq 0,
 \end{align}
 where $C > 0$ is some constant, $r \in L^1(\R_+)$ is given as in \eqref{eq: r definition} and
 \begin{align*}
  \notag D(t) &= \| \xi(t) - \xi(\infty) \|_{L^p(\Omega,\P;V)} + \int_t^{\infty}\|E_k(s)\|_{L(H_F,H)}ds + \left( \int_t^{\infty} K_{\mathrm{lin}}(s)^2 ds \right)^{\frac{1}{2}}.
 \end{align*}
 Moreover, if $\eta \in \mathcal{H}_{loc}^{\infty,p}(\R_+; V) \cap C((0,\infty);L^p(\Omega, \F_0, \P; V))$ also satisfies \eqref{eq: xi limit} and $\xi(\infty) = \eta(\infty)$, then $\pi_{\xi} = \pi_{\eta}$.
\end{Theorem}
\begin{proof}
 \textit{Step 1.} Fix $\tau > 0$. We seek to construct a coupling $(u(t), \widetilde{u}_{\tau}(t))$ of $(u(t), u(t+\tau))$ for $t \geq 0$, where $\widetilde{u}_{\tau}$ is obtained from \eqref{eq: mild equation xi} with an adjusted forcing term $\widetilde{\xi}_{\tau}$ (obtained by restarting the process after time $\tau$). To find the corresponding coupling, we let, possibly on an extension of the original probability space, $(\overline{W}_t)_{t \geq 0}$ and $\overline{\xi}$ be copies of $(W_t)_{t \geq 0}$ and $\xi$ independent of $(\F_t)_{t \geq 0}$. Let $\overline{u} \in \mathcal{H}_{loc}^{\infty,p}(\R_+; V) \cap C( (0,\infty); L^p(\Omega, \F, \P; V))$ be the unique solution of 
 \[
  \overline{u}(t) = \overline{\xi}(t) + \int_0^t E_k(t-s)F(\overline{u}(s))ds + \int_0^t E_h(t-s)\sigma(\overline{u}(s))d\overline{W}_s
 \]
 whose existence is guaranteed by Theorem \ref{thm: existence and uniqueness resolvent cylindrical case} while $\overline{u} \in C((0,\infty); L^p(\Omega, \F,\P; V))$ follows from Theorem \ref{cor: continuous modification}. Define a new forcing term by
 \[
  \widetilde{\xi}_{\tau}(t) = \overline{\xi}(t+\tau) + \int_0^{\tau} E_k(t+\tau - s)F(\overline{u}(s))ds + \int_0^{\tau}E_h(t+\tau - s)\sigma(\overline{u}(s))d\overline{W}_s.
 \]
 Then $\widetilde{\xi}_{\tau} \in \mathcal{H}_{loc}^{\infty,p}(\R_+; V) \cap C(\R_+; L^p(\Omega, \F, \P; V))$. Indeed, let $T > 0$ be arbitrary, then 
 \begin{align*}
  &\ \| \widetilde{\xi}_{\tau}(t)\|_{L^p(\Omega,\P;V)}
  \\ &\leq \|\overline{\xi}(t+\tau)\|_{L^p(\Omega,\P;V)} + \int_0^{\tau} \| E_k(t+\tau-s)\|_{L(H_F,V)} \|F(\overline{u}(s))\|_{L^p(\Omega,\P;H_F)} ds 
  \\ &\qquad + c_p^{1/p} \left( \E\left[ \left( \int_0^{\tau} \|E_h(t+\tau-s)\sigma(\overline{u}(s))\|_{L_2(U,V)}^2 ds \right)^{\frac{p}{2}} \right] \right)^{\frac{1}{p}}
  \\ &\leq \| \overline{\xi}\|_{\mathcal{H}^{\infty,p}([0,T+\tau];V)} + \| F(\overline{u})\|_{\mathcal{H}^{\infty,p}([0,\tau]; H_F)} \int_0^{\tau}\|E_k(t+s)\|_{L(H_F,V)}ds
  \\ &\qquad + \max\{1,\|i\|_{L(V,H)}\}c_p^{1/p} \left( \E\left[ \left( \int_0^{\tau} K_{\mathrm{lin}}(t+\tau-s)^2 (1 + \|\overline{u}(s))\|_V)^2 \right)^{\frac{p}{2}} \right] \right)^{\frac{1}{p}}.
 \end{align*}
 For the last term, we obtain from the Jensen inequality
 \begin{align*}
     &\ \left( \E\left[ \left( \int_0^{\tau} K_{\mathrm{lin}}(t+\tau-s)^2 (1 + \|\overline{u}(s))\|_V)^2 \right)^{\frac{p}{2}} \right] \right)^{\frac{1}{p}}
     \\ &\leq \left( \int_0^{\tau}K_{\mathrm{lin}}(t+\tau-s)^2 ds \right)^{\frac{1}{2} - \frac{1}{p}} \left( \int_0^{\tau} K_{\mathrm{lin}}(t+\tau-s)^2 \E\left[(1 + \|\overline{u}(s))\|_V)^p \right] ds \right)^{\frac{1}{p}}
     \\ &\leq 2(1+\| \overline{u}\|_{\mathcal{H}^{\infty,p}([0,\tau]; V)}) \| K_{\mathrm{lin}} \|_{L^2([0,T+\tau])} < \infty.
 \end{align*}
 Since $T > 0$ is arbitrary, we conclude that $\widetilde{\xi}_{\tau} \in \mathcal{H}^{\infty,p}_{loc}(\R_+; V)$. By assumption $\xi$ and hence also $\overline{\xi}$ is continuous in $L^p$. For the drift, use the decomposition
 \begin{align*}
    \int_0^{\tau} E_k (t + \tau - s)F(\overline{u}(s))ds 
     &= \int_0^{t + \tau}E_k(t +\tau - s)F(\overline{u}(s))ds 
     \\ &\qquad - \int_{0}^{t} E_k(s)F(\overline{u}(\tau + t - s))ds.
 \end{align*}
 Both terms are, for fixed $\tau$, continuous in $L^p$ due to Lemma \ref{lemma: continuity convolution}. Finally, for the  stochastic integral, we use assumption \eqref{eq: continuity 2} and Jensen inequality to bound  
 \begin{align*}
    &\ \E\left[ \left\| \int_0^{\tau}E_h(t+\tau-s)\sigma(\overline{u}(s))d\overline{W}_s - \int_0^{\tau}E_h(t' + \tau - s)\sigma(\overline{u}(s))d\overline{W}_s\right\|_V^p\right]
    \\ &\leq c_p\E\left[ \left( \|(E_h(t+\tau-s) - E_h(t'+\tau -s))\sigma(\overline{u}(s))\|_{L_2(U,V)}^2 ds \right)^{\frac{p}{2}} \right]
    \\ &\leq \max\{1,\|i\|_{L(V,H)}^p\}c_p \E\left[ \left( \int_0^{\tau} \widetilde{K}(t+\tau,t'+\tau,s)^2(1+\|\overline{u}(s)\|_V)^2 ds \right)^{\frac{p}{2}} \right]
    \\ &\leq \max\{1,\|i\|_{L(V,H)}^p\}c_p\left( \int_0^{\tau} \widetilde{K}(t+\tau,t'+\tau,s)^2ds \right)^{\frac{p}{2}-1} 
    \\ &\qquad \qquad \cdot \int_0^{\tau} \widetilde{K}(t+\tau,t'+\tau,s)^2 \E[(1+\| \overline{u}(s)\|_V)^p] ds
    \\ &\leq \max\{1,\|i\|_{L(V,H)}^p\}c_p2^{p-1}\left(1 + \| \overline{u}\|_{\mathcal{H}^{\infty,p}([0,\tau]; V)}^p \right) \left( \int_0^{\tau} \widetilde{K}(t+\tau,t'+\tau,s)^2ds \right)^{\frac{p}{2}}.
 \end{align*}
 The right-hand side tends to zero due to \eqref{eq: continuity 1}, whence $\widetilde{\xi}_{\tau} \in C(\R_+; L^p(\Omega, \F, \P; V))$.
 To define the desired coupling $\widetilde{u}_{\tau}$, let us first define the extended filtration 
 \[
  \widetilde{\F}_t = \F_t \vee \sigma(\overline{W}_t \ : \ t \geq 0) \vee \sigma(\overline{\xi}_0).
 \]
 Since $(\overline{W}_t)_{t \geq 0}$ and $\overline{\xi}$ are independent of $(\F_t)_{t \geq 0}$, it is not difficult to see that $(W_t)_{t \geq 0}$ is also an $(\widetilde{\F}_t)_{t \geq 0}$ Wiener process, and $(\overline{W}_t)_{t \geq 0}$ and $\overline{\xi}$ are measurable with respect to $\widetilde{\F}_0$. Hence, the desired coupling can be defined as the unique solution of
 \begin{align*}
     \widetilde{u}_{\tau}(t) = \widetilde{\xi}_{\tau}(t) + \int_0^t E_k(t-s)F(\widetilde{u}_{\tau}(s))ds + \int_0^t E_h(t-s)\sigma(\widetilde{u}_{\tau}(s))dW_s
 \end{align*}
 on the filtered probability space $(\Omega, \F, (\widetilde{\F}_t)_{t \geq 0}, \P)$.
 By Theorem \ref{thm: existence and uniqueness resolvent cylindrical case} combined with Proposition \ref{cor: continuous modification}, we obtain $\widetilde{u}_{\tau} \in \mathcal{H}_{loc}^{\infty,p}(\R_+; V) \cap C((0,\infty); L^p(\Omega, \F,\P; V))$. To verify that $\widetilde{u}_{\tau}(t) \sim u(t+\tau)$ for $t \geq 0$, we note that  
 \begin{align*}
     u(t + \tau) &= \xi_{\tau}(t) + \int_{0}^{t}E_k(t-s)F(u(s+\tau))ds
     + \int_{0}^{t}E_h(t -s)\sigma(u(s+\tau))dW_s^{\tau}
 \end{align*}
 where $W^{\tau}_s = W_{s+\tau} - W_{\tau}$ denotes the restarted Wiener process with respect to the filtration $(\widetilde{\F}_{t + \tau})_{t \geq 0}$, and
 \[
  \xi_{\tau}(t) = \xi(t+\tau) + \int_0^{\tau}E_k(t+\tau-s)F(u(s))ds + \int_0^{\tau}E_h(t+\tau-s)\sigma(u(s))dW_s.
 \]
 Thus, by construction $\xi_{\tau}$ and $\widetilde{\xi}_{\tau}$ have the same law. By uniqueness also $\widetilde{u}_{\tau}$ and $u(\cdot + \tau)$ have the same law. Thus $(u(t), \widetilde{u}_{\tau}(t))$ is a coupling of $(u(t), u(t+\tau))$. This completes the proof of step 1.
 
 \textit{Step 2.} Next, using the coupling from step 1 and contraction estimate from Proposition \ref{prop: contraction estimate}.(a), we obtain
 \begin{align} \notag
  W_p(u(t;\xi), u(t+\tau;\xi))^p &\leq \E\left[ \| u(t) - \widetilde{u}_{\tau}(t)\|_V^p \right]
  \\ \label{eq:03} &\leq 3^{p-1}\E\left[ \| \xi(t) - \widetilde{\xi}_{\tau}(t)\|_V^p \right] 
  \\ \notag &\qquad + 3^{p-1}\int_0^t r(t-s)\E\left[ \| \xi(s) - \widetilde{\xi}_{\tau}(s)\|_V^p \right] ds
 \end{align}
 where $r$ is given as in \eqref{eq: r definition}. Let us prove that
 \[
  \lim_{t \to \infty}\sup_{\tau \geq 0} \E\left[ \| \xi(t) - \widetilde{\xi}_{\tau}(t) \|_V^p \right] = 0.
 \] 
 Indeed, it holds that
 \begin{align*}
     &\ \E\left[ \| \xi(t) - \widetilde{\xi}_{\tau}(t) \|_V^p \right]
     \\ &\leq 3^{p-1}\E\left[ \| \xi(t) - \overline{\xi}(t+\tau) \|^p \right] 
     \\ &\qquad + \max\{1,\|i\|_{L(V,H)}^p\}3^{p-1}C_{F,\mathrm{lin}}^p\E\left[ \left(\int_0^{\tau} \| E_k(t+\tau - s)\|_{L(H_F,V)} (1+\| \overline{u}(s)\|_{V})ds \right)^p \right]
     \\ &\qquad + \max\{1,\|i\|_{L(V,H)}^p\}3^{p-1}c_p\E\left[ \left( \int_0^{\tau}K_{\mathrm{lin}}(t+\tau-s)^2 (1+\|\overline{u}(s)\|_V)^2 ds \right)^{\frac{p}{2}} \right]
     \\ &\leq 6^{p-1}\E\left[ \| \xi(t) - \xi(\infty) \|_V^p \right] + 6^{p-1}\E\left[ \| \xi(\infty) - \overline{\xi}(t+\tau) \|_V^p \right] 
     \\ &\qquad + \max\{1,\|i\|_{L(V,H)}^p\}3^{p-1} C_{F,\mathrm{lin}}^p \left(\int_0^{\tau} \|E_k(t+s)\|_{L(H_F,V)}ds \right)^{p-1}
     \\ &\qquad \qquad \qquad  \cdot \int_0^{\tau} \|E_k(t+\tau-s)\|_{L(H_F,V)} \E\left[ (1+\|\overline{u}(s)\|_V)^p \right] ds
     \\ &\qquad + \max\{1,\|i\|_{L(V,H)}^p\}3^{p-1}c_p\left( \int_0^{\tau} K_{\mathrm{lin}}(t+s)^2 ds \right)^{\frac{p}{2} -1}
     \\ &\qquad \qquad \qquad \cdot \int_0^{\tau}K_{\mathrm{lin}}(t+\tau-s)^2 \E\left[(1+\|\overline{u}(s)\|_V)^p\right] ds 
     \\ &\leq 6^{p-1}\E\left[ \| \xi(t) - \xi(\infty) \|_V^p \right] + 6^{p-1}\E\left[ \| \xi(\infty) - \overline{\xi}(t+\tau) \|_V^p \right] 
     \\ &\qquad + \max\{1,\|i\|_{L(V,H)}^p\}6^{p-1} C_{F,\mathrm{lin}}^p \left(\int_0^{\tau} \|E_k(t+s)\|_{L(H_F,H)}ds \right)^{p} (1+\| \overline{u}\|_{\mathcal{H}^{\infty,p}(\R_+; V)}^p)
     \\ &\qquad + \max\{1,\|i\|_{L(V,H)}^p\}6^{p-1}c_p\left( \int_0^{\tau} K_{\mathrm{lin}}(t+s)^2 ds \right)^{\frac{p}{2}} (1+\| \overline{u}\|_{\mathcal{H}^{\infty,p}(\R_+; V)}^p)
 \end{align*} 
 and the right-hand side is finite due to $\| \overline{u}\|_{\mathcal{H}^{\infty,p}(\R_+; V)}^p < \infty$ by Theorem \ref{thm: uniform moment}.(i). The first two terms on the right-hand side converge to zero due to \eqref{eq: xi limit}. For the remaining terms we have, by condition \eqref{eq: growth 1}, $\int_0^{\tau} \|E_k(t+s)\|_{L(H_F,H)}ds + \int_0^{\tau} K_{\mathrm{lin}}(t+s)^2 ds \to 0$ as $t \to \infty$. Thus, in view of \eqref{eq:03}, \eqref{eq: xi limit}, and combined with $r \in L^1(\R_+)$ which follows from \eqref{eq: growth 1} and Lemma \ref{lemma: integrability resolvent}, we obtain $W_p(u(t+\tau), u(t)) \longrightarrow 0$ as $t \to \infty$ uniformly in $\tau$. In particular, the law of $(u(t))_{t \geq 0}$ is a Cauchy sequence in $\mathcal{P}_p(V)$ with respect to the $p$-Wasserstein distance. Hence it has a limit denoted by $\pi_{\xi} \in \mathcal{P}_p(V)$. 
 
 \textit{Step 3.} To bound the rate of convergence, note that the above estimates from Step 2 yield for some constant $C > 0$
 \begin{align*}
  \limsup_{\tau \to \infty} \E\left[ \| \xi_{\tau}(t) - \widetilde{\xi}_{\tau}(t)\|_V^p \right] 
  \leq C\left(1+\| u\|_{\mathcal{H}^{\infty,p}(\R_+; V)}^p\right) \widetilde{D}(t)
 \end{align*}
 where 
 \[
  \widetilde{D}(t) = \E\left[ \| \xi(t) - \xi(\infty) \|_V^p \right] + \left(\int_t^{\infty}\|E_k(s)\|_{L(H_F,V)}ds\right)^p + \left( \int_t^{\infty} K_{\mathrm{lin}}(t)^2 ds \right)^{p/2}. 
 \]
 Hence we obtain for some constant $C' > C$ 
 \begin{align*}
     W_p(u(t;\xi), \pi_{\xi}) &\leq \liminf_{\tau \to \infty} W_p(u(t; \xi), u(t+\tau; \xi))
     \\ &\leq 3^{1 - \frac{1}{p}}\liminf_{\tau \to \infty}\left(\E\left[ \| \xi_{\tau}(t) - \widetilde{\xi}_{\tau}(t)\|_V^p \right] \right)^{\frac{1}{p}}
     \\ &\qquad + 3^{1- \frac{1}{p}}\liminf_{\tau \to \infty}\left(\int_0^t r(t-s)\E\left[ \| \xi_{\tau}(s) - \widetilde{\xi}_{\tau}(s)\|_V^p \right] ds \right)^{\frac{1}{p}}   
     \\ &\leq C' \left(1+\| u\|_{\mathcal{H}^{\infty,p}(\R_+; V)}\right) \left( \widetilde{D}(t)^{1/p} 
     + \left(\int_0^t r(t-s)\widetilde{D}(s)ds \right)^{1/p} \right)
 \end{align*}
 where the second inequality follows from \eqref{eq:03}. We note that
 \begin{align*}
     \int_0^t r(t-s)\widetilde{D}(s)ds &= \int_0^{t/2} r(s) \widetilde{D}(t-s)ds + \int_{t/2}^t r(s)\widetilde{D}(t-s)ds
     \\ &\leq \widetilde{D}(t/2) \int_0^{\infty}r(s)ds + \widetilde{D}(0)\int_{t/2}^{\infty} r(s)ds.
 \end{align*}
 Since $\widetilde{D}(t)^{1/p} \leq D(t)$, we conclude with the desired estimate on the rate of convergence.
 
 \textit{Step 4.} Let $\pi_{\xi}, \pi_{\eta}$ be the limit distributions for $\xi,\eta$ satisfying $\xi(\infty) = \eta(\infty)$. Then we obtain $W_p(\pi_{\xi}, \pi_{\eta}) \leq W_p(\pi_{\xi}, u(t;\xi)) + W_p(u(t;\xi), u(t;\eta)) + W_p(u(t;\eta), \pi_{\eta})$. Taking the limit $t \to \infty$ we obtain
 \[
  W_p(\pi_{\xi}, \pi_{\eta}) \leq \limsup_{t \to \infty} W_p(u(t;\xi), u(t;\eta))
  \leq \limsup_{t \to \infty}\left(\E\left[ \| u(t;\xi) - u(t;\eta)\|_V^p\right] \right)^{\frac{1}{p}}.
 \]
 We estimate the right-hand side by
 \begin{align*}
  &\ \E\left[ \| u(t;\xi) - u(t;\eta)\|_V^p \right] 
  \\ &\leq 3^{p-1}\E\left[ \| \xi(t) - \eta(t) \|_V^p \right] + 3^{p-1}\int_0^t r(t-s) \E\left[ \| \xi(s) - \eta(s) \|_V^p \right] ds
  \\ &\leq 6^{p-1}\E\left[ \| \xi(t) - \xi(\infty)\|_V^p \right] + 6^{p-1}\int_0^t r(t-s)\E\left[ \| \xi(s) - \xi(\infty)\|_V^p \right]ds
  \\ &\qquad + 6^{p-1}\E\left[ \| \eta(\infty) - \eta(t)\|_V^p \right] + 6^{p-1}\int_0^t r(t-s)\E\left[ \| \eta(\infty) - \eta(s)\|_V^p \right]ds,
 \end{align*}
 where the first inequality was shown in Proposition \ref{prop: contraction estimate}.(a). The right-hand side converges to zero due to $r \in L^1(\R_+)$, which proves the assertion. 
\end{proof}

Let us remark that when $F,\sigma$ are bounded in the sense that $\sup_{x \in H}\|F(x)\|_{H_F} < \infty$, $\|E_h(t)\sigma(x)\|_{L_2(U,V)} \leq K_{\mathrm{lip}}(t)$ holds uniformly in $x \in H$, and \eqref{eq: continuity 2} holds, then we may weaken condition $\xi \in \mathcal{H}^{\infty,p}_{loc}(\R_+; V)$ to $\xi \in \mathcal{H}_{loc}^{p,p}(\R_+; V)$. For additive noise, the conditions can be further weakened as stated in the theorem below.

\begin{Theorem}\label{thm: limit distribution additive noise}
Suppose that conditions (A1) -- (A3) are satisfied, $\sigma(x) = \sigma_0$ for all $x \in H$, that $\|E_h(\cdot)\sigma_0\|_{L_2(U,V)} \in L^2(\R_+)$, $E_k \in L_{loc}^{p/(p-1)}(\R_+; L_s(H_F,V))$ for some $1 \leq p < \infty$ and \eqref{eq: limit distribution additive case}. Let $\xi \in \mathcal{H}_{loc}^{\infty,p}(\R_+; V) \cap C((0,\infty); L^p(\Omega, \F_0, \P; V))$, and suppose that \eqref{eq: xi limit} holds. Then there exists $\pi_{\xi} \in \mathcal{P}_p(V)$ such that \eqref{eq: convergence rate} is satisfied for some constant $C > 0$, $r \in L^1(\R_+)$ given as in \eqref{eq: r definition} with $\rho(t) = \|i\|_{L(V,H)}C_{F,\mathrm{lip}} \|E_{k}(t)\|_{L(H_F,V)}$, and 
 \[
  D(t) = \|\xi(t) - \xi(\infty)\|_{L^p(\Omega,\P;V)} + \int_t^{\infty}\|E_k(s)\|_{L(H_F,V)}ds.
 \]
 Moreover, if $\eta \in \mathcal{H}_{loc}^{\infty,p}(\R_+; V) \cap C((0,\infty);L^p(\Omega, \F_0, \P; V))$ also satisfies \eqref{eq: xi limit} and $\xi(\infty) = \eta(\infty)$, then $\pi_{\xi} = \pi_{\eta}$.
\end{Theorem}

The proof of Theorem \ref{thm: limit distribution additive noise} is exactly the same as of Theorem \ref{thm: limit distribution} but now we use Proposition \ref{prop: contraction estimate}.(ii) combined with Theorem \ref{thm: uniform moment}.(b) to bound the $p$-Wasserstein distance. Details are left for the reader. The proofs of both theorems show that, if $F$ is bounded, then we may drop assumption $E_k \in L_{loc}^{p/(p-1)}(\R_+; L_s(H_F,V))$.

\appendix

\section{One-dimensional Volterra equations}
\label{sec:1dvolterra}

\begin{Lemma}\label{lemma: positive solution Volterra}
 Let $k \in L_{loc}^1(\R_+)$ be nonnegative and $\mu \geq 0$. Then
 \begin{align}\label{eq: r}
  r(t) = k(t) + \mu \int_0^t k(t-s)r(s) \, ds
 \end{align}
 has a unique nonnegative solution in $L_{loc}^1(\R_+)$. 
 Moreover, if $x, f \in L_{loc}^1(\R_+)$ satisfy
 \[
  x(t) \leq f(t) + \mu \int_0^t k(t-s)x(s) \, ds, \qquad \text{ a.a. } t \geq 0
 \]
 then for a.a. $t \geq 0$
 \[
  x(t) \leq f(t) + \mu \int_0^t r(t-s)f(s) \, ds.
 \]
\end{Lemma}
\begin{proof}
 Fix $T > 0$. Following the proof of \cite[Chapter 2, Theorem 3.1]{MR1050319}, we let $\sigma > 0$ be large enough such that $k_{\sigma}(t) = e^{-\sigma t}k(t)$ satisfies $\| k_{\sigma}\|_{L^1([0,T])} < 1$. Then the unique solution $r_{\sigma}$ of $r_{\sigma} = k_{\sigma} + \mu (k_{\sigma}\ast r_{\sigma})$ can be obtained from the convergent series $r_{\sigma} = \sum_{j=1}^{\infty} \mu^{j-1}k_{\sigma}^{\ast j}$ in $L^1([0,T])$. In particular, $r_{\sigma} \geq 0$. Letting $r(t) = e^{\sigma t}r_{\sigma}(t)$, it is easy to verify that $r$ is the unique solution of \eqref{eq: r}. 
 
 For the second assertion define $\xi(t) = f(t) + \mu \int_0^t k(t-s)x(s)ds - x(t) \geq 0$. Then
 \begin{align*}
  x &= f - \xi + \mu (k \ast x)
  \\ &= f - \xi + \mu \left( r- \mu (k \ast r)\right) \ast x
  \\ &= f - \xi + \mu r \ast \left( x - \mu (k \ast x) \right)
  \\ &= f - \xi + \mu r \ast ( f-\xi) 
  \leq f + \mu (r \ast f).
 \end{align*}
 This proves the assertion.
\end{proof}

\begin{Lemma}\label{lemma: integrability resolvent}
 Let $\rho \in L^1(\R_+)$ be nonnegative, and let $r \in L_{loc}^1(\R_+)$ be the unique solution of  $r(t) = \rho(t) + \int_0^t \rho(t-s)r(s)ds$. Then $r \in L^1(\R_+)$ holds if and only if $\int_0^{\infty} \rho(t)dt < 1$.
\end{Lemma}
\begin{proof}
 Since $\rho \in L^1(\R_+)$, the Paley-Wiener Theorem (see \cite[Chapter 2, Theorem 4.1]{MR1050319}) applied to $-\rho$ asserts that $r \in L^1(\R_+)$ if and only if $\int_0^{\infty} e^{-tz}\rho(t)dt \neq 1$ for each $\mathrm{Re}(z) \geq 0$. By $\rho \geq 0$, this requirement is equivalent to $\int_0^{\infty}\rho(t)dt < 1$ which proves the assertion.      
\end{proof}

Let $\mu > 0$, $k,h \in L_{loc}^1(\R_+)$, and let $e_k(\cdot;\mu), e_h(\cdot;\mu)$ be the unique solutions of \eqref{eq: e rho} with $\rho = k,h$, i.e.
\begin{align*}
  e_k(t;\mu) + \mu \int_0^t k(t-s)e_{k}(s;\mu)ds &= k(t),
  \\ e_h(t;\mu) + \mu \int_0^t k(t-s)e_{h}(s;\mu)ds &= h(t).
\end{align*}

We are particularly interested in sufficient conditions for $e_k(\cdot;\mu) \in L^q(\R_+)$ with $q \in [1,\infty)$, $e_h(\cdot;\mu) \in L^2(\R_+)$, and bounds on their $L^q$-norms with respect to the parameter $\mu$.

\begin{Lemma}\label{lemma: general ek integral}
 Let $\mu > 0$, $k \in L_{loc}^1(\R_+)$ be nonnegative such that $\int_0^{\infty}e^{-\e t}k(t)dt < \infty$ for each $\e > 0$, and assume that $e_k(\cdot; \mu) \geq 0$ is nonincreasing. Then 
 \[
  \int_0^{\infty}e_k(t;\mu)dt = \begin{cases}\frac{1}{\mu}, & k \not \in L^1(\R_+)
  \\ \frac{\|k\|_{L^1(\R_+)}}{1 + \mu \|k\|_{L^1(\R_+)}}, & k \in L^1(\R_+)\end{cases}
 \]
 and, in particular, $\|e_k(\cdot;\mu)\|_{L^1(\R_+)} \leq \frac{1}{\mu}$. Moreover, $e_k(\cdot; \mu) \in L^q(\R_+)$ holds for each $q \in [1,\infty)$ that satisfies $k \in L_{loc}^q(\R_+)$.
\end{Lemma}
\begin{proof}
 The first assertion is a direct consequence of
  \begin{align*}
  \int_0^{\infty}e_k(t;\mu)dt 
  &= \lim_{\sigma \to 0^+}\int_0^{\infty}e^{-\sigma t}e_k(t;\mu)dt 
  \\ &= \lim_{\sigma \to 0^+} \frac{\widehat{k}(\sigma)}{1+\mu\widehat{k}(\sigma)}
  = \begin{cases}\frac{1}{\mu}, & k \not \in L^1(\R_+),
         \\ \frac{\|k\|_{L^1(\R_+)}}{1+\mu\|k\|_{L^1(\R_+)}}, & k \in L^1(\R_+)
        \end{cases}
 \end{align*}
 where we have used the Laplace transform $\widehat{k}$ of $k$. Since $e_k(\cdot,\mu) \geq 0$, it follows that $0 \leq e_k(\cdot;\mu) \leq k$ and hence for each $\e > 0$
 \begin{align*}
  \int_0^{\infty}|e_k(t;\mu)|^q dt 
  &\leq \int_0^{\e}|e_k(t;\mu)|^q dt + e_k(\e;\mu)^{q-1}\int_{\e}^{\infty}e_k(t;\mu)dt
  \\ &\leq \int_0^{\e} k(t)^q dt + k(\e)^{q-1} \int_0^{\infty}e_k(t;\mu)dt < \infty.
 \end{align*}
\end{proof}

The proof of $e_k(\cdot;\mu) \in L^q(\R_+)$ provides us with an additional method to obtain bounds on the $L^q(\R_+)$-norm of $e_h \in L_{loc}^1(\R_+)$ where $h \in L_{loc}^1(\R_+)$. We are particularly interested in the case where $q = 2$.

\begin{Lemma}\label{lemma: Lq norm eh}
    Let $k \geq 0$ satisfy $\int_0^{\infty}e^{-\e t}k(t)dt < \infty$ for each $\e > 0$, and suppose there exist $C_{\delta} > 0$ and $\delta \in (0,1)$ such that $k(t) \leq C_{\delta} t^{-\delta}$ for $t \in (0,1)$. Let $h \in L_{loc}^1(\R_+)$ be given by $h = k \ast \nu$ with $\nu$ a finite measure on $\R_+$. Suppose that $e_k(\cdot;\mu)$ is nonnegative and nonincreasing for some $\mu > 0$. Then
    \[
    \int_0^{\infty}|e_h(t;\mu)|^q dt \leq \nu(\R_+)^q \max\left\{\frac{C_{\delta}^q}{1-q\delta}, C_{\delta}^{q-1}\right\} (1 \vee \mu)^{- 1 + \frac{\delta}{1-\delta}}
    \]
    holds for each $q \in [1,1/\delta)$.
\end{Lemma}
\begin{proof}
 Since $h = k \ast \nu$, it is easy to see that $e_h(\cdot;\mu)$ is given by
 \[
  e_h(t;\mu) = h(t) - \mu\int_0^t h(t-s)e_k(s;\mu)ds
  = \int_{[0,t]}e_k(t-s;\mu)\nu(ds).
 \]
 Hence we obtain $\|e_h(\cdot;\mu)\|_{L^q(\R_+)} \leq \nu(\R_+) \|e_k(\cdot;\mu)\|_{L^q(\R_+)}$. Since $e_k(\cdot;\mu) \geq 0$ is nonincreasing, we obtain for $\e = 1 \wedge \mu^{-\kappa}$ with $\kappa = \frac{1}{1-\delta}$
 \begin{align*}
  \int_0^{\infty}|e_h(t;\mu)|^q dt 
  &\leq \nu(\R_+)^q \int_0^{\infty}|e_k(t;\mu)|^q dt
  \\ &\leq \nu(\R_+)^q \left(\int_0^{1 \wedge \mu^{-\kappa}} k(t)^q dt + k(1\wedge \mu^{-\kappa})^{q-1} \int_0^{\infty}e_k(t;\mu)dt \right)
  \\ &\leq \nu(\R_+)^q \frac{C^q}{1-q\delta} (1\vee\mu)^{-(1 - q\delta)\kappa} + \nu(\R_+)^qC^{q-1} (1\vee \mu)^{(q-1)\kappa \delta} \mu^{-1} 
  \\ &= \nu(\R_+)^q \max\left\{\frac{C^q}{1-q\delta}, C^{q-1}\right\} (1\vee \mu)^{- \frac{1 - q\delta}{1-\delta}}.
 \end{align*}
\end{proof}

\begin{Example}
 Let $k(t) = \log(1 + 1/t)$. Then $k$ is completely monotone and hence previous lemma is applicable with $\delta \in (0,1/2)$ and $C_{\delta} = \delta^{-1}$ which gives
 \[
  \int_0^{\infty}|e_h(t;\mu)|^2 \, dt \leq \frac{\nu(\R_+)^2}{\delta^2(1-2\delta)} \mu^{-\frac{1 - 2\delta}{1-\delta}}
  = \nu(\R_+)^2\frac{(1+\e)^3}{\e^2(1-\e)} \mu^{-(1 - \e)}
 \]
 where we used the substitution $\e = \frac{\delta}{1- \delta} \in (0,1)$ so that $\delta = \frac{\e}{1+\e}$.
 Notice that this implies that we can trade the order of decay in $\mu$ with the smallness of the constant $\nu(\R_+)^2\frac{(1+\e)^3}{\e^2(1-\e)}$ by varying $\e \in (0,1)$.
\end{Example} 

Finally, we provide sufficient conditions when the solutions $e_h(\cdot;\mu), e_k(\cdot;\mu)$ are increasing in $\mu$.

\begin{Lemma}\label{lemma: monotonicity}
 Let $k,h \in L_{loc}^1(\R_+)$. Suppose that $e_h(\cdot;\mu), e_k(\cdot; \mu) \geq 0$ holds for each $\mu > 0$. Then $0 \leq e_h(t; \mu) \leq e_h(t;\widetilde{\mu})$ for each $t > 0$ and all $0 < \mu < \widetilde{\mu}$. 
\end{Lemma}
\begin{proof}
 Let $0 < \mu < \widetilde{\mu}$. Then $h(t) = e_h(t;\mu) + \mu \int_0^t k(t-s)e_h(s;\mu)ds \leq e_h(t;\mu) + \widetilde{\mu}\int_0^t k(t-s)e_h(s;\mu)ds$. Define $\xi(t) := e_h(t;\mu) + \widetilde{\mu}\int_0^t k(t-s)e_h(s;\mu)ds - h(t) \geq 0$. Then 
 \[
  e_h(t; \mu) + \widetilde{\mu}\int_0^t k(t-s)e_h(s;\mu)ds = h(t) + \xi(t).
 \]
 Since $\mu e_k(\cdot; \mu) \geq 0$ is the resolvent of the second kind of $\mu k$, \cite[Theorem 2.3.5]{MR1050319} implies that 
 \begin{align*}
  e_h(t;\mu) &= h(t) + \xi(t) - \int_0^t \widetilde{\mu}e_h(t-s;\widetilde{\mu})(h(s) + \xi(s))ds
  \\ &= h(t) - \int_0^t \widetilde{\mu}e_h(t-s; \widetilde{\mu})h(s)ds + \left( \xi(t) - \int_0^t \widetilde{\mu}e_h(t-s; \widetilde{\mu})\xi(s)ds \right)
  \\ &= e_h(t; \widetilde{\mu}) + \left( \xi(t) - \int_0^t \widetilde{\mu}e_h(t-s; \widetilde{\mu})\xi(s)ds \right).
 \end{align*}
 Since $e_h(t;\widetilde{\mu}) \geq 0$, it suffices to show that $\left( \xi(t) - \int_0^t \widetilde{\mu}e_h(t-s; \widetilde{\mu})\xi(s)ds \right) \geq 0$. To prove this, first note that $\xi(t) = (\widetilde{\mu} - \mu) \int_0^t k(t-s)e_h(s;\mu)ds$. Thus we obtain
 \begin{align*}
  \xi(t) - \int_0^t \widetilde{\mu}e_h(t-s; \widetilde{\mu})\xi(s)ds 
  &= ( \widetilde{\mu} - \mu) \left( k - \widetilde{\mu} (e_h(\cdot;\widetilde{\mu}) \ast k \right) \ast e_h(\cdot; \mu)(t)
  \\ &= (\widetilde{\mu} - \mu) \int_0^t e_h(t-s; \widetilde{\mu})e_h(s;\mu) \, ds \geq 0.
 \end{align*}
\end{proof} 

Below we collect a few sufficient conditions on $k$ under which $e_k(\cdot;\mu) \geq 0$ is nonincreasing so that the above results can be applied.

\begin{Remark}\label{remark: sufficient condition for ek}
 If $k \in L_{loc}^1(\R_+)$ is nonincreasing, $k > 0$ on $(0,\infty)$, and $\ln(k)$ convex, then $e_k \geq 0$ due to \cite[Corollary 8.8, Chapter 9]{MR1050319}. If additionally $\ln(-k')$ is convex, then $e_k(\cdot;\mu) \geq 0$ is nonincreasing due to \cite{MR0521860}. These conditions are satisfied whenever $k \in L_{loc}^1(\R_+)$ is completely monotone and not constant, see \cite[Theorem 2.8 and Theorem 3.1, Chapter 5]{MR1050319}).
\end{Remark}

\section{Fractional one-dimensional equations}

In this section, we study the functions $e_k(\cdot;\mu)$ and $e_h(\cdot;\mu)$ given by \eqref{eq: e rho} with $\rho = k,h$ fractional kernels \eqref{eq: fractional kernels}, i.e. $k(t) = \frac{t^{\alpha-1}}{\Gamma(\alpha)}$ and $h(t) = \frac{t^{\beta-1}}{\Gamma(\beta)}$. It is easy to see that $e_k(t;\mu) = t^{\alpha - 1}E_{\alpha,\alpha}(-\mu t^{\alpha})$ and $e_h(t;\mu) = t^{\beta - 1}E_{\alpha,\beta}(-\mu t^{\beta})$ where $E_{\alpha,\beta}$ denotes the two-parameter Mittag-Leffler function $E_{\alpha,\beta}$, see \cite{MR4179587}. Recall that the constant $c_q(\alpha,\beta)$ was defined in \eqref{eq:c_q}. The next lemma follows from the well-known asymptotics of the two-parameter Mittag-Leffler function.
\begin{Lemma}\label{lemma: Mittag-Leffler function}
 Let $\alpha \in (0,2)$, $\beta > 0$, and $\mu > 0$. Then 
 \begin{enumerate}
     \item[(a)] $e_k(\cdot;\mu) \in L^q(\R_+)$ for $q \in [1,\infty)$ with $1 < \alpha + \frac{1}{q}$. Moreover, it holds that 
     \[
      \int_0^{\infty}|e_k(t;\mu)|^q dt = \mu^{- q + \frac{q-1}{\alpha}} c_q(\alpha,\alpha).
     \]
     \item[(b)] If $\alpha \neq \beta$, then $e_h(\cdot;\mu) \in L^q(\R_+)$ whenever $1 - \frac{1}{q} < \beta < \alpha + 1 - \frac{1}{q}$. Moreover
     \[
      \int_0^{\infty} |e_h(t;\mu)|^q dt = \mu^{-\frac{\beta q}{\alpha} + \frac{q-1}{\alpha}} c_q(\alpha,\beta)
     \]
 \end{enumerate}
\end{Lemma} 
\begin{proof}
 Since $E_{\alpha,\beta}$ is continuous on $\R_+$, we obtain $e_k(\cdot;\mu) \in L_{loc}^q(\R_+)$ for all $q \in [1,\infty)$ satisfying $q(\alpha - 1) > -1$, i.e. $1 < \alpha + \frac{1}{q}$. Likewise, $e_h(\cdot;\mu) \in L_{loc}^q(\R_+)$ for $\beta > 1 - \frac{1}{q}$. For the large values of $t$, we use the Poincar\'e asymptotics of the Mittag-Leffler function. When $\alpha = \beta$, we obtain $E_{\alpha,\alpha}(-z) \sim \frac{\sin(\pi \alpha)}{\pi \alpha} \frac{1}{\Gamma(\alpha)}z^{-2}$ as $z \to \infty$ and hence $e_k(t;\mu) \sim t^{- 1 - \alpha}$ as $t \to \infty$ which is always $q$-integrable. For $\alpha \neq \beta$, we obtain $E_{\alpha,\beta}(-z) \sim - \frac{1}{\Gamma(\beta - \alpha)}z^{-1}$ as $z \to \infty$ and hence $e_h(\cdot;\mu) \sim t^{\beta - 1 - \alpha}$ which is $q$-integrable whenever $\beta < \alpha + 1 - \frac{1}{q}$. The remaining assertions follow by substitution $r = \mu^{1/\alpha}t$.
\end{proof}

\section{Operator convolutions}

\begin{Lemma}\label{lemma: continuity convolution}
 Fix $T > 0$. Let $V,H$ be separable Hilbert spaces, $E \in L^p([0,T]; L_s(V,H))$, and $f \in L^q([0,T]; V)$ with $\frac{1}{p} + \frac{1}{q} = 1 + \frac{1}{r}$ and $p,q,r \in [1,\infty]$. Then 
 \[
  E \ast f(t) = \int_0^t E(t-s)f(s)ds, \qquad t \in [0,T]
 \]
 is $dt$-a.e.~well-defined and satisfies
 \[
  \| E \ast f\|_{L^r([0,T]; H)} \leq \| E \|_{L^p([0,T]; L(V,H))} \| f\|_{L^q([0,T];V)}.
 \]
 Suppose that $1 \leq p < \infty$ and $r = \infty$, then $E \ast f$ has a version that is continuous.
\end{Lemma}
\begin{proof}
    By approximation of $f$ through simple functions and using $E(\cdot) \in L(V,H)$, one can show that $(t,s) \longmapsto E(t-s)f(s) \in H$ is measurable. Hence the first assertion follows by applying first the triangle inequality and then the classical Young inequality, i.e., when $r \neq \infty$ we have
    \begin{align*}
        \| E \ast f\|_{L^r([0,T]; H)} &\leq \left( \int_0^T \left( \int_0^t \|E(t-s)\|_{L(V,H)} \| f(s)\|_V ds \right)^r dt \right)^{\frac{1}{r}}
        \\ &\leq \| E \|_{L^p([0,T]; L(V,H))} \| f\|_{L^q([0,T];V)}.
    \end{align*}
    Similarly we prove the case $r = \infty$, i.e.,
    \[
     \left\| \int_0^t E(t-s)f(s)ds \right\| \leq \int_0^t \|E(t-s)\|_{L(V,H)}\|f(s)\|_V ds \leq \| E \|_{L^p([0,T]; L(V,H))} \| f\|_{L^q([0,T];V)}.
    \]

    Let us prove the second assertion. First note that the space of continuous functions is dense in $L^q([0,T];V)$. In particular, we find $(f_n)_{n \geq 1} \subset C([0,T]; V)$ such that $f_n \longrightarrow f$ in $L^q([0,T];V)$. It is easy to see that $E \ast f_n \in C([0,T]; V)$. Moreover, it holds that
    \[
     \| E \ast f_n(t) - E \ast f_m(t) \| \leq \|E\|_{L^p([0,T]; L(V,H))} \| f_n - f_m\|_{L^q([0,T];V)} \longrightarrow 0
    \]
    as $n,m \to \infty$ uniformly in $t \in [0,T]$. Hence $(E\ast f_n)_{n\geq 1}$ is a Cauchy sequence in $C([0,T]; H)$ and hence has a limit $y \in C([0,T]; H)$. Then
    \begin{align*}
     \|y - E\ast f\|_{L^{\infty}([0,T];H)} \leq \|y - E\ast f_n\|_{L^{\infty}([0,T];H)} + \| E\|_{L^{p}([0,T];L(V,H))}\| f_n - f\|_{L^q([0,T];V)}
    \end{align*}
    which converges to zero as $n \to \infty$. Hence $y = E \ast f$ a.e., which proves the assertion.
\end{proof}

\subsection*{Acknowledgements}

The authors gratefully acknowledge the financial support received via the ECIU travel mobility fund for the year 2022. M.~F.~would like to express its gratitude for the financial support and kind hospitality at Trento University throughout 2021 and 2022 during which a large part of this research was carried out.

\bibliographystyle{amsplain}
\bibliography{Bibliography}

\end{document}